\documentclass[notitlepage,11pt,reqno]{amsart}
\usepackage[foot]{amsaddr}
\usepackage{amssymb,nicefrac,bm,upgreek,mathtools,verbatim}
\usepackage[final]{hyperref}
\usepackage[mathscr]{eucal}
\usepackage{dsfont}
\usepackage[normalem]{ulem}
\usepackage{amsopn,esint}
\usepackage[T1]{fontenc}

\usepackage[utf8]{inputenc}
\usepackage{apptools}
\AtAppendix{\counterwithin{lemma}{section}} 

\usepackage[margin=1in]{geometry}
\newcommand{\stkout}[1]{\ifmmode\text{\sout{\ensuremath{#1}}}\else\sout{#1}\fi}

\newtheorem{lemma}{Lemma}[section]
\newtheorem{theorem}{Theorem}[section]

\theoremstyle{definition}
\newtheorem{definition}{Definition}[section]

\newtheorem{remark}{Remark}[section]
\numberwithin{theorem}{section}
\numberwithin{equation}{section}

\hypersetup{
  colorlinks=true,
  citecolor=dgreen,
  linkcolor=mblue,
  urlcolor = blue,
  anchorcolor = blue,
  frenchlinks=false,
  pdfborder={0 0 0},
  naturalnames=false,
  hypertexnames=false,
  breaklinks}
\usepackage[capitalize,nameinlink]{cleveref}

\usepackage[abbrev,msc-links,nobysame,citation-order]{amsrefs}

\crefname{section}{Section}{Sections}
\crefname{subsection}{Section}{Sections}
\crefname{condition}{Condition}{Conditions}
\crefname{hypothesis}{Hypothesis}{Conditions}
\crefname{assumption}{Assumption}{Assumptions}
\crefname{lemma}{Lemma}{Lemmas}
\crefname{fact}{Fact}{Facts}

\Crefname{figure}{Figure}{Figures}

\crefformat{equation}{\textup{#2(#1)#3}}
\crefrangeformat{equation}{\textup{#3(#1)#4--#5(#2)#6}}
\crefmultiformat{equation}{\textup{#2(#1)#3}}{ and \textup{#2(#1)#3}}
{, \textup{#2(#1)#3}}{, and \textup{#2(#1)#3}}
\crefrangemultiformat{equation}{\textup{#3(#1)#4--#5(#2)#6}}%
{ and \textup{#3(#1)#4--#5(#2)#6}}{, \textup{#3(#1)#4--#5(#2)#6}}%
{, and \textup{#3(#1)#4--#5(#2)#6}}

\Crefformat{equation}{#2Equation~\textup{(#1)}#3}
\Crefrangeformat{equation}{Equations~\textup{#3(#1)#4--#5(#2)#6}}
\Crefmultiformat{equation}{Equations~\textup{#2(#1)#3}}{ and \textup{#2(#1)#3}}
{, \textup{#2(#1)#3}}{, and \textup{#2(#1)#3}}
\Crefrangemultiformat{equation}{Equations~\textup{#3(#1)#4--#5(#2)#6}}%
{ and \textup{#3(#1)#4--#5(#2)#6}}{, \textup{#3(#1)#4--#5(#2)#6}}%
{, and \textup{#3(#1)#4--#5(#2)#6}}

\crefdefaultlabelformat{#2\textup{#1}#3}
\newcommand{\vertiii}[1]{{\left\vert\kern-0.25ex\left\vert\kern-0.25ex\left\vert #1
    \right\vert\kern-0.25ex\right\vert\kern-0.25ex\right\vert}}

\newcommand{\cH}{{\mathcal{H}}}  

\newcommand{\cK}{{\mathcal{K}}}  
\newcommand{\cM}{{\mathcal{M}}}

\newcommand{\sD}{\mathscr{D}}

\newcommand{\RR}{\mathds{R}}
\newcommand{\NN}{\mathds{N}}

\newcommand{\Rn}{{\mathds{R}^{n}}}

\newcommand{\D}{\mathrm{d}}

\DeclareMathOperator*{\supp}{support}

\newcommand{\grad}{\nabla}

\usepackage{color}
\definecolor{dmagenta}{rgb}{.4,.1,.4}
\definecolor{dblue}{rgb}{.0,.0,.5}
\definecolor{mblue}{rgb}{.0,.4,.7}
\definecolor{ddblue}{rgb}{.0,.0,.4}
\definecolor{dred}{rgb}{.9,.0,.0}
\definecolor{dgreen}{rgb}{.0,.5,.0}
\definecolor{Eeom}{rgb}{.0,.0,.5}
\definecolor{dbrown}{rgb}{.6,.0,.0}

\allowdisplaybreaks

\newcommand{\ttl}{\Large Existence of global entropy solution for Eulerian droplet models and two-phase flow model with non-constant air velocity}
\begin{document}
\title[Eulerian droplet model and two-phase flow model]
{\ttl}

\author{Abhrojyoti Sen$^1$}
\address{$^1$Department of Mathematics, Indian Institute of Science Education and Research, Dr. Homi Bhabha Road, Pune 411008, India. Email: abhrojyoti.sen@acads.iiserpune.ac.in.}

\author{Anupam Sen$^2$}
\address{$^2$Centre for Applicable Mathematics, Tata Institute of Fundamental Research, Post Bag No 6503, Sharadanagar, Bangalore - 560065, India. Email: anupam21@tifrbng.res.in}



\begin{abstract}
This article addresses the question concerning the existence of global entropy solution for generalized Eulerian droplet models with air velocity depending on both space and time variables. When $f(u)=u,$ $\kappa(t)=const.$ and $u_a(x,t)=const.$ in \eqref{system}, the study of the Riemann problem has been carried out by  Keita and Bourgault \cite{sanajmaa} $\&$ Zhang \textit{et al.} \cite{zhangq}. We show the global existence of the entropy solution to \eqref{system} for any strictly increasing function $f(\cdot)$ and $u_a(x,t)$ depending only on time with mild regularity assumptions on the initial data via \textit{shadow wave tracking} approach. This represents a  significant improvement over the findings of Yang \cite{yangjde}. Next, by using the \textit{generalized variational principle,} we prove the existence of an explicit entropy solution to \eqref{system} with $f(u)=u,$ for all time $t>0$ and initial mass $v_0>0,$ where $u_a(x,t)$ depends on both space and time variables,  and also has an algebraic decay in the time variable. This improves the results of many authors such as Ha \textit{et al.} \cite{hajde}, Cheng and Yang \cite{chengaml} $\&$ Ding and Wang \cite{ding-huang} in various ways. Furthermore, by employing the shadow wave tracking procedure, we discuss the existence of global entropy solution to the generalized two-phase flow model with time-dependent air velocity that extends the recent results of Shen and Sun \cite{shenjde}.
\end{abstract}

\maketitle
\tableofcontents

\section{Introduction and main results}
\subsection{Overview} In this article, we consider the $1$D Eulerian droplet model

\begin{align}\label{system} 
\begin{aligned}
v_{t}+(vf(u))_{x}&=0,\,\,\, &x\in \RR, t>0,\\
(vu)_{t}+(vuf(u))_{x}&=\kappa(t)(u_{a}(x,t)-u)v, \,\,\, &x\in \RR, t>0,
\end{aligned}
\end{align}

adjoined with the initial data 
\begin{equation}\label{initial data}
(v,u)(x,0)=(v_0(x),u_0(x)), \,\,\, \,\,\, x\in \RR.
\end{equation}
The precise conditions on the initial data will be specified later depending on different situations. Here $v, u$ in \eqref{system} denotes the volume friction and the velocity of the droplets, respectively. Moreover, we take the following assumption on $f$ in \eqref{system}:\\

\,\,\,\,\,\,\,\,\,\,\,\,\,\,\,\,\,\,\,\,\,\,\,\,\,\,\,\,\,\,\,\,\,\,\,\,\,\,\,\,\,\,\,\,\,\,\,\,\,\,\,\,\,\,$f: \RR \to \RR $ is a $C^1$ strictly increasing function.\\

Furthermore, $u_a(x,t)$ is a locally bounded function that indicates the velocity of the carrier fluid or air, which depends on the position of the particle and time, and $\kappa(t) \in L^{\infty}([0, T); \RR)$  is the drag coefficients between the carrier fluid and the droplets. The above system \eqref{system} can be derived from a more general system
\begin{align}\label{system with garvity}
\begin{aligned}
v_{t}+(vf(u))_{x}&=0,\\
(vu)_{t}+(vuf(u))_{x}&=\kappa(t)(u_{a}(x,t)-u)v+\left(1-\frac{\rho}{\rho_l}\right)\frac{1}{Fr^2}\textbf{g},
\end{aligned}
\end{align}
by neglecting the source term involving the \textit{gravitational force} \textbf{g}. The system \eqref{system with garvity}, for $f(u)=u,$ $\kappa(t)=\frac{C_D R e_d}{24K}$ where $K=\frac{\rho_l d^2 U_{\infty}}{18L \mu}$ is an inertia parameter, $Fr=\frac{U_{\infty}}{\sqrt{Lg_0}}$ is the Froude number, $U_{\infty}$ is the speed of air at infinity; $g_0$ is a characteristic external field; and $L$ is a characteristic length, and $u_a(x,t)= \bf{u_a}$ was introduced by Bourgault \textit{et al.} \cite{YB99}. For a detailed physical description of \eqref{system with garvity}, applications, and numerical experiments, see \cite{YB99, Keitathesis, sanajmaa}.

In order to understand the system \eqref{system} from a more analytical point of view, we consider different cases depending on the function $f(\cdot),$ the air velocity $u_a(x,t)$ and the drag coefficient $\kappa(t)$ as follows:

\noindent\textbf{Case I:} $f(u)=u.$ This case can be split into two subcases depending on the contributions from the drag coefficient $\kappa(t).$

\noindent\textbf{Subcase I.} $\kappa(t)=0.$ In this case the system \eqref{system} turns out to be the usual system of pressureless gas dynamics and the initial value problem has been extensively studied in the last few decades. As it is well known now, among others, one of the main issues is that, $v$ is no longer a function, but a measure. So the natural space where one should search for a weak solution to \eqref{system} is the space of Radon measures. The existing results consist of different notions of weak solutions, for example, measure-valued solutions \cite{bouchut1}, duality solutions, and solutions via vanishing viscosity approach \cite{bouchut-james, BJ98, boudin}.  The global existence of weak solutions via mass and momentum potentials was established in \cite{BG, DZ96}. An explicit formula using generalized potentials and variational principles was obtained in \cite{rykov, haung-wang, Wang-haung}. A new perspective to the global existence of weak solutions for $1$D pressureless gas dynamics equation is due to Natile and Sava\'{r}e \cite{NS2009} by constructing sticky particle solutions using a suitable metric projection onto the cone of monotone maps. Later, Cavalletti \textit{et al.} \cite{cavalletti} gave a more direct proof by using the notion of differentiability of metric projections introduced by Haraux. On the other hand, Nguyen and Tudorascu \cite{nguyen, NT08} gave a general global existence result for \eqref{system} with or without viscosity by constructing an entropy solution for appropriate scalar conservation laws. They also obtained the uniqueness of the solution via the contraction principle in the Wasserstein metric. Other uniqueness results are due to Wang and Ding \cite{wang-ding} $\&$ Haung and Wang \cite{haung-wang} where they used generalized characteristics introduced by Dafermos \cite{dafermos1}. On the contrary,  Bressan and Nguyen \cite{BN2014} showed the non-uniqueness and non-existence of solutions in the multi-D case by constructing different initial data. Regarding numerical methods employed to the pressureless gas model, we refer to \cite{mathiaud, bouchut-jin}.

\noindent\textbf{Subcase II.}  Now we consider $\kappa(t)\neq 0.$ The work of Ha \textit{et al.} \cite{hajde} is the first useful result for us to consider in this scenario. They took $\kappa(t)=1$ and $u_a(x,t)=0,$ i.e., the system \eqref{system} takes the form 
\begin{align}\label{flocking system}
\begin{aligned}
v_{t}+(vu)_{x}&=0,\\
(vu)_{t}+(vu^2)_{x}&=-uv.
\end{aligned}
\end{align}
The system \eqref{flocking system} is strongly related to the pressureless Euler system with flocking dissipation
\begin{align}\label{full flock system}
\begin{aligned}
    \partial_t v+\grad_x \cdot (vu)&=0,\,\,\,\, x\in \Rn, \,\, t>0,\\
    \partial_t (uv)+\grad_x \cdot (vu \otimes u)&=-Kv \int_{\Rn} \psi(|x-y|)(u(x)-u(y))v(y)dy,
    \end{aligned}
\end{align}
where $K$ is the positive coupling strength and $\psi$ is a Lipschitz continuous function that denotes the communication weight. The system \eqref{flocking system} can be obtained from \eqref{full flock system} by setting the following quantities:
\begin{align*}
    n=1, K\equiv1, \psi\equiv 1, \int_{-\infty}^{\infty}v dx=1, \,\,\, \text{and}\,\,\, \int_{-\infty}^{\infty}v u dx=0.
\end{align*}
To study the initial value problem for \eqref{flocking system}, the authors used a variational approach. Furthermore, they showed the uniqueness of the entropy solution by adopting the arguments of \cite{wang-ding} in their setting. As it is mentioned earlier, in the case when $\kappa(t)$ and $u_a(x,t)$ both are constants, the Riemann problem for \eqref{system} is studied by Bourgault and Keita in \cite{sanajmaa}. More recently, Cheng and Yang \cite{chengaml} studied the Riemann problem for the system
\begin{align}\label{AML system}
\begin{aligned}
    v_t + (vu)_x&=0,\\
    (vu)_t+(vu^2)_x&=(kx-\alpha u)v.
    \end{aligned}
\end{align}
The nonhomogeneous term in \eqref{AML system} can easily be obtained by setting $\kappa(t)=\alpha>0$ and $u_a(x,t)= k x/\alpha$ in \eqref{system}. For other related studies, we refer to \cite{shen, ding-huang} and the references cited therein where the authors considered the pressureless Euler system with a coulomb-like friction term $\beta v,$ $\beta>0$ and a source term of the form $vx,$ respectively. Recently, Leslie and Tan \cite{leslie2023} developed a global wellposedness theory and long-time behavior for weak solutions of the 1D Euler-alignment system (similar to the system \eqref{full flock system}) with measure-valued density, and bounded velocity which is an example of a nonlocal system \cite{kiselev2018}.

\noindent\textbf{Case II:} $f: \RR \to \RR$ be any function satisfying $f^{\prime}(u)>0.$ Similar to the above, we consider two subcases below.\\
\noindent\textbf{Subcase I.} $\kappa(t)=0.$ This case corresponds to the generalized pressureless Euler system. To the best of our knowledge, Yang \cite{yangjde} first considered the Riemann problem for the homogeneous version of \eqref{system}, where $f(u)$ is assumed to be a smooth and strictly monotone function. He used the characteristics method to obtain the Riemann solution and showed the existence of a non-classical measure-valued solution. In fact, he proved there are only two kinds of solution: one that involves vacuum and the other one contains a delta measure in the component $v.$ Furthermore, he proposed a generalized Rankine-Hugoniot relation for delta shock solutions to the system \eqref{system}. 

The next result on the homogeneous version of system \eqref{system} is due to Huang \cite{huang}, where he established the existence of a global weak solution with the initial data $v_0(x)\geq 0, u_0(x) \in L^{\infty}(\RR).$ When $f(u)\neq u,$ the key difference between this system and the usual pressureless gas dynamics is that, here one has to deal with two kinds of speeds: one is the characteristics speed $f(u)$ and the other one is the physical velocity $u.$ Due to this, several tools, including duality solutions, sticky particle approaches, and most importantly the generalized variational principle (GVP) are inapplicable in this situation. Huang \cite{huang} established the global weak solution by developing an approach that can be seen as a combination of the front-tracking method and generalized characteristics. First, the initial data is approximated by piecewise constant functions and a sequence of Riemann solution $(v_n, u_n)$ is obtained up to a finite time when the first interactions of waves occur. Then a set of Riemann problems with $\delta$-initial data are solved to continue the process. 

In \cite{darko}, Mitrovi\'{c} and Nedeljkov showed that the Riemann solutions to the generalized pressureless gas dynamics equation (which is a non-strictly hyperbolic system) can be obtained as a vanishing pressure limit of the strictly hyperbolic system
\begin{align}\label{stricly hyperbolic}
\begin{aligned}
    v_t+(v f(u))_x&=0,\\
    (vu)_t+(vu f(u)+\varepsilon p(v))_x&=0,
    \end{aligned}
\end{align}
where $\varepsilon>0$ is a small parameter and the pressure term $p$ is a non-negative $C^2$-function satisfying: $p^{\prime}\geq0$ and $p^{\prime\prime}>0.$ The system \eqref{stricly hyperbolic} is strictly hyperbolic \cite{bressan} and can be solved for arbitrary Riemann data. The distributional limit as $\varepsilon \to 0$ of the $BV$ solutions to the system \eqref{stricly hyperbolic} converges to the delta shock solution of  the homogeneous version of \eqref{system}.
\begin{remark}
In this context, it is important to note that when $f^{\prime}(u)$ changes sign, system \eqref{system} can be associated with the general system of Keyfitz-Kranzer \cite{KK1980} or Aw-Rascle type \cite{AW2000}. The global existence results to such systems have been obtained by Lu (see \cite{Lu2013, LuSIMA} and the references cited therein) using compensated compactness arguments for the homogeneous case, i.e., when $\kappa(t)=0.$
\end{remark}
\noindent \textbf{Subcase II.} $\kappa(t)\neq 0.$ For $f(u)\neq u,$ there are very few papers in the literature that considers this case. Recently, Zhang \textit{et al.} \cite{zhangq} studied the Riemann problem for the system
\begin{align}\label{recent one}
\begin{aligned}
    v_t+(v f(u))_x&=0,\\
    (vu)_t+(vu f(u))_x&=(\beta-\alpha u)v,
    \end{aligned}
\end{align}
where the constants $\alpha$ and $\beta$ denote the dissipation coefficient and the friction coefficient, respectively. Note that \eqref{recent one} can be derived from \eqref{system} by simply setting $u_a(x,t)=\beta/\alpha$ and $\kappa(t)=\alpha.$ Also, for $\alpha=0,$ the Riemann solutions are obtained by Zhang and Zhang in \cite{zhangcpaa}.

\subsection{Main results} In this section, we state our main results. Keeping the above literature in mind, we ask the following question:

\textbf{Q.} \textit{Depending on the function $f$ and the source term, what are the possible cases for which the system \eqref{system}-\eqref{initial data} admits a global entropy solution?}

Our answer is two-fold:
\begin{itemize}
\item When $f$ is any $C^1$ strictly increasing function, we establish the existence of global entropy solution for the following system
\begin{align}\label{system dependent on t} 
\begin{aligned}
v_{t}+(vf(u))_{x}&=0,\\
(vu)_{t}+(vuf(u))_{x}&=\kappa(t)(u_{a}(t)-u)v.
\end{aligned}
\end{align}
To achieve our objective, we utilize shadow wave tracking method \cite{RN21}. Note that here the drag coefficient $\kappa$ and the air velocity $u_a$ are locally bounded functions of $t.$

\vspace{.3cm}
\item When $f(u)=u, $ using the generalized variational principle, we obtain an explicit representation of the entropy solution for the system
\begin{align}\label{source term depending on x and t} 
\begin{aligned}
v_{t}+(vu)_{x}&=0,\\
(vu)_{t}+(vu^2)_{x}&=\frac{1}{t+\upkappa}\left(\frac{x}{t+\upkappa}-u\right)v.
\end{aligned}
\end{align}
We derive the above system \eqref{source term depending on x and t} by setting  $\kappa(t)=\frac{1}{t+\upkappa}$ and $u_a(x,t)=\frac{x}{t+\upkappa}$ in \eqref{system} where $\upkappa \in \RR^+.$ Note that, here the air velocity $u_a(x,t)$ depends both on the time and space variables. Also, $\kappa(t)$ is a function that decays algebraically in time.
\end{itemize}

\vspace{.5cm}
The first part of the article is devoted to the results related to the system \eqref{system dependent on t}. We start with the concept of shadow waves \cite{marko1, marko2, marko3}. Shadow waves (in short SDW) are constructed as a net of piecewise constant (more precisely, piecewise constant for each time $t$) functions that approximate delta shocks in a small neighborhood of the shock location. Let a delta shock is supported by a curve $x=c(t)$ with speed $c^{\prime}(t).$ We perturb the curve from both sides by a small parameter $\varepsilon>0,$ and replace the  delta shock with a fan of shocks that depend on $\varepsilon.$ Next we give a formal definition of shadow wave solution.
\begin{definition}\label{first Defn}
    A shadow wave is a piecewise constant (for each time $t$) function of the form
    \begin{align}\label{intro-def-SDW}
        U^{\varepsilon}(x,t)= \begin{cases}
            (v_l(t), u_l(t)), \,\,\,\,\, x< c(t)-a_{\varepsilon}(t)-x_{l, \varepsilon},\\
            (v_{l, \varepsilon}(t), u_{l,\varepsilon}(t)),\,\,\,\,\, c(t)-a_{\varepsilon}(t)-x_{l, \varepsilon}<x<c(t),\\
            (v_{r, \varepsilon}(t), u_{r,\varepsilon}(t)),\,\,\,\,\, c(t)<x<c(t)+b_{\varepsilon}(t)+x_{r, \varepsilon},\\
            (v_r(t), u_r(t)), \,\,\,\,\, x>c(t)+b_{\varepsilon}(t)+x_{r, \varepsilon},
        \end{cases}
    \end{align}
    where $a_{\varepsilon}(t), b_{\varepsilon}(t), x_{l, \varepsilon}, x_{r, \varepsilon}$ are $\mathcal{O}(\varepsilon)$ for each $t>0.$ We say that the SDW \eqref{intro-def-SDW} solves the system \eqref{system dependent on t} if its substitution in the RHS and LHS of \eqref{system dependent on t} gives the same limit as $\varepsilon \to 0 $ in the sense of distributions.
\end{definition}
The idea of replacing delta (or singular) shocks with a fan of shocks reminds us of the method of front tracking for conservation laws (see \cite{bressan, bressan1, dafermos, dafermos1, holden, risebro, B1995} ). As a first step, the initial data \eqref{initial data} is approximated by piecewise constant functions and finitely many Riemann problems are solved at the initial level $t=0.$ The solution can be continued until $t=t_1,$ when the first interaction of waves occurs. Since the interaction of two waves produces only a single delta wave, the number of shock fronts decreases in time. At the time level $t=t_1,$ one needs to solve a finite number (less than the initial case) of Riemann problems with delta initial data and the process can be continued further.

In \cite{huang}, Huang started with a similar method but later on, he defined generalized characteristics and mass-momentum-energy potentials by using the approximate solution to produce a complete solution. We take a different route of using shadow wave solution at each stage of interaction (including the initial stage where no interaction happens) and obtain a complete solution in an approximated sense. One of the advantages of this approach is that it can be implemented to study $3 \times 3$ systems (see \cref{sec2}) of having unbounded solutions whereas it seems that Huang's method is restrictive in such cases. However, the solution constructed by Huang can be seen as an actual solution that satisfies the weak formulation.

\vspace{.5cm}
Next, we present the global existence result for the system \eqref{system dependent on t}.
\begin{theorem}\label{TH1.1}
     Let $v(x)\in L^{\infty}([R, \infty))$ be positive, $u(x) \in L^{\infty}([R, \infty))\cap C([R, \infty))$ and $u(x)$ be a  function having finitely many extremes. Take a partition $\{Y_i\}_{i\in \NN \cup \{0\}}$ of $[R, \infty)$ such that $Y_0=R$ and $ C_1 \varepsilon^{\alpha}<Y_i-Y_{i+1}<C_2 \rho(\varepsilon)$ for every $i=0,1,2,\cdots\cdots$ where $C_1, C_2\geq 1, \alpha\in (0, 1)$ and $\rho(\varepsilon) \to 0$ as $\varepsilon \to 0.$ Then there exists a global admissible solution to \eqref{system dependent on t} and \eqref{initial data approx}. More precisely, there exists a function $U^{\varepsilon}=(v^{\varepsilon}, u^{\varepsilon})$ that satisfies
     \begin{align*}
         \lim_{\varepsilon \to 0}\begin{cases}
             \left\langle \frac{\partial}{\partial t}v^{\varepsilon},  \varphi\right\rangle + \left\langle \frac{\partial}{\partial x}(v^{\varepsilon}f(u^{\varepsilon})),  \varphi\right\rangle=0,\\
             \left\langle \frac{\partial}{\partial t}(v^{\varepsilon} u^{\varepsilon}),  \varphi\right\rangle + \left\langle \frac{\partial}{\partial x}(v^{\varepsilon} u^{\varepsilon} f(u^{\varepsilon})),  \varphi\right\rangle=\left\langle \kappa(t)(u_{a}(t)-u^{\varepsilon})v^{\varepsilon}, \varphi \right\rangle,
         \end{cases}
     \end{align*}
     for every test function $\varphi \in C^{\infty}_c\left(\RR \times [0, \infty)\right)$ and the admissibility condition.
 \end{theorem} 

 \cref{TH1.1} can be extended to a $3\times 3$ system of pressureless drift-flux equations of two-phase flow model 
 \begin{align}\label{s7eq1} 
\begin{aligned}
v_{t}+(vf(u))_{x}&=0,\\
w_{t}+(wf(u))_{x}&=0,\\
((v+w)u)_{t}+((v+w)uf(u))_{x}&=\kappa(t)(u_{a}(t)-u)(v+w),
\end{aligned}
\end{align} in which $v$ and $w$ represent the masses of gas and liquid, respectively. For more on the drift-flux model, we refer to \cite{huangtwo, evje1, evje2}. Recently, considering $f(u)=u$ and the source term to be $-\mu (v+w),$  Shen and Sun \cite{shenjde} studied the Riemann problem for \eqref{s7eq1} and showed the existence of delta shock wave invoking the vanishing pressure limit approach. We can prove an analogous result of \cref{TH1.1} for \eqref{s7eq1} as our method only requires the existence of a unique solution to the Riemann problem and the interactions consisting of shadow waves or elementary waves.

 Next, we prove that a sequence of solutions constructed in \cref{TH1.1} has a weak limit in the space of Radon measures.
\begin{theorem}\label{TH1.2}
	Grant the assumptions of \cref{TH1.1} on the initial data \eqref{initial data approx}. Take a partition $\{Y^{\nu}_i\}_{i\in \NN \cup \{0\}}$ of $[R, \infty)$ such that $Y_0=R, C_1 \varepsilon_{\nu}^{\alpha}<Y_i-Y_{i+1}<C_2 \rho(\varepsilon_{\nu})$ for every $i=0,1,2,\cdot\cdots$ where $C_1, C_2\geq 1, \alpha\in (0, 1)$ and $\rho(\varepsilon_{\nu}) \to 0$ as $\varepsilon_{\nu} \to 0$ for any sequence $\{\varepsilon_{\nu}\}_{\nu\in \NN\cup \{0\} }.$ Let $\{U^\nu\}_{\nu \in \NN\cup \{0\}}$ be a sequence of approximated solution obtained in \cref{TH1.1}. Then there exists a subsequence still denoted as $\{U^{\nu}\}_{\nu\in \NN\cup \{0\}}$ and a Radon measure $U^*$ such that $U^{\nu} \stackrel{\ast}{\rightharpoonup} U^*$ as $\nu \to \infty.$
\end{theorem}
In the second part of the paper, we obtain the explicit formula for \eqref{source term depending on x and t} and show that it satisfies the weak formulation (see \cref{intro-defn-weak formulation}). We use the method of generalized variational principle (GVP). The next paragraphs are dedicated to briefly discussing the method and stating this part's main result.

As mentioned earlier, Rykov \textit{et al.} \cite{rykov} introduced the generalized variational principle for pressureless gas dynamics equation by generalizing the variational principle due to Lax and Oleinik for scalar conservation laws, in particular for Burger's equation. Huang and Wang \cite{haung-wang} $\&$  Ding \textit{et al.} \cite{Wang-haung } extended the method of generalized potentials when the initial data $u_0$ is not continuous and $v_0 \geq 0$ is a Radon measure, respectively. In this setting, the solution concept is the following: we show that $(v,u)$ is actually a weak solution to the system \eqref{source term depending on x and t}. First, we construct locally bounded measurable functions $m(x,t)$ and $u(x,t)$ such that $m(x,t)$ is of locally $BV$ in $x$ for a.e $t.$ Therefore $m$ defines a Lebesgue-Stieltjes measure $dm$ and its derivative in the sense of distribution defines a Radon measure $v=m_x.$ These two objects are the same through the identification
\begin{align*}
    -\left\langle m, \varphi_x \right\rangle=-\int_{-\infty}^{\infty}\varphi_x m \D{x}=\int_{-\infty}^{\infty}\varphi m_x \D{x}=\int_{-\infty}^{\infty}\varphi \D{m}=\left\langle v, \varphi \right\rangle \,\,\,\, \text{for all}\,\,\,\, \varphi\in C^{\infty}_c(\RR).
\end{align*}
Furthermore, similar identification allows us to define
\begin{align*}
    \left\langle vu, \varphi\right\rangle=\int_{-\infty}^{\infty}\varphi u \D{m}.
\end{align*}
These identifications lead to the notion of generalized solution to \eqref{source term depending on x and t}. The first equation can be written in the distributional sense as
\begin{align*}
    0=\left\langle v, \varphi_t\right\rangle+ \left\langle vu ,\varphi_x\right\rangle =-\iint \varphi_{xt} m \D{x}\D{t}+\iint \phi_x u \D{m}\D{t}.
\end{align*}
Similarly, the second equation of \eqref{source term depending on x and t} can be written as
\begin{align*}
  0&=\left\langle vu, \varphi_t\right\rangle+ \left\langle vu^2 ,\varphi_x\right\rangle+\left\langle\frac{1}{t+\upkappa}\left(\frac{x}{t+\upkappa}-u\right)v, \varphi\right\rangle \\
  &=\iint u \varphi_t \D{m}\D{t}+\iint u^2 \phi_x \D{m}\D{t}+\iint \frac{1}{t+\upkappa}\left(\frac{x}{t+\upkappa}-u\right) \varphi \D{m}\D{t}.   
\end{align*}
Therefore the weak formulation to the system \eqref{source term depending on x and t} is the following:
\begin{definition}\label{intro-defn-weak formulation}
    The pair $(v, u)$ is said to be a generalized solution to the system \eqref{second system} if the following integral identities
\begin{align}
    &\iint \varphi_t m \D{x}\D{t}-\iint \varphi u \D{m}\D{t}=0,\label{weak formulation}\\
    &\iint \varphi_t u + \varphi_x u^2 +\frac{1}{(t+\upkappa)}\left(\frac{x}{t+\upkappa}-u\right)\varphi \D{m}\D{t}=0,\label{wf2}
\end{align} hold for all test functions $\varphi \in C^{\infty}_c(\RR\times\RR^+),$ where the distributional derivative $m_x$ defines the Radon measure $v.$ 
\end{definition}
The construction of the generalized solution $(m, u)$ is done in two levels. First, by introducing generalized potential $F(y, x,t)$ we construct $u$ and then we introduce the momentum and energy potentials $q(x,t)$ and $E(x,t),$ respectively and some auxiliary functionals $H(\cdot, x,t), I(\cdot, x,t), J(\cdot, x,t).$ Moreover, by establishing relations between the measures $\D{q}, \D{E},$ $\D{m},$ and $\D{J}$ we show that $q, E, m$ satisfy \cref{intro-defn-weak formulation}.

Now we state the main result of this part.
\begin{theorem}\label{TH1.3}
Let $v_0(x)>0, u_0(x)$ are locally bounded measurable functions, then the pair $(m,u)$ given by \eqref{exf1}-\eqref{exf} is a global weak solution to the system \eqref{source term depending on x and t}-\eqref{initial data} in the sense of \cref{intro-defn-weak formulation}.
\end{theorem}
\begin{remark}
    We want to point out that the system \eqref{source term depending on x and t} can be associated with the system \eqref{full flock system} if we consider a more general communication weight $\psi(|x-y|,t)$ which is of the form $\kappa(t):=1/(t+\upkappa).$ Also, if we take further assumptions $\int_{-\infty}^{\infty}v(y)\D{y}=1$ and $\int_{-\infty}^{\infty}v(y)u(y)\D{y}=1,$ the nonhomogeneous term in \eqref{full flock system} would take a form $\frac{1}{t+\upkappa} \left(1-u\right)v.$ The source term considered in \eqref{source term depending on x and t} is even more general involving the space variable. Note that, in the situation described above the communication weight is a decaying function of time. Therefore, \cref{TH1.3} essentially gives an answer to the question of Ha \textit{et al.} \cite[Section 7]{hajde} where they made a query: \textit{whether the generalized variational principle would apply for non-constant communication weights, for instance, algebraically decaying communication weights.}
\end{remark}
\subsection{Plan of the paper} The article is organized into two separate parts. The first part of the article consists of \cref{sec2}, \cref{sec3}, \cref{sec4} and \cref{sec5}. In \cref{sec2}, we study the Riemann problem and the interactions for \eqref{system dependent on t} and \eqref{s7eq1}. In \cref{sec3}, using the entropy-entropy flux pair, we introduce the notion of dissipative shadow waves for the system \eqref{system dependent on t} and show its equivalence to the overcompressibility condition. \cref{sec4} is devoted to prove \cref{TH1.1} and \cref{TH1.2}. In \cref{sec5}, we provide some examples of physically relevant models that are included in \eqref{system dependent on t}. The second part of the paper consists of \cref{sec6} where we give the proof of \cref{TH1.3}.

\vspace{1cm}
\noindent\textbf{Part I: Global existence results for \eqref{system dependent on t} and \eqref{s7eq1}.}

\section{Riemann problem and interactions}\label{sec2} 
In this section, we study the shadow wave solution for the Riemann type initial data and initial data containing $\delta$-measure for the systems \eqref{system dependent on t} and \eqref{s7eq1}. We start with the system \eqref{system dependent on t}.\\

\noindent{\bf Notation.} Let $g$ be any function that depends ``only'' on time $t,$ i.e. $g: [0, \infty) \to \RR,$ then to denote the derivative of $g$ with respect to $t,$ we use $\frac{\partial}{\partial t}, \frac{\D}{\D{t}}, \cdot, \prime$ interchangeably throughout the article which convey the same meaning.
\subsection{Riemann problem to the system \cref{system dependent on t}}\label{subsection 2} First, we observe that for a smooth solution the system \eqref{system dependent on t}
reduces to $$u_t+f(u)u_x=\kappa(t)(u_{a}(t)-u).$$ From the characteristic equation we have the system of ODEs as follows:
\begin{align}\label{characteristic}
\begin{cases}
&\frac{\D{x(t)}}{\D{t}}=f(u(x(t),t)),\\
&\frac{\D{u(x(t),t)}}{\D{t}}=\kappa(t)(u_{a}(t)-u(x(t),t)),\\
& x(0)=x_0.
\end{cases}
\end{align}
Solving \eqref{characteristic}, we obtain
\begin{align*}
&u(x,t)=e^{-\int_0^t\kappa(\theta)\D{\theta}}\left(\int_0^t F(\theta)\D{\theta}+u_0(x_0)\right),\,\,\, \text{and}\,\,\,\\
&x(t)=x_0+\int_0^t f\left(e^{-\int_0^s \kappa(\theta)\D{\theta}}\left(\int_0^s F(\theta)\D{\theta}+u_0(x_0)\right)\right)\D{s},
\end{align*}
where $F$ is given by $F(t)=\kappa(t)\cdot u_{a}(t)\cdot e^{\int_0^t\kappa(\theta)\D{\theta}}$. This  motivates us to consider the  shadow wave in the following form:
\begin{align}\label{sdw}
U^{\varepsilon}=(v^{\varepsilon}, u^{\varepsilon})(x,t)=
\begin{cases}
\left(v_l, U_l(t)\right),\,\,\,\,\,\,\,\,\,\,\,\, &x< c(t)-\frac{\varepsilon}{2}t-x_{\varepsilon},\\
\left(v_{\varepsilon}(t), u_{\varepsilon}(t)\right),\,\,\,\,\,\, &c(t)-\frac{\varepsilon}{2}t-x_{\varepsilon}<x<c(t)+\frac{\varepsilon}{2}t+x_{\varepsilon},\\
\left(v_r, U_r(t)\right),\,\,\,\,\,\,\,\,\,\,\, &x>c(t)+\frac{\varepsilon}{2}t+x_{\varepsilon},
\end{cases}
\end{align}
where $U_{l,r}(t):=e^{-\int_0^t\kappa(\theta)\D{\theta}}\cdot\left(\int_0^t F(\theta)\D{\theta}+u_{l,r}\right)$ and  $x_{\varepsilon}, v_{\varepsilon}(t)$ are $\mathcal{O}(\varepsilon)$ and $\mathcal{O}(1/\varepsilon),$  and $\lim\limits_{\varepsilon \to 0}u_{\varepsilon}(t)=\chi(t).$ First, we study the above system \eqref{system dependent on t} when initial data contains a $\delta$-measure and is of the following form 
\begin{align}\label{delta initial data}
(v, u)(x,0)=
\begin{cases}
(v_l, u_l),\,\,\,\,\,\, &x<0,\\
(\bar{m}\delta(x), \bar{u}),\,\,\,\,\,\, &x=0,\\
(v_r, u_r),\,\,\,\,\,\, &x>0,
\end{cases}
\end{align}
where $v_{l,r}\geq0$, $\bar{m}>0$ and $u_l>\bar{u}>u_r$. This situation arises when two approaching shock waves interact. Suppose we are given a piecewise constant data: $(v_l, u_l), (v_m, u_m)$ and $(v_r, u_r)$ with $u_l>u_m>u_r.$ The delta shock curve joining $(v_l, u_l)$ to $(v_m, u_m)$ interacts with another delta shock curve connecting the states $(v_m, u_m)$ to $(v_r, u_r)$ at some point $(X, T)$ and at this level, we need to solve a Riemann problem with a $\delta$-initial data. When $\bar{m}=0,$ then the data is purely of Riemann type and can be seen as a particular case of \eqref{delta initial data}. Without loss of any generality, we may assume $(X,T)=(x,0).$

Substituting the shadow wave solution \eqref{sdw} into the system \eqref{system dependent on t}, from the definition of shadow wave we have
\begin{align}
& \lim_{\varepsilon \to 0} \left[\left\langle \frac{\partial}{\partial t}v^{\varepsilon},\varphi \right\rangle+\left\langle \frac{\partial}{\partial x}\left(v^{\varepsilon} f(u^{\varepsilon})\right),\varphi \right\rangle \right]=0,\label{distribution1}\\
&\lim_{\varepsilon \to 0} \left[\left\langle \frac{\partial}{\partial t}\left(v^{\varepsilon}u^{\varepsilon}\right),\varphi \right\rangle+\left\langle \frac{\partial}{\partial x}\left(v^{\varepsilon} u^{\varepsilon} f(u^{\varepsilon})\right),\varphi \right\rangle -\left\langle \kappa(t)(u_a(t)-u^{\varepsilon})v^{\varepsilon}, \varphi \right\rangle \right]=0,\label{distribution2}
\end{align}
for all $\varphi \in C^{\infty}_c(\RR\times [0, \infty))$.\\
Now for a fixed $\varepsilon >0,$ by using integration by parts in the term involving time derivative of \eqref{distribution1}, we get
\begin{align*}
-\left\langle \frac{\partial}{\partial t}v^{\varepsilon},\varphi \right\rangle&=\int_{0}^{\infty}\int_{-\infty}^{\infty}v^{\varepsilon}\frac{\partial}{\partial t}\varphi(x,t)\D{x}\D{t}+\int_{-\infty}^{\infty}v^{\varepsilon}(x,0)\varphi(x,0)\D{x}\\
&=\int_{0}^{\infty}\int_{-\infty}^{c(t)-\frac{\varepsilon}{2}t-x_{\varepsilon}} v_l\frac{\partial}{\partial t}\varphi(x,t)\D{x}\D{t}+ \int_{-\infty}^{-x_{\varepsilon}}v^{\varepsilon}(x,0)\varphi(x,0)\D{x}\\
&+\int_{0}^{\infty}\int_{c(t)-\frac{\varepsilon}{2}t-x_{\varepsilon}}^{c(t)+\frac{\varepsilon}{2}t+x_{\varepsilon}} v_{\varepsilon}(t)\frac{\partial}{\partial t}\varphi(x,t)\D{x}\D{t}+\int_{-x_{\varepsilon}}^{x_{\varepsilon}}v^{\varepsilon}(x,0)\varphi(x,0)\D{x}\\
&+\int_{0}^{\infty}\int_{c(t)+\frac{\varepsilon}{2}t+x_{\varepsilon}}^{\infty} v_r \frac{\partial}{\partial t}\varphi(x,t)\D{x}\D{t}+\int_{x_{\varepsilon}}^{\infty}v^{\varepsilon}(x,0)\varphi(x,0)\D{x}.
\end{align*}

\vspace{.5cm}
Simplifying the above expression, we obtain
\begin{align}\label{EQ2.6}
-\left\langle \frac{\partial}{\partial t}v^{\varepsilon},\varphi \right\rangle&=\int_{0}^{\infty}\left(v_{\varepsilon}(t)-v_l\right)\varphi\left(c(t)-\frac{\varepsilon}{2}t-x_{\varepsilon},t\right)\left(c^{\cdot}(t)-\frac{\varepsilon}{2}\right)\D{t}\nonumber\\
&+\int_{0}^{\infty}\left(v_r-v_{\varepsilon}(t)\right)\varphi\left(c(t)+\frac{\varepsilon}{2}t+x_{\varepsilon},t\right)\left(c^{\cdot}(t)+\frac{\varepsilon}{2}\right)\D{t}\nonumber\\ 
&-\int_0^{\infty}\int_{c(t)-\frac{\varepsilon}{2}t-x_{\varepsilon}}^{c(t)+\frac{\varepsilon}{2}t+x_{\varepsilon}}\frac{\partial}{\partial t}v_{\varepsilon}(t)\varphi(x,t)\D{t}+\int_{-x_{\varepsilon}}^{x_{\varepsilon}} v_{\varepsilon}(0)\varphi(x,0)\D{x}.
\end{align}

\vspace{.5cm}
Similarly, the term involving spatial derivatives of \eqref{distribution1} gives
\begin{align}\label{EQ2.7}
-\left\langle \frac{\partial}{\partial x}v^{\varepsilon} f(u^{\varepsilon}),\varphi \right\rangle &=\int_0^{\infty}\left[v_lf(U_l(t))-v_{\varepsilon}(t) f(u_{\varepsilon}(t))\right]\varphi\left(c(t)-\frac{\varepsilon}{2}t-x_{\varepsilon}, t\right)\D{t}\nonumber\\
&+\int_0^{\infty}\left[v_{\varepsilon}(t) f(u_{\varepsilon}(t))-v_rf(U_r(t))\right]\varphi\left(c(t)+\frac{\varepsilon}{2}t+x_{\varepsilon}, t\right)\D{t}.
\end{align}

\vspace{.5cm}
Next, we use the following Taylor series expansion for $\varphi$ with respect to $x=c(t)$ to evaluate the above integrals, we have
\begin{align}\label{taylor expansion}
&\varphi\left(c(t)-\frac{\varepsilon}{2}t-x_{\varepsilon},t\right)=\varphi(c(t),t)-\left(\frac{\varepsilon}{2}t+x_{\varepsilon}\right)\frac{\partial}{\partial x} \varphi(c(t),t)+ \mathcal O(\varepsilon^2), \nonumber\\
&\varphi\left(c(t)+\frac{\varepsilon}{2}t+x_{\varepsilon},t\right)=\varphi(c(t),t)+\left(\frac{\varepsilon}{2}t+x_{\varepsilon}\right)\frac{\partial}{\partial x} \varphi(c(t),t)+ \mathcal O(\varepsilon^2),\\
&\varphi(x,t)=\varphi(c(t),t)+\mathcal O(\varepsilon),\,\,\, \textnormal{for}\,\,\, c(t)-\frac{\varepsilon}{2}t-x_{\varepsilon}<x<c(t)+\frac{\varepsilon}{2}t+x_{\varepsilon}.\nonumber
\end{align}

\vspace{.5cm}
By employing the Taylor expansions into the equations \eqref{EQ2.6}-\eqref{EQ2.7} and simplifying, we obtain
\begin{align}\label{EQ2.8}
\begin{aligned}
&-\left\langle \frac{\partial}{\partial t}v^{\varepsilon},\varphi \right\rangle-\left\langle \frac{\partial}{\partial x}v^{\varepsilon} f(u^{\varepsilon}),\varphi \right\rangle =\int_{0}^{\infty}\left[c^{\cdot}(t)(v_r-v_l)-\frac{\partial}{\partial t}\left(2\left(\frac{\varepsilon}{2}t+x_{\varepsilon}\right)v_{\varepsilon}(t)\right)\right]\varphi(c(t),t)\D{t}\\
&+\int_{0}^{\infty}\left[\left(v_l f(U_l(t))-v_r f(U_r(t))\right)+\frac{\varepsilon}{2}\left(v_l+v_r\right)\right]\varphi(c(t),t)\D{t}\\
&+\int_{0}^{\infty}c^{\cdot}(t)\left[v_r-2v_{\varepsilon}(t)+v_r\right]\left(\frac{\varepsilon}{2}t+x_{\varepsilon}\right)\frac{\partial}{\partial x}\varphi(c(t),t)\D{t}+\int_{-x_{\varepsilon}}^{x_{\varepsilon}}v_{\varepsilon}(0)\varphi(x,0)\D{x}\\
&+\int_{0}^{\infty}\left[2v_{\varepsilon}(t)f(u_{\varepsilon}(t))-v_lf\left(U_l(t)\right)-v_r f(U_r(t))\right]\left(\frac{\varepsilon}{2}t+x_{\varepsilon}\right)\frac{\partial}{\partial x}\varphi(c(t),t)\D{t}+\mathcal O(\varepsilon).
\end{aligned}
\end{align}
In a similar way as above, we can calculate \eqref{distribution2}. However, we have to consider the contribution of the nonhomogeneous term. Following the same calculations, simplifying the terms for time derivative and the source, we get
\begin{align}\label{EQ2.9}
\begin{aligned}
&-\left\langle \frac{\partial}{\partial t}\left(v^{\varepsilon}u^{\varepsilon}\right),\varphi \right\rangle +\left\langle \kappa(t)(u_a(t)-u^{\varepsilon})v^{\varepsilon}, \varphi \right \rangle\\&=\int_{0}^{\infty}\left[v_{\varepsilon}(t)u_{\varepsilon}(t)-v_l U_l(t)\right]\varphi\left(c(t)-\frac{\varepsilon}{2}t-x_{\varepsilon},t\right)\left(\dot{c}(t)-\frac{\varepsilon}{2}t\right)\D{t}\\
&+\int_{0}^{\infty}\left[v_r U_r(t)-v_{\varepsilon}(t)u_{\varepsilon}(t)\right]\varphi\left(c(t)+\frac{\varepsilon}{2}t+x_{\varepsilon},t\right)\left(\dot{c}(t)+\frac{\varepsilon}{2}t\right)\D{t}\\
&+\int_{0}^{\infty}\int_{c(t)-\frac{\varepsilon}{2}t-x_{\varepsilon}}^{c(t)+\frac{\varepsilon}{2}t+x_{\varepsilon}}\left[\kappa(t)(u_{a}(t)-u_{\varepsilon}(t))v_{\varepsilon}(t)-\frac{\partial}{\partial t}\left(v_{\varepsilon}(t)u_{\varepsilon}(t)\right)\right]\varphi(x,t)\D{x}\D{t}+\int_{-x_{\varepsilon}}^{x_{\varepsilon}}v_{\varepsilon}(0)u_{\varepsilon}(0)\varphi(x,0)\D{x},
\end{aligned}
\end{align}
and for spatial derivatives, we have
\begin{align}\label{EQ2.10}
\begin{aligned}
- \left\langle \frac{\partial}{\partial x}\left(v^{\varepsilon} u^{\varepsilon} f(u^{\varepsilon})\right),\varphi \right\rangle
&=\int_{0}^{\infty}\left[v_l U_l(t)f(U_l(t))-v_{\varepsilon}(t)u_{\varepsilon}(t)f(u_{\varepsilon}(t))\right]\varphi\left(c(t)-\frac{\varepsilon}{2}t-x_{\varepsilon}, t\right)\D{t}\\
&+\int_{0}^{\infty}\left[v_{\varepsilon}(t)u_{\varepsilon}(t)f(u_{\varepsilon}(t))-v_r U_r(t)f(U_r(t))\right]\varphi\left(c(t)+\frac{\varepsilon}{2}t+x_{\varepsilon}, t\right)\D{t}.
\end{aligned}
\end{align}
Again using the Taylor expansion of the test function $\varphi$ in \eqref{EQ2.9}-\eqref{EQ2.10}, we obtain
\begin{align}\label{EQ2.11}
&-\left\langle \frac{\partial}{\partial t}\left(v^{\varepsilon}u^{\varepsilon}\right),\varphi \right\rangle-\left\langle \frac{\partial}{\partial x}\left(v^{\varepsilon} u^{\varepsilon} f(u^{\varepsilon})\right),\varphi \right\rangle +\left\langle \kappa(t)(u_a(t)-u^{\varepsilon})v^{\varepsilon}, \varphi \right \rangle\nonumber\\
& =\int_{0}^{\infty}\left[\dot{c}(t)\left(v_r U_r(t)-v_l U_l(t)\right)+\frac{\varepsilon}{2}\left(v_r U_r(t)+v_l U_l(t)\right)+v_l U_l(t)f(U_l(t))-v_r U_r(t)f(U_r(t))\right]\varphi(c(t),t)\D{t}\nonumber\\
&+\int_{0}^{\infty}2\left(\frac{\varepsilon}{2}t+x_{\varepsilon}\right)\left[\kappa(t)(u_a(t)-u_{\varepsilon}(t))v_{\varepsilon}(t)-\partial_t\left(2\left(\frac{\varepsilon}{2}t+x_{\varepsilon}\right)v_{\varepsilon}(t)u_{\varepsilon}(t)\right)
\right]\varphi(c(t),t)\D{t}\nonumber\\
&+\int_{0}^{\infty}\left[\dot{c}(t)\left(v_r U_r(t)-2v_{\varepsilon}(t)u_{\varepsilon}(t)+v_l U_l(t)\right)\right]\left(\frac{\varepsilon}{2}t+x_{\varepsilon}\right)\frac{\partial}{\partial x}\varphi(x,t)\D{t}\nonumber\\
&+2\int_0^{\infty}\left[v_{\varepsilon}(t)u_{\varepsilon}f(u_{\varepsilon}(t))-v_l U_l(t)f(U_l(t))-v_r U_r(t)f(U_r(t))\right]\left(\frac{\varepsilon}{2}t+x_{\varepsilon}\right)\varphi_x(x,t)\D{t}.
\end{align}
Now passing to the limit as $\varepsilon$ tends to zero in the equations \eqref{EQ2.8} and \eqref{EQ2.11}, we obtain the following relations
\begin{equation}\label{ode}
\begin{aligned}
& \frac{\partial}{\partial t} (\xi(t))=\dot{c}(t)[v]-[vf(U(t))],\,\,\, \xi(0)=\bar{m},\\
& \frac{\partial}{\partial t}\left(\xi(t) \chi(t)\right)+\kappa(t)(\chi(t)-u_a(t))\xi(t)=\dot{c}(t)[vU(t)]-[v U(t)f(U(t))], \,\,\, \xi(0)\chi(0)=\bar{m}\bar{u},\\
&f(\chi(t))=\dot{c}(t),
\end{aligned}
\end{equation}
where $\lim\limits_{\varepsilon \to 0}2\left(\frac{\varepsilon}{2}t+x_{\varepsilon}\right)v_{\varepsilon}(t)=\xi(t)$,  $\lim\limits_{\varepsilon \to 0}u_{\varepsilon}(t)=\chi(t)$ and $[\cdot]:=\cdot_r-\cdot_l$ denotes the jump across the discontinuity curve.
 From the first equation of \eqref{ode}, we have
\begin{equation*}
\xi(t)=c(t)[v]+v_l\int_{0}^{t}f(U_l(\theta))\D{\theta}-v_r\int_{0}^{t}f(U_r(\theta))\D{\theta}+\bar{m}.
\end{equation*}
Setting $\Theta(t)=\xi(t)\chi(t),$ second equation of \eqref{ode} can be written as
\begin{align*}
e^{\int_0^t \kappa(\theta) \D{\theta}} \left[\frac{\partial}{\partial t}\Theta(t)+\kappa(t) \Theta(t)\right]=&F(t)\left[c(t)(v_r-v_l)+v_l\int_0^t  f(U_l(\theta))\D{\theta}-v_r\int_0^t f(U_r(\theta))\D{\theta}  +\bar{m}\right]\\
& +\dot{c}(t) \Big[v_r \Big(\int_0^t F(\theta)\D{\theta}+ u_r\Big)-v_l \Big(\int_0^t F(\theta)\D{\theta}+ u_l\Big)\Big]\\
&-\Big[v_r \Big(\int_0^t F(\theta)\D{\theta}+u_r\Big)f(U_r(t))-v_l \Big(\int_0^t F(\theta)\D{\theta}+u_l\Big)f(U_l(t))\Big], 
\end{align*}
which implies the following identity.
\begin{align}\label{EQ2.12}
\begin{aligned}
\frac{\partial}{\partial t} \left(e^{\int_0^t \kappa(\theta) \D{\theta}} \Theta(t)\right)&= \dot{c}(t)[vu]+[v] \frac{\partial}{\partial t} \left(c(t) \int_0^t F(\theta)\D{\theta}\right)+[v u f(U(t))]+ \bar{m} F(t)\\
&+ v_l \frac{\partial}{\partial t}\left(\int_0^t F(\theta)\D{\theta} \cdot \int_0^t f(U_l(\theta))\D{\theta}\right)-v_r \frac{\partial}{\partial t}\left(\int_0^t F(\theta)\D{\theta} \cdot \int_0^t f(U_r(\theta))\D{\theta}\right).
\end{aligned}
\end{align}
The ODE presented in \eqref{EQ2.12} directly yields,
\begin{equation*}
\Theta(t):=\xi(t)\chi(t)=c(t)[v U(t)]+v_l U_l(t)\int_{0}^{t}f(U_l(\theta))\D{\theta}-v_rU_r(t)\int_{0}^{t}f(U_r(\theta))\D{\theta}+\bar{m}\bar{U}(t)
\end{equation*}
where $\bar{U}(t)=e^{-\int_0^t \kappa(\theta)\D{\theta}} \left(\int_0^t F(\theta)\D{\theta} + \bar{u}\right)$. Thus $\chi(t)$ is of the following form
\begin{align*}
\chi(t)=\frac{c(t)[v U(t)]+v_l(U_l(t))\int_{0}^{t}f(U_l(\theta))\D{\theta}-v_rU_r(t)\int_{0}^{t}f(U_r(\theta))\D{\theta}+\bar{m}(\bar{U}(t))}{c(t)[v]+v_l\int_{0}^{t}f(U_l(\theta))\D{\theta}-v_r\int_{0}^{t}f(U_r(\theta))\D{\theta}+\bar{m}}.
\end{align*}
Now, we need to consider the region $\sD:=\{(c(t),t) |\int_{0}^{t}f(U_r(\theta))\D{\theta}\leq c(t)\leq \int_{0}^{t}f(U_l(\theta))\D{\theta}\}$ to prove the overcompressibility condition  as $u_l>\bar{u}>u_r$. 
In the above region $\sD$, we find $\xi(t)>0,$
\begin{align*}
\chi(t)=U_l(t)+\frac{v_r(U_l(t)-U_r(t))\left(\int_{0}^{t}f(U_r(\theta))d\theta-c(t)\right)+\bar{m}\left(\bar{U}(t)-U_l(t)\right)}{\xi(t)}<U_l(t),
\end{align*}
and
\begin{align*}
\chi(t)=U_r(t)+\frac{v_l(U_l(t)-U_r(t))\left(\int_{0}^{t}f(U_l(\theta))d\theta-c(t)\right)+\bar{m}\left(\bar{U}(t)-U_r(t)\right)}{\xi(t)}>U_r(t).
\end{align*}
Furthermore, since $f$ is $C^1$ the following ODE
\begin{align*}
    &\dot{c}(t)=f(\chi(t))=\cH(c(t), t),\\
    &c(0)=0,
\end{align*} has a unique solution \cite{Coddington, hanchun2007} in the region $\sD$.
Finally, by using the increasing property of $f,$ we have $f(U_r(t))<f(\chi(t))<f(U_l(t)).$\\
Next, we turn our attention to the case of Riemann-type initial data, i.e.,
\begin{equation}\label{Rdata}
\begin{aligned}
(v, u)(x,0)=
\begin{cases}
(v_l, u_l),\,\,\,\,\,\,\ x<0,\\
(v_r, u_r),\,\,\,\,\,\, x>0,
\end{cases}
\end{aligned}
\end{equation}
where $v_{l,r}>0$.
When $u_l> u_r$ in \eqref{Rdata},  substituting the shadow wave
\begin{align}\label{sdw2}
U^{\varepsilon}=(v^{\varepsilon}, u^{\varepsilon})(x,t)=
\begin{cases}
\left(v_l, U_l(t)\right),\,\,\,\,\,\,\,\,\,\,\,\, &x< c(t)-\frac{\varepsilon}{2}t,\\
\left(v_{\varepsilon}(t), u_{\varepsilon}(t)\right),\,\,\,\,\,\, &c(t)-\frac{\varepsilon}{2}t<x<c(t)+\frac{\varepsilon}{2}t,\\
\left(v_r, U_r(t)\right),\,\,\,\,\,\,\,\,\,\,\, &x>c(t)+\frac{\varepsilon}{2}t,
\end{cases}
\end{align}
we obtain the system of ODE's \eqref{ode} where $\lim\limits_{\varepsilon \to 0}\varepsilon t v_{\varepsilon}(t)=\xi(t), \lim\limits_{\varepsilon\to 0}u_{\varepsilon}(t)=\chi(t)$ with the initial conditions $\xi(0)=0$ and $\xi(0)\chi(0)=0$. We take the particular form of $u_{\varepsilon}(t)=e^{-\int_0^t \kappa(\theta)d\theta} \left(\int_0^t F(\theta)d\theta + u_{\varepsilon}\right)$, in which $u_{\varepsilon}$'s are constants and independent of $t$. Hence, $$\lim\limits_{\varepsilon \to 0}e^{-\int_0^t \kappa(\theta)d\theta} \left(\int_0^t F(\theta)d\theta + u_{\varepsilon}\right)=e^{-\int_0^t \kappa(\theta)d\theta} \left(\int_0^t F(\theta)d\theta + \chi_0\right)=\chi(t),$$ where $\chi_0$ is constant. Thus the second equation of \eqref{ode} takes the simple form
\begin{align*}
		\frac{\partial \xi(t)}{\partial t}\cdot\chi(t)=c^{\cdot}(t)[v(U(t))]-[v(U(t))f(U(t))].
\end{align*}
		The first equation of \eqref{ode} yields
		\begin{equation*}
		f(\chi(t))\left([v](\chi(t))-[v(U(t))]\right)-(\chi(t))\left[vf(U(t))\right]+\left[v(U(t))f(U(t))\right]=0.
		\end{equation*}
		Let us consider the function $\cK(x)$ as
		\begin{equation*}
		\cK(x)=f(x)\Big([v]x-[v(U(t))]\Big)-x[vf(U(t))]+[v(U(t))f(U(t))].
		\end{equation*}
		
		One can observe that $\cK(U_l(t))=v_r(u_l-u_r)\left(f(U_l(t))-f(U_r(t))\right)>0$, since $f$ is increasing. Similarly, we have $\cK(U_r(t))<0.$
		Also, a simple calculation shows
		\begin{equation*}
		\cK^{\prime}(x)=v_r\left(x-(U_r(t))\right)f^{\prime}(x)+v_l\left((U_l(t))-x\right)f^{\prime}(x)+v_r\left(f(x)-f(U_r(t))\right)+v_l\left(f(U_l(t))-f(x)\right).
		\end{equation*}
		Since $f$ is increasing, we find $\cK^{\prime}(x)>0$ and therefore $\cK(x)=0$ has a unique solution in $(U_r(t), U_l(t))$. Now using $\cK(\chi(t))=0$,  we obtain the overcompressibility condition:

		\begin{equation*}
		U_r(t)<\chi(t)<U_l(t)\,\,\,\, \text{implies}\,\,\,\, f(U_r(t))<f(\chi(t))<f(U_l(t)).
		\end{equation*}  
  When $u_l<u_r$ the solution consists of contact discontinuity and vacuum, i.e.,
  \begin{align}\label{eqq2.16}
(v, u)(x,t)=
\begin{cases}
\left(v_l, U_l(t)\right),\,\,\,\,\,\,\,\,\,\,\,\, &x< \int_0^t f(U_l(s))ds,\\
\left(0, z(x,t)\right),\,\,\,\,\,\, &\int_0^t f(U_l(s))ds<x<\int_0^t f(U_r(s))ds,\\
\left(v_r, U_r(t)\right),\,\,\,\,\,\,\,\,\,\,\, &x>\int_0^t f(U_r(s))ds,
\end{cases}
  \end{align}
  where $z(x,t)$ is a continuous function that satisfies $z\left(\int_0^t f(U_l(s))ds,t\right)=U_l(t)$ and $z\left(\int_0^t f(U_r(s))ds,t\right)=U_r(t).$
  Summarizing the above discussion we prove the following:
  \begin{lemma}[Riemann solution]\label{RPL}
      The system \eqref{system dependent on t} with initial data \eqref{Rdata} has a unique shadow wave solution of the form \eqref{sdw2} in the case $u_l>u_r.$ If $u_l\leq u_r,$ then the solution is a combination of contact discontinuities and vacuum of the form \eqref{eqq2.16}. 
  \end{lemma}
  
\begin{lemma}[Interaction of shadow waves]\label{IL}
      If $u_l>\bar{u}>u_r,$ then the system \eqref{system dependent on t} with initial data \eqref{delta initial data} has a unique shadow wave solution of the form \eqref{sdw}.
  \end{lemma}
\begin{remark}
We observe that conservation of mass holds true due to the first equation in the system \eqref{system dependent on t}. For the quantity
\begin{align*}
\cM_0(t):=\xi(t)+\int_{-\infty}^{c(t)}v(x,t)\D{x}+\int_{c(t)}^{\infty}v(x,t)\D{x},
\end{align*}
we have $\cM_0(0)=\cM_0(t)$ for all $t>0.$ Indeed, using the first equation of \eqref{ode}, we get
\begin{align*}
    \dot{\cM_0}(t)&= \dot{\xi}(t) + v(c(t)-,t)\dot{c}(t)+\int_{-\infty}^{c(t)}\frac{\partial}{\partial t}v(x,t)\D{x}-v(c(t)+,t)\dot{c}(t)+\int_{c(t)}^{\infty}\frac{\partial}{\partial t}v(x,t)\D{x}\\
    &=\dot{\xi}(t)-\dot{c}(t)[v]+[vf(U)]=0.
\end{align*}
However, the momentum satisfies a differential equation. Define
\begin{align*}
    \cM_1(t):=\xi(t)\chi(t)+\int_{-\infty}^{c(t)}v(x,t)u(x,t)\D{x}+\int_{c(t)}^{\infty}v(x,t)u(x,t)\D{x}.
\end{align*}
A similar calculation as above leads to
\begin{align*}
    \dot{\cM_1}(t)=&\frac{\partial}{\partial t}(\xi(t)\chi(t))-\dot{c}(t)[v]+[vUf(U)]+\int_{-\infty}^{c(t)}\kappa(t)(u_{a}(t)-u(x,t))v(x,t)\D{x}\\
    &+\int_{c(t)}^{\infty}\kappa(t)(u_{a}(t)-u(x,t))v(x,t)\D{x}\\
    =&\frac{\partial}{\partial t}(\xi(t)\chi(t))-\dot{c}(t)[v]+[vUf(U)]-\kappa(t)\left[\int_{-\infty}^{c(t)}v(x,t)u(x,t)\D{x}+\int_{c(t)}^{\infty}v(x,t)u(x,t)\D{x}\right]\\
    &+\kappa(t)u_{a}(t)\left[\int_{-\infty}^{c(t)}v(x,t)\D{x}+\int_{c(t)}^{\infty}v(x,t)\D{x}\right]\\
    =&\kappa(t)(u_{a}(t)-\chi(t))\xi(t)-\kappa(t)(\cM_1(t)-\xi(t)\chi(t))+\kappa(t)u_{a}(t)(\cM_0(t)-\xi(t))\\
    =&\kappa(t)(u_{a}(t)\cM_0(t)-\cM_1(t)).
\end{align*}
Therefore momentum $\cM_1$ satisfies the ODE
\begin{align}\label{momentum ODE}
    \dot{\cM_1}(t)+\kappa(t)\cM_1(t)=\kappa(t)u_{a}(t)\cM_0(t).
\end{align}
Solving \eqref{momentum ODE} explicitly, we get 
\begin{align}\label{nonconstant momentum}
    \cM_1(t)= e^{-\int_0^t \kappa(\theta)\D{\theta}}\left[\int_0^t e^{\int_0^s \kappa(\theta)\D{\theta}}\kappa(s)u_{a}(s)\cM_0(s)\D{s}+\bar{m}\bar{u}\right].
\end{align}
When $\kappa(t):=\kappa$ and $u_{a}(t):=u_a$ are constants, using $\dot{\cM_0}(t)=0$ we get a simplified form of \eqref{momentum ODE}
\begin{align*}
    \Ddot{\cM_1}(t)+\kappa \dot{\cM_1}(t)=0,
\end{align*}
which gives
\begin{align}\label{constant momentum}
    \cM_1(t)=\left(\bar{m}\bar{u}-C\right)+C e^{-\kappa t}
\end{align}
for some constant $C.$ From the expressions \eqref{nonconstant momentum} and \eqref{constant momentum} it can be easily seen that the momentum is conserved as $t\to 0$.
\end{remark}
\subsection{Riemann problem to the system \cref{s7eq1}} Now we extend the \cref{RPL} and \cref{IL} for the drift flux equation of two-phase flow. We start with the construction of shadow wave solution for \eqref{s7eq1} adjoined with the $\delta$ initial data. Similar to \cref{subsection 2}, we consider the shadow waves solution
\begin{align}\label{s7eq3}
U^{\varepsilon}=(v^{\varepsilon}, w^{\varepsilon}, u^{\varepsilon})(x,t)=
\begin{cases}
\left(v_l, w_l, U_l(t)\right),\,\,\,\,\,\,\,\,\,\,\,\, &x< c(t)-\frac{\varepsilon}{2}t-x_{\varepsilon},\\
\left(v_{\varepsilon}(t), w_{\varepsilon}(t), u_{\varepsilon}(t)\right),\,\,\,\,\,\, &c(t)-\frac{\varepsilon}{2}t-x_{\varepsilon}<x<c(t)+\frac{\varepsilon}{2}t+x_{\varepsilon},\\
\left(v_r, w_r, U_r(t)\right),\,\,\,\,\,\,\,\,\,\,\, &x>c(t)+\frac{\varepsilon}{2}t+x_{\varepsilon},
\end{cases}
\end{align}
where $U_{l,r}(t):=e^{-\int_0^t\kappa(\theta)\D{\theta}}\cdot\left(\int_0^t F(\theta)\D{\theta}+u_{l,r}\right)$ and  $x_{\varepsilon}, v_{\varepsilon}(t)$ are $\mathcal{O}(\varepsilon)$ and $\mathcal{O}(1/\epsilon),$ and $\lim\limits_{\varepsilon \to 0}u_{\varepsilon}(t)=\chi(t).$ We study the system \eqref{s7eq1} when initial data contains a $\delta$-measure and is of the following form 
\begin{align*}
(v, w, u)(x,0)=
\begin{cases}
(v_l, v_l, u_l),\,\,\,\,\,\, &x<0,\\
(\bar{m}\delta(x), \bar{n}\delta(x),  \bar{u}),\,\,\,\,\,\, &x=0,\\
(v_r, w_r, u_r),\,\,\,\,\,\, &x>0,
\end{cases}
\end{align*}
where $v_{l,r}, w_{l, r}\geq0$, $\bar{m}, \bar{n}>0$ and $u_l>\bar{u}>u_r$. 

\vspace{.5cm}
Substituting the above shadow wave solution \eqref{s7eq3} into the system \eqref{s7eq1}, from \cref{first Defn} we have
\begin{align}
& \lim_{\varepsilon \to 0} \left[\left\langle \frac{\partial}{\partial t}v^{\varepsilon},\varphi \right\rangle+\left\langle \frac{\partial}{\partial x}\left(v^{\varepsilon} f(u^{\varepsilon})\right),\varphi \right\rangle \right]=0,\label{s7distribution1}\\
& \lim_{\varepsilon \to 0} \left[\left\langle \frac{\partial}{\partial t}w^{\varepsilon},\varphi \right\rangle+\left\langle \frac{\partial}{\partial x}\left(w^{\varepsilon} f(u^{\varepsilon})\right),\varphi \right\rangle \right]=0,\label{s7distribution2}\\
&\lim_{\varepsilon \to 0} \left[\left\langle \frac{\partial}{\partial t}\left((v^{\varepsilon}+w^{\varepsilon})u^{\varepsilon}\right),\varphi \right\rangle+\left\langle \frac{\partial}{\partial x}\left((v^{\varepsilon}+w^{\varepsilon})u^{\varepsilon} f(u^{\varepsilon})\right),\varphi \right\rangle -\left\langle \kappa(t)(u_a(t)-u^{\varepsilon})(v^{\varepsilon}+w^{\varepsilon}), \varphi \right\rangle \right]=0,\label{s7distribution3}
\end{align}
for all $\varphi \in C^{\infty}_c(\RR\times [0, \infty))$.\\

Now passing to the limit as $\varepsilon$ tends to zero in the equations \eqref{s7distribution1}-\eqref{s7distribution3}, we obtain the following relations
\begin{align}\label{s7eq8}
& \frac{\partial}{\partial t} (\xi_v(t))=\dot{c}(t)[v]-[vf(U(t))],\,\,\, \xi_v(0)=\bar{m},\\
& \frac{\partial}{\partial t} (\xi_w(t))=\dot{c}(t)[w]-[wf(U(t))],\,\,\, \xi_w(0)=\bar{n},\label{s7eq2.25}\\
& \frac{\partial}{\partial t}\left((\xi_v(t)+\xi_w(t))\chi(t)\right)+\kappa(t)(\chi(t)-u_a(t))(\xi_v(t)+\xi_w(t))=\dot{c}(t)[(v+w)U(t)]-[(v+w) U(t)f(U(t))],\nonumber\\
&(\xi_v(0)+\xi_w(0))\chi(0)=(\bar{m}+\bar{n})\bar{u},\label{s7eq2.26}\\
&f(\chi(t))=\dot{c}(t),\label{s7eq2.27}
\end{align}
where $\lim\limits_{\varepsilon \to 0}2\left(\frac{\varepsilon}{2}t+x_{\varepsilon}\right)v_{\varepsilon}(t)=\xi_v(t)$,  $\lim\limits_{\varepsilon \to 0}2\left(\frac{\varepsilon}{2}t+x_{\varepsilon}\right)w_{\varepsilon}(t)=\xi_w(t)$, $\lim\limits_{\varepsilon \to 0}u_{\varepsilon}(t)=\chi(t)$ and $[\cdot]:=\cdot_r-\cdot_l$ denotes the jump across the discontinuity curve.

Following exactly the same calculations of \cref{subsection 2} from \eqref{s7eq8}-\eqref{s7eq2.27}, we obtain
\begin{align*}
\Theta(t)+\Phi(t):&=\chi(t)(\xi_v(t)+\xi_w(t))
=c(t)[(v+w) U(t)]\\
&+(v_l+w_l)U_l(t)\int_{0}^{t}f(U_l(\theta))\D{\theta}-(v_r+w_r)U_r(t)\int_{0}^{t}f(U_r(\theta))\D{\theta}+(\bar{m}+\bar{n})\bar{U}(t),
\end{align*}
where $\Theta(t):=\xi_v(t)\chi(t)$ and $\Phi(t):=\xi_w(t)\chi(t)$ and $\bar{U}$ denotes the same from \cref{subsection 2}. Again, the overcompressibility for shadow wave solution \eqref{s7eq3} follows from the arguments of \cref{subsection 2}.

Now we will consider the case of Riemann-type initial data, i.e.,
\begin{equation}\label{s7eq10}
\begin{aligned}
(v, w, u)(x,0)=
\begin{cases}
(v_l, w_l, u_l),\,\,\,\,\,\,\ x<0,\\
(v_r, w_r, u_r),\,\,\,\,\,\, x>0,
\end{cases}
\end{aligned}
\end{equation}
where $v_{l,r}, w_{l,r}>0$.
For Riemann type of initial data \eqref{s7eq10} when $u_l> u_r,$  substituting the shadow wave
\begin{align*}
U^{\varepsilon}=(v^{\varepsilon}, w^{\varepsilon},  u^{\varepsilon})(x,t)=
\begin{cases}
\left(v_l, w_l,  U_l(t)\right),\,\,\,\,\,\,\,\,\,\,\,\, &x< c(t)-\frac{\varepsilon}{2}t,\\
\left(v_{\varepsilon}(t), w_{\varepsilon}(t), u_{\varepsilon}(t)\right),\,\,\,\,\,\, &c(t)-\frac{\varepsilon}{2}t<x<c(t)+\frac{\varepsilon}{2}t,\\
\left(v_r, w_r, U_r(t)\right),\,\,\,\,\,\,\,\,\,\,\, &x>c(t)+\frac{\varepsilon}{2}t,
\end{cases}
\end{align*}
we have the ODE's \eqref{s7eq8}-\eqref{s7eq2.26} with $\lim\limits_{\varepsilon \to 0}\varepsilon t v_{\varepsilon}(t)=\xi_v(t), \lim\limits_{\varepsilon \to 0}\varepsilon t w_{\varepsilon}(t)=\xi_w(t) ~\text{and}~\lim\limits_{\varepsilon\to 0}u_{\varepsilon}(t)=\chi(t)$ with the initial conditions $\xi_v(0)=0$, $\xi_w(0)=0$ and $(\xi_v(0)+\xi_w(0))\chi(0)=0$. As in the \cref{sec2}, we take the particular form of $u_{\varepsilon}(t)=e^{-\int_0^t \kappa(\theta)d\theta} \left(\int_0^t F(\theta)d\theta + u_{\varepsilon}\right)$, where $u_{\varepsilon}$'s are constants and independent of $t$. Hence $$\lim\limits_{\varepsilon \to 0}e^{-\int_0^t \kappa(\theta)d\theta} \left(\int_0^t F(\theta)d\theta + u_{\varepsilon}\right)=e^{-\int_0^t \kappa(\theta)d\theta} \left(\int_0^t F(\theta)d\theta + \chi_0\right),$$ where $\chi_0$ is constant. Thus the third equation of \eqref{s7eq8} takes the simple form
\begin{align*}
		\frac{\partial}{\partial t}\left((\xi_v(t)+\xi_w(t)\right)\cdot\chi(t)=c^{\cdot}(t)[(v+w)U(t)]-[(v+w)U(t)f(U(t))].
\end{align*}
		The first equation of \eqref{s7eq8} implies
		\begin{equation*}
		f(\chi(t))\left([v+w](\chi(t))-[(v+w)(U(t))]\right)-(\chi(t))\left[(v+w)f(U(t))\right]+\left[(v+w)(U(t))f(U(t))\right]=0.
		\end{equation*}
Now considering the function $\cK(x)$ as
		\begin{equation*}
		\cK(x)=f(x)\Big([v+w]x-[(v+w)(U(t))]\Big)-x[(v+w)f(U(t))]+[(v+w)(U(t))f(U(t))].
		\end{equation*}
and following the arguments of \cref{sec2}, we show that the overcompressibility condition
		\begin{equation*}
		U_r(t)<\chi(t)<U_l(t)\,\,\,\, \text{implies}\,\,\,\, f(U_r(t))<f(\chi(t))<f(U_l(t))
		\end{equation*}  
holds.

In the case of $u_l<u_r$ the solution consists of contact discontinuity and vacuum, i.e.,
  \begin{align*}
(v, w, u)(x,t)=
\begin{cases}
\left(v_l, w_l,  U_l(t)\right),\,\,\,\,\,\,\,\,\, &x< \int_0^t f(U_l(s))ds,\\
\left(0, 0, z(x,t)\right),\,\,\,\,\,\, &\int_0^t f(U_l(s))ds<x<\int_0^t f(U_r(s))ds,\\
\left(v_r, w_r, U_r(t)\right),\,\,\,\,\,\,\,\,\,\,\, &x>\int_0^t f(U_r(s))ds,
\end{cases}
  \end{align*}
  where $z(x,t)$ is a continuous function that satisfies $z\left(\int_0^t f(U_l(s))ds,t\right)=U_l(t)$ and $z\left(\int_0^t f(U_r(s))ds,t\right)=U_r(t).$
Summarizing the above calculations, we proved the analogous results of \cref{RPL} and \cref{IL} for system \eqref{s7eq1}.
\section{Entropy inequality}\label{sec3}
In this section, we introduce the \textit{dissipative shadow wave} solution for the system \eqref{system dependent on t} and show the equivalence of dissipative shadow waves and overcompressibility condition. We can express the entropy-entropy flux pair as follows:
\begin{align*}
\eta(v,u)=\frac{1}{2}vf^2(u)~~\text{and}~~q(v,u)=\frac{1}{2}vf^3(u),
\end{align*}
where $\eta$ and $q$ represent entropy and entropy flux, respectively.
\begin{definition}
	A shadow wave $U^{\varepsilon}=(v^{\varepsilon},u^{\varepsilon})$ is called a dissipative shadow wave solution if it satisfies the following entropy inequality
	\begin{align}\label{entropy}
	\lim_{\varepsilon\to 0}\left(\left\langle\frac{\partial}{\partial t}\eta(v^{\varepsilon},u^{\varepsilon}), \varphi\right\rangle+\left\langle\frac{\partial}{\partial x}q(v^{\varepsilon},u^{\varepsilon}), \varphi\right\rangle-\left\langle\kappa(t)(u_{a}(t)-u^{\varepsilon})v^{\varepsilon}f(u^{\varepsilon})f^{\prime}(u^{\varepsilon}), \varphi\right\rangle\right)\leq 0,
	\end{align}
	for all non-negative $\varphi \in C^{\infty}_c(\RR\times(0,\infty))$.
\end{definition}
\begin{theorem}
	A shadow wave solution $U^{\varepsilon}=(v^{\varepsilon}, u^{\varepsilon})$ is dissipative if and only if 
	\begin{align}\label{dissipative}
	&2\xi(t)f(\chi(t))f^{\prime}(\chi(t)) \left\{\frac{\partial}{\partial t}\chi(t)+\kappa(t)\left(\chi(t)-u_{a}(t)\right)\right\}+ \frac{\partial}{\partial t}\xi(t)f^2(\chi(t))\nonumber\\
	&\leq f(\chi(t))[vf^2(U(t))]-[vf^3(U(t))],
	\end{align}
	where $\xi(t)$ and $\chi(t)$ are defined as before.
\end{theorem}
\begin{proof} Inserting the the shadow wave solution \eqref{sdw} and following the similar calculation as previous, we obtain
\begin{align}\label{eqq3.3}
-\left\langle\frac{\partial}{\partial t}\eta(v^{\varepsilon},u^{\varepsilon}), \varphi\right\rangle&=\frac{1}{2}\int_{0}^{\infty}\left(v_{\varepsilon}(t)f^2(u_{\varepsilon}(t))-v_lf^2(U_l(t))\right)\left(\dot{c}(t)-\frac{\varepsilon}{2}\right)\varphi\left(c(t)-\frac{\varepsilon}{2}t-x_{\varepsilon}, t\right)\D{t}\nonumber\\
&+\frac{1}{2}\int_{0}^{\infty}\left(v_rf^2(U_r(t))-v_{\varepsilon}(t)f^2(u_{\varepsilon}(t))\right)\left(\dot{c}(t)+\frac{\varepsilon}{2}\right) \varphi\left(c(t)+\frac{\varepsilon}{2}t+x_{\epsilon}, t\right)\D{t}\nonumber\\
&-\int_{0}^{\infty}\int_{-\infty}^{c(t)-\frac{\varepsilon}{2}t-x_{\varepsilon}}\kappa(t)(u_a(t)-u_l)v_lf(U_l(t))f^{\prime}(U_l(t))\varphi(x,t)\D{x}\D{t} \nonumber\\
&-\frac{1}{2}\int_{0}^{\infty}\int_{c(t)-\frac{\varepsilon}{2}t-x_{\varepsilon}}^{c(t)+\frac{\varepsilon}{2}t+x_{\varepsilon}}\frac{\partial}{\partial t}\Big(v_{\epsilon}(t)f^2(u_{\epsilon}(t))\Big)\varphi(x,t)\D{x}\D{t}\nonumber\\
&-\int_{0}^{\infty}\int_{c(t)+\frac{\varepsilon}{2}t+x_{\varepsilon}}^{\infty}\kappa(t)(u_a(t)-u_r)v_rf(U_r(t))f^{\prime}(U_r(t))\varphi (x,t)\D{x}\D{t}, 
\end{align}
where in  the third and fifth lines we used 
\[\dot{U}_{l,r}(t)=\kappa(t)\left(u_{a}(t)-U_{l,r}(t)\right).\] Similarly,
\begin{align}\label{eqq3.4}
    -\left\langle\frac{\partial}{\partial x}q(v^{\varepsilon},u^{\varepsilon}), \varphi\right\rangle&=\frac{1}{2}\int_{0}^{\infty}\left(v_lf^3(U_l(t))-v_{\varepsilon}(t)f^3(u_{\varepsilon}(t))\right)\varphi\left(c(t)-\frac{\varepsilon}{2}t-x_{\varepsilon},t\right)\D{t}\nonumber\\
    &+\frac{1}{2}\int_0^{\infty}\left(v_{\varepsilon}(t)f^3(u_{\varepsilon}(t))-v_rf^3(U_r(t))\right)\varphi\big(c(t)+\frac{\varepsilon}{2}t+x_{\varepsilon},t\big)\D{t}.
\end{align}
Now using the Taylor expansions \eqref{taylor expansion} in \eqref{eqq3.3}-\eqref{eqq3.4} and considering the contributions from the source term of \eqref{entropy}, we find
	\begin{align}\label{eqq3.5}
	&\frac{1}{2}\int_{0}^{\infty}\Bigg[\left(\dot{c}(t)[vf^2(U(t))]+\frac{\varepsilon}{2}\left(v_rf(U_r(t))+v_lf^2(U_l(t))\right)\right)-\frac{\partial}{\partial t}\left(2\left(\frac{\varepsilon}{2}t+x_{\varepsilon}\right)v_{\varepsilon}(t)f^2(u_{\varepsilon}(t))\right)\nonumber\\
	&-[vf^3(U(t))]+4\kappa(t)\left(\frac{\varepsilon}{2}t+x_{\varepsilon}\right)(u_{a}(t)-u_{\varepsilon}(t))v_{\varepsilon}(t)f(u_{\varepsilon}(t))f^{\prime}(u_{\varepsilon}(t))\Bigg]\varphi(c(t),t)\D{t} \nonumber\\
	&+\frac{1}{2}\int_{0}^{\infty}\Bigg[c^{\cdot}(t)\left(v_rf^2(U_r(t))+v_lf(U_l(t))-2v_{\varepsilon}(t)f^2(U_{\varepsilon}(t))\right)\nonumber\\
	&+\left(2v_{\varepsilon}(t)f^3(u_{\varepsilon}(t))-v_lf^3(U_l(t))-v_rf^3(U_r(t))\right)\left(\frac{\varepsilon}{2}t+x_{\varepsilon}\right)\Bigg]\frac{\partial}{\partial x}\varphi(c(t),t)\D{t}.
	\end{align}
From the equation \eqref{entropy}, we have
	\begin{align}\label{eq3.3}
	\lim_{\varepsilon \to 0}\int_{0}^{\infty}\int_{-\infty}^{\infty}\left(\eta(v^{\varepsilon}, u^{\varepsilon})\frac{\partial}{\partial t}\varphi(x,t)+q(v^{\varepsilon}, u^{\varepsilon})\frac{\partial}{\partial t}\varphi(x,t)\right)\D{x}\D{t}\nonumber\\
 + \int_{0}^{\infty}\int_{-\infty}^{\infty}\kappa(t)(u_a(t)-u^{\varepsilon})f(u^{\varepsilon})f^{\prime}(u^{\varepsilon})v^{\varepsilon}\varphi(x,t)\D{x}\D{t}\geq 0.
	\end{align}
	
Now passing to the limit as $\varepsilon \to 0$, from \eqref{eqq3.5} and \eqref{eq3.3}, we obtain
\begin{equation}\label{diss}
\frac{\partial}{\partial t}\left(\xi(t)f^2(\chi(t))\right)\leq \dot{c}(t)[vf^2(U(t))]-[vf^3(U(t))] + 2\kappa(t)\left(u_{a}(t)-\chi(t)\right)\xi(t)f(\chi(t))f^{\prime}(\chi(t)),
\end{equation}
and simplifying we get the inequality \eqref{dissipative}.
\end{proof}

\vspace{.5cm}
\begin{lemma}
	Let $f^2$ be a convex function, then a shadow wave solution $U^{\varepsilon}$ is dissipative if and only if it satisfies the overcompressibility condition.
\end{lemma}
\begin{proof}Using convexity of $f^2$, we have a function $g$ such that
	\begin{align}\label{convexity}
	f^2(x)=f^2(x_0)+2(x-x_0)f(x_0)f^{\prime}(x_0)+g(x)
	\end{align}
	with $g\geq 0$ and $g(x_0)=g^{\prime}(x_0)=0.$
Inserting \eqref{convexity} into the inequality \eqref{diss} and taking $x=\chi(t), x_0=\chi(t_0)$ we have
	\begin{align}
	&\frac{\partial}{\partial t}\left(\xi(t)\left[2f(\chi(t_0))f^{\prime}(\chi(t_0))\chi(t)+f^2(\chi(t_0))-2f(\chi(t_0))f^{\prime}(\chi(t_0))\chi(t_0)+ g(\chi(t))\right]\right)\Bigg|_{t=t_0}\nonumber\\
	&\leq f(\chi(t_0)) [vf^2(U(t_0))]-[vf^3(U(t_0))] + 2\kappa(t_0)\left(u_{a}(t_0)-\chi(t_0)\right)\xi(t_0)f(\chi(t_0))f^{\prime}(\chi(t_0)).\nonumber
	\end{align}
	Since $g(\chi(t_0))=g^{\prime}(\chi(t_0))=0$, we find $\frac{\partial}{\partial t}(\xi(t)g(\chi(t)))\Big|_{t=t_0}=0$ and the above inequality takes the form
	\begin{equation*}
	\begin{aligned}
	&2f(\chi(t_0))f^{\prime}(\chi(t_0)) \frac{\partial}{\partial t}\left(\xi(t)\chi(t)\right)\Big|_{t=t_0}+\left(f^2(\chi(t_0))-2f(\chi(t_0))f^{\prime}(\chi(t_0))\chi(t_0) \right) \frac{\partial}{\partial t}\xi(t)\Big|_{t=t_0}\\
	&\leq f(\chi(t_0)) [vf^2(U(t_0))]-[vf^3(U(t_0))] + 2\kappa(t_0)(u_{a}(t_0)-\chi(t_0))\xi(t_0)f(\chi(t_0))f^{\prime}(\chi(t_0)).
	\end{aligned}
	\end{equation*}
	Now using the ODE \eqref{ode} and simplifying the above inequality, we find
	\begin{equation}\label{entropy over}
	v_r\Big(f(\chi(t_0))-f(U_r (t_0))\Big)g(U_r (t_0))-v_l\Big(f(\chi(t_0))-f(U_l (t_0))\Big)g(U_l (t_0))\geq 0.
	\end{equation}
	If the shadow wave is overcompressive, non-negativity of $v_{l,r}$ and positivity of $g$ implies that the above inequality \eqref{entropy over} holds, i.e., the shadow wave is dissipative for any $t_0$. The other way is not difficult to see and follows from the arguments of \cite{marko-manas}.
\end{proof}
\begin{remark}
 In the above proof, we assumed that $f\in C^1$ and $f^2$ be convex, but there is no assumption on the  sign of $f^{\prime}$,  which is assumed to be positive in the previous section. $f^{\prime}> 0$ becomes an admissible assumption if we consider the relation of the system \eqref{system dependent on t} with scalar conservation laws. As we mentioned earlier for the smooth solution the system \eqref{system dependent on t} reduced to
\[u_t+f(u)u_x=\kappa(t)(u_{a}(t)-u)\] which is of the form $u_t+F(u)_x=h(x,t,u)$ with $F^{\prime}(u)=f(u).$ For scalar conservation laws convexity condition $F^{\prime \prime}(u)=f^{\prime}(u) > 0$ on flux function is useful to obtain explicit formulas. Let us define the sets $S_1:=\left\{f \,\,|\,\, f \,\,\,\text{is}\,\,\, C^1 \,\,\,\text{and}\,\,\, f^{\prime}> 0\right\}$ and $S_2:=\left\{f\,\, |\,\, f^2\,\,\, \text{is convex}\right\}$. We observe that there is no containment relation between $S_1$ and $S_2.$ For instance, $S_1\not\subset S_2:$ let $f(x)=\log x,$ $f^{\prime}=\frac{1}{x}> 0$ for $x>0$ and $(f^2)^{\prime\prime}(x)=\frac{2}{x^2}(1-\log x) <0$ for some $x>0.$ Again $S_2\not\subset S_1:$ let $f(x)=-x,$ then $f^2=x^2$ which is convex but $f^{\prime}=-1<0.$ Hence in this section we are working with the class of functions $S_1\cap S_2.$
\end{remark}
\section{Global existence of entropy solution }\label{sec4}
In this section, we begin with a concise description of the shadow wave tracking procedure. Subsequently, we give the proofs for  \cref{TH1.1} and \cref{TH1.2}.
\subsection{The wave tracking algorithm} Let us introduce the shadow wave tracking algorithm. To begin, we establish the notations primarily borrowed from \cite{RN21}.
\begin{enumerate}
    \item [1.] For any two given states $(v_i, u_i)$ and $(v_j, u_j)$ with $i<j, v_i, v_j >0,$ $SDW_{i,j}$ denotes a shadow wave solution joining $(v_i,u_i)$ on the left and $(v_j, u_j)$ on the right. Note that $SDW_{i,j}$ exists if $u_i>u_j.$
    \item [2.] For any given three states $(v_i, u_i), (v_m, u_m), (v_j, u_j)$ with $i<m<j, v_i, v_m, v_j >0.$ Suppose the states $(v_i, u_i)$ and $(v_m, u_m)$ are connected by contact discontinuities $CD^i_1, CD^m_2$ and a vacuum $Vac_{i,m}=(0, u_i(x,t))$ and the states $(v_m, u_m)$ and $(v_j, u_j)$ are connected by a shadow wave $SDW_{m,j}.$ Now at time $t=T,$ if the states $(0, u_i(x, T))$ on the left and $(v_j, u_j)$ on the right are connected by a shadow wave, we denote it as $~^i\!{SDW}_j.$ Note that $~^i\!{SDW}_m=CD^m_2.$
    \item [3.] In a similar situation described above, $SDW^j_i$ denotes the shadow wave joining the states $(v_i, u_i)$ on the left to the state $(0, u_j(x,t))$ on the right.
    \item [4.] Finally $~^i\!{SDW}^j$ denotes a shadow wave joining the state $(0, u_i(x,t))$ on the left to a state $(0, u_j(x,t))$ on the right.
\end{enumerate}

Now we describe the wavefront tracking algorithm. We take the initial data for the system \eqref{system dependent on t} as 
\begin{align}\label{initial data approx}
    (v, u)(x,0)=\begin{cases}
        (v_0, u_0),\,\,\,\,\,\, &x\leq R,\\
        (v(x), u(x)),\,\,\,\,\,\, &x>R.
    \end{cases}
\end{align}
Here $v_0, u_0 \in \RR,$ $v_0>0$ and $v(x)>0, u(x) \in C_b\left([R,\infty)\right).$ Let $\varepsilon$ be given sufficiently small positive real number. Define $((v^{\varepsilon}(x), u^{\varepsilon}(x)))_{\varepsilon}$ as a piecewise constant approximation for $(v(x), u(x)).$ For a fixed $\varepsilon >0,$ take a corresponding partition $R:=Y_0<Y_1<Y_2 \cdot \cdot \cdots$ of the interval $[R, \infty)$ that satisfies $Y_i-Y_{i+1} \leq \rho(\varepsilon), i=0,1,2, \cdot\cdot\cdots$ where $\rho(\varepsilon)\to 0$ as $\varepsilon \to 0.$ Now choose $v_{i+1}:=v^{\varepsilon}(x)=v(Y_{i+1})$ and $u_{i+1}:=u^{\varepsilon}(x)=u(Y_{i+1})$ for $x \in (Y_i, Y_{i+1}], i \in \NN \cup \{0\}.$

\noindent\textbf{Algorithm}:

 \noindent\textbf{Step 0:} Let $v_0, u_0$ be given as in \eqref{initial data approx} and $\left(\{v_i\}, \{u_i\}\right)_{i\in \NN}$ be a piecewise constant approximation obtained by the procedure describes as above. Therefore an approximation of the initial data \eqref{initial data approx} can be expressed as follows:
 \begin{align}\label{initial data approx 2}
    U_i =(v, u)(x,0)=\begin{cases}
        (v_i, u_i),\,\,\,\,\,\, &x\leq Y_i,\\
        (v_{i+1}, u_{i+1}),\,\,\,\,\,\, &x>Y_i,
    \end{cases}
\end{align}
for $i=0,1,2, \cdot\cdot\cdots$.

\noindent\textbf{Step 1:} Let $S_0$ denotes the collection of all states at time $t=T_0,$ i.e., $S_0=\{U_i \,\,\,\,|\,\,\,\, i \in I_0\}$ where $U_i$'s are given by \eqref{initial data approx 2} and $I_0=\{0,1,2,3, \cdot\cdots\}$ be an index set. A shadow wave solution of the Riemann problem consisting of the states in $S_0$ gives rise to two possibilities: either there is no further interaction of waves or two (possibly more) waves interact at time $t=T_1.$ If there is no further interaction then the procedure has to be stopped at $t=T_0.$ In the event of wave interaction, there are four possible ways in which the interaction can occur, as described above. Each interaction leads to a single resulting shadow wave. The new resulting waves and the non-interacted waves together give a new collection of initial states $S_1=\{U_i\,\,\,\,|\,\,\,\, i\in I_1\}$ where $I_1\subset I_0$ is a new index set for $t\geq T_1.$ 

\noindent \textbf{Step j to j+1:} Suppose that $j$-th interaction happens at a time $t=T_j$. Then eliminating all the middle states from $S_{j-1},$ we obtain a new collection of states $S_j=\{U_i\,\,\,\, |\,\,\,\, i\in I_j\}$ where $I_j=\{0,j_1, j_2,\cdots\cdots\}\subseteq I_{j-1}, 1\leq j_1< j_2<\cdots$ denotes the index set at $j$-th level. All non-interacting waves continue to propagate after $t>T_j$. We can repeat the procedure by substituting $j+1$ in place of $j$ after a new interaction at $t=T_{j+1}.$ The algorithm finishes if there is no such $T_{j+1}$.
 
 The above-stated procedure will help us to prove the global existence of an admissible solution to the problem \eqref{system dependent on t} and \eqref{initial data approx}.
\subsection{Proof of \cref{TH1.1}} Let us consider the initial data 
\begin{equation*}
    \begin{aligned}
    (v, u)(x,0)=
    \begin{cases}
    (v_l, u_l),\,\,\,\,\,\, &x<X_1,\\
    (v_m, u_m),\,\,\,\,\,\, &X_1<x<X_2,\\
    (v_r, u_r),\,\,\,\,\,\, &x>X_2,
    \end{cases}
    \end{aligned}
\end{equation*}
with $u_l>u_m>u_r$. Then two shadow waves are emanating from $X_1$ and $X_2$ with the central shadow wave line $x(t)=X_1+c(t)$ and $x(t)=X_2+c(t)$. Let $x_1(t)=X_1+c(t)+\frac{\varepsilon}{2}t+x_{1,\varepsilon}$ and $x_2(t)=X_2+c(t)-\frac{\varepsilon}{2}t-x_{2,\varepsilon}$ are the right external shadow wave line and left external shadow wave line, respectively. Suppose $T_{\varepsilon}$ is the time when two external shadow wave lines interact and $T$ is the time when two central shadow wave lines interact. Then $T-T_{\varepsilon}=\mathcal O(\varepsilon)$. In summary of the aforementioned fact, we can state the following lemma.
 \begin{lemma}[\cite{RN21}, Lemma 3.2]\label{LM4.1}
Let two approaching shadow waves with the central lines given by $x=c_l(t)$
and $x=c_r(t)$ interact at time $t=\tilde{T}$. The value of $\tilde{T}$ is obtained by solving the equation
\begin{align*}
c_l(t)+\frac{\varepsilon}{2}(t-T_l)+x_{l, \varepsilon}=c_r(t)-\frac{\varepsilon}{2}(t-T_r)+x_{r, \varepsilon}
\end{align*}
where $c_l(t)+\frac{\varepsilon}{2}(t-T_l)+x_{l, \varepsilon}$ is the right external SDW line of the first approaching shadow wave, while $c_r(t)-\frac{\varepsilon}{2}(t-T_r)+x_{r, \varepsilon}$ is the left external SDW line of the second approaching shadow wave. Also, let $x_{l, \varepsilon}, x_{r, \varepsilon}\sim \varepsilon$. A solution T to $c_l(t)=c_r(t)$ will be called the interaction time since the area bounded by two external shadow wave lines, and the lines $t = T$ and $t=\tilde{T}$ is of order $\varepsilon^2$ and all terms of order $\varepsilon^\alpha, \alpha>1$ are neglected. Note that $T=\tilde{T}+\mathcal{O}(\varepsilon)$. The assertion stays true if one of the shadow waves is substituted by a contact discontinuity.
\end{lemma}
The next lemma is the first step towards proving \cref{TH1.1}. We will be using \cref{LM4.1} in the following proofs without mentioning it repeatedly.
 \begin{lemma}\label{LM3.2}
     Let $v(x)>0, u(x) \in L^{\infty}([R, \infty))\cap C([R, \infty))$ and $u(x)$ be an increasing function on $[R, \infty)$. Take a partition $\{Y_i\}_{i\in \NN \cup \{0\}}$ of $[R, \infty)$ such that $Y_0=R, C_1 \varepsilon^{\alpha}<Y_i-Y_{i+1}<C_2 \rho(\varepsilon)$ for every $i=0,1,2,\cdot\cdots$ where $C_1, C_2\geq 1, \alpha\in (0, 1)$ and $\rho(\varepsilon) \to 0$ as $\varepsilon \to 0.$ Then there exists a global admissible solution to \eqref{system dependent on t} and \eqref{initial data approx}. More precisely, there exists a function $U^{\varepsilon}=(v^{\varepsilon}, u^{\varepsilon})$ that satisfies
     \begin{align}\label{E1}
         \lim_{\varepsilon \to 0}\begin{cases}
             \left\langle \frac{\partial}{\partial t}v^{\varepsilon},  \varphi\right\rangle + \left\langle \frac{\partial}{\partial x}(v^{\varepsilon}f(u^{\varepsilon})),  \varphi\right\rangle=0,\\
             \left\langle \frac{\partial}{\partial t}(v^{\varepsilon} u^{\varepsilon}),  \varphi\right\rangle + \left\langle \frac{\partial}{\partial x}(v^{\varepsilon} u^{\varepsilon} f(u^{\varepsilon})),  \varphi\right\rangle=\left\langle \kappa(t)(u_{a}(t)-u^{\varepsilon})v^{\varepsilon}, \varphi \right\rangle,
         \end{cases}
     \end{align}
     for every test function $\varphi \in C^{\infty}_c\left(\RR \times [0, \infty)\right)$ and the admissibility condition.
 \end{lemma}
 \begin{proof}
   To prove this lemma, we consider two cases. 
     
     \noindent \textbf{Case 1.} When $u_0 \leq u(Y_1).$ Since $u$ is an increasing function on $[R, \infty),$ we get $u_0<u_1<u_2, \cdot\cdot\cdots$ and at each $\{Y_i\}_{i \in \NN\cup \{0\}}$ we have a solution as a combination of contact discontinuities and vacuum, i.e., $CD^i_1+Vac_{i, i+1}+CD^{i+1}_2$ for $i=0,1,2, \cdot\cdot\cdots.$ As the speeds of the fronts do not overlap each other, no interaction occurs in this case.

     \noindent \textbf{Case 2.} When $u_0>u(Y_1).$ Since $u(x)$ is bounded for $x>R,$ we have $\lim\limits_{i\to \infty}u_i=\tilde{u}.$ Then two cases arise.
     
     \noindent \textbf{Subcase 1.} Let $\tilde{u} \leq u_0.$ Let $T_{01}$ be the time when first interaction occurs, i.e., $SDW_{0,1}$ meets $CD^1_1=Y_1+\int_0^t f(U_1(\xi))d\xi.$ Therefore the first step is to study $-\left\langle\frac{\partial}{\partial t} v^{\varepsilon}, \varphi\right\rangle$ in the interval $[0, T_{01}].$ In the interval $[0, T_{0,1}],$ we have
     \begin{align*}
         -\left\langle\frac{\partial}{\partial t} v^{\varepsilon}, \varphi\right\rangle&=\int_0^{T_{01}} \int_{-\infty}^{\infty}v^{\varepsilon} \frac{\partial}{\partial t}\varphi (x,t)\D{x}\D{t}-\int_{-\infty}^{\infty}\left[(v^{\varepsilon}\varphi)(x, T_{01}-)-(v^{\varepsilon}\varphi)(x,0)\right]\D{x}\\
         &=I_0-\int_{-\infty}^{\infty}\left[(v^{\varepsilon}\varphi)(x, T_{01}-)-(v^{\varepsilon}\varphi)(x,0)\right]\D{x}.
     \end{align*}
     Inserting the shadow wave
     \begin{align*}
         U^{\varepsilon}=(v^{\varepsilon}, u^{\varepsilon})=\begin{cases}
             \left(v_0, U_0(t)\right),\,\,\,\,\,\,\,\,\,\,\,\, &x< R+c(t)-\frac{\varepsilon}{2}t,\\
\left(v_{0,\varepsilon}(t), u_{0,\varepsilon}(t)\right),\,\,\,\,\,\, &R+c(t)-\frac{\varepsilon}{2}t<x<R+c(t)+\frac{\varepsilon}{2}t,\\
\left(v_1, U_1(t)\right),\,\,\,\,\,\,\,\,\,\,\, &x>R+c(t)+\frac{\varepsilon}{2}t,
         \end{cases}
     \end{align*}
     in the above equation, we obtain
\begin{align}
    I_0&=\int_{0}^{T_{01}}\int_{-\infty}^{R+c(t)-\frac{\varepsilon}{2}t}v_0\frac{\partial}{\partial t}\varphi(x, t)\D{x}\D{t}
+\int_{0}^{T_{01}}\int_{R+c(t)-\frac{\varepsilon}{2}t}^{R+c(t)+\frac{\varepsilon}{2}t}v_{0, \varepsilon}(t)\frac{\partial}{\partial t}\varphi(x, t)\D{x}\D{t}\nonumber\\
&+\int_{0}^{T_{01}}\int_{R+c(t)+\frac{\varepsilon}{2}t}^{Y_1+\int_{0}^{t}f(U_1(\xi))d\xi}v_1\frac{\partial}{\partial t}\varphi(x, t)\D{x}\D{t}\nonumber\\
&+\sum_{i=1}^{\infty}\int_{0}^{T_{01}}\int_{Y_i+\int_{0}^{t}f(U_i(\xi))d\xi}^{Y_{i+1}+\int_{0}^{t}f(U_{i+1}(\xi))d\xi}v_{i+1}\frac{\partial}{\partial t}\varphi(x, t)\D{x}\D{t} =I_{01}+I_{02}+I_{03}+I_{04}\nonumber.
 \end{align}
\noindent Simplifying the above expressions, we get
\begin{align*}
I_{01}=-\int_{0}^{T_{01}}v_0 \varphi\left(R+c(t)-\frac{\varepsilon}{2}t,t\right) \left(\dot{c}(t)-\frac{\varepsilon}{2}\right)\D{t}+\int_{-\infty}^{R+c(T_{01})-\frac{\varepsilon}{2}T_{01}}v_0 \varphi(x, T_{01}) \D{x}-\int_{-\infty}^{R}v_0\varphi(x, 0)\D{x},
\end{align*}
\begin{align*}
I_{02}=\int_{R+c(T_{01})-\frac{\varepsilon}{2}T_{01}}^{R+c(T_{01})+\frac{\varepsilon}{2}T_{01}}v_{0,\varepsilon}(T_{01}) \varphi\left(x, T_{01}\right)\D{x}-\int_{0}^{T_{01}}v_{0,\varepsilon} (t)\varphi\left(R+c(t)+\frac{\varepsilon}{2}t, t\right) \left(\dot{c}(t)+\frac{\varepsilon}{2}\right)\D{t}\\
+\int_{0}^{T_{01}}v_{0,\varepsilon}(t) \varphi\left(R+c(t)-\frac{\varepsilon}{2}t, t\right) \left(\dot{c}(t)-\frac{\varepsilon}{2}\right)\D{t}-\int_{R+c(t)-\frac{\varepsilon}{2}t}^{R+c(t)+\frac{\varepsilon}{2}t}\frac{\partial}{\partial t}v_{0, \varepsilon}(t)\varphi(x,t)\D{x}\D{t},
\end{align*}
\begin{align*}
I_{03}&=\int_{R+c(T_{01})+\frac{\varepsilon}{2}T_{01}}^{Y_1+\int_{0}^{T_{01}}f(U_1(\xi))d\xi}v_{1}\varphi(x, T_{01})\D{x} -\int_{R}^{Y_1}v_1\varphi(x, 0)\D{x}
+\int_{0}^{T_{01}}v_{1} \varphi\left(R+c(t)+\frac{\varepsilon}{2}t\right) \left(\dot{c}(t)+\frac{\varepsilon}{2}\right)\D{t}\\
&- \int_{0}^{T_{01}}v_{1} \varphi\left(Y_1+\int_{0}^{t}f(U_1(\xi))d\xi, t\right) f\left(U_1(t)\right)\D{t},
\end{align*}
\begin{align*}
&I_{04}=\sum_{i=1}^{\infty}\int_{Y_{i}+\int_{0}^{T_{01}}f(U_{i}(\xi))d\xi}^{Y_{i+1}+\int_{0}^{T_{01}}f(U_{i+1}(\xi))d\xi}v_{i+1}\varphi(x, T_{01})\D{x} -\sum_{i=1}^{\infty}\int_{Y_i}^{Y_{i+1}}v_{i+1}\varphi(x, 0)\D{x}\\
&-\sum_{i=1}^{\infty}\int_{0}^{T_{01}}v_{i+1} \varphi\left(Y_{i+1}+\int_{0}^{t}f(U_{i+1}(\xi))d\xi,t\right) f\left(U_{i+1}(t)\right)\D{t}\\
&+ \sum_{i=1}^{\infty}\int_{0}^{T_{01}}v_{i+1} \varphi\left(Y_i+\int_{0}^{t}f(U_{i+1}(\xi))d\xi,t\right) f\left(U_{i+1}(t)\right)\D{t}.
\end{align*}
Finally adding up $I_{01}, I_{02}, I_{03}$ and $ I_{04},$ we obtain
\begin{align*}
I_0=&\int_{0}^{T_{01}}(v_{0,\varepsilon}(t)-v_0)\varphi\left(R+c(t)-\frac{\varepsilon}{2}t, t\right)\left(\dot{c}(t)-\frac{\varepsilon}{2}\right)\D{t}-\int_{R+c(t)-\frac{\varepsilon}{2}t}^{R+c(t)+\frac{\varepsilon}{2}t}\frac{\partial}{\partial t}v_{0, \varepsilon}(t)\varphi(x,t)\D{x}\D{t}\\
&+\int_{0}^{T_{01}}(v_1-v_{0,\varepsilon}(t))\varphi\left(R+c(t)+\frac{\varepsilon}{2}t, t\right)\left(\dot{c}(t)+\frac{\varepsilon}{2}\right)\D{t}\nonumber
+\int_{-\infty}^{\infty}v^{\varepsilon}\varphi(x, T_{01})\D{x}\\
&-\int_{-\infty}^{\infty}v^{\varepsilon}\varphi(x, 0)\D{x}-\int_{0}^{T_{01}}v_{1}  \varphi\left(Y_1+\int_{0}^{t}f(U_1(\xi))d\xi,t\right)f\left(U_1(t)\right)\D{t}\nonumber\\
&-\sum_{i=1}^{\infty}\int_{0}^{T_{01}}v_{i+1} \varphi\left(Y_{i+1}+\int_{0}^{t}f(U_{i+1}(\xi))d\xi,t\right) f\left(U_{i+1}(t)\right)\D{t}\nonumber\\
&+ \sum_{i=1}^{\infty}\int_{0}^{T_{01}}v_{i+1} \varphi\left(Y_i+\int_{0}^{t}f(U_{i+1}(\xi))d\xi,t\right) f\left(U_{i+1}(t)\right)\D{t}\nonumber.
\end{align*}
A similar calculation for the flux term in $[0,T_{01}]$ gives
\begin{align*}
    -\left\langle \frac{\partial}{\partial x} v^{\varepsilon}f(u^{\varepsilon}), \varphi\right\rangle=&\int_{0}^{T_{01}} \left(v_0 f(U_0(t))-v_{0, \varepsilon}(t)f(u_{0, \varepsilon}(t))\right)\varphi\left(R+c(t)-\frac{\varepsilon}{2}t, t\right)\D{t}\\
&+\int_{0}^{T_{01}} \left(v_{0, \varepsilon}(t)f(u_{0, \varepsilon}(t))-v_1 f(U_1(t))\right)\varphi\left(R+c(t)+\frac{\varepsilon}{2}t, t\right)\D{t}\nonumber\\
&+ \int_{0}^{T_{01}}v_{1} f(U_{1} (t))\varphi\left(Y_1+\int_{0}^{t}f(U_1(\xi))d\xi, t\right) \D{t}\nonumber\\
&+\sum_{i=1}^{\infty}\int_{0}^{T_{01}} v_{i+1} f(U_{i+1} (t))\varphi\left(Y_{i+1}+\int_{0}^{t}f(U_{i+1}(\xi))d\xi, t\right) \D{t}\nonumber\\
&-\sum_{i=1}^{\infty}\int_{0}^{T_{01}} v_{i+1} f(U_{i+1} (t))\varphi\left(Y_{i}+\int_{0}^{t}f(U_{i+1}(\xi))d\xi, t\right)\D{t}\nonumber
\end{align*}
and hence we have
\begin{align}
   -\left\langle\frac{\partial}{\partial t} v^{\varepsilon}, \varphi\right\rangle -\left\langle \frac{\partial}{\partial x} v^{\varepsilon}f(u^{\varepsilon}), \varphi\right\rangle=&\int_{0}^{T_{01}}(v_{0,\varepsilon}(t)-v_0)\varphi\left(R+c(t)-\frac{\varepsilon}{2}t, t\right) \left(\dot{c}(t)-\frac{\varepsilon}{2}\right)\D{t}\nonumber\\
   &+\int_{0}^{T_{01}}(v_1-v_{0,\varepsilon}(t))\varphi\left(R+c(t)+\frac{\varepsilon}{2}t, t\right) \left(\dot{c}(t)+\frac{\varepsilon}{2}\right)\D{t}\nonumber\\
   &-\int_{R+c(t)-\frac{\varepsilon}{2}t}^{R+c(t)+\frac{\varepsilon}{2}t}\frac{\partial}{\partial t}v_{0, \varepsilon}(t)\varphi(x,t)\D{x}\D{t}\nonumber\\
&+\int_{0}^{T_{01}} \left(v_0 f(U_0(t))-v_{0, \varepsilon}(t)f(u_{0, \varepsilon}(t))\right)\varphi\left(R+c(t)-\frac{\varepsilon}{2}t, t\right)\D{t}\nonumber\\
&+\int_{0}^{T_{01}} \left(v_{0, \varepsilon}(t)f(u_{0, \varepsilon}(t))-v_1 f(U_1(t))\right)\varphi\left(R+c(t)+\frac{\varepsilon}{2}t, t\right)\D{t}.\label{eqq4.4}
\end{align}
Now using the Taylor series expansion \eqref{taylor expansion} in \eqref{eqq4.4} and following the calculations of \cref{sec2} (cf. \cref{RPL}), we find
\begin{align}\label{eqq4.5}
    -\left\langle\frac{\partial}{\partial t} v^{\varepsilon}, \varphi\right\rangle -\left\langle \frac{\partial}{\partial x} v^{\varepsilon}f(u^{\varepsilon}), \varphi\right\rangle=\mathcal {O}(\varepsilon)\quad \text{in}\quad [0,T_{01}].
\end{align}
Next by replacing $v^{\varepsilon}$ with $v^{\varepsilon}u^{\varepsilon}$ in \eqref{eqq4.5} and performing calculations similar to the ones described above, \cref{RPL} gives
\begin{align}\label{eqq4.6}
    -\left\langle\frac{\partial}{\partial t} v^{\varepsilon}u^{\varepsilon}, \varphi\right\rangle -\left\langle \frac{\partial}{\partial x} v^{\varepsilon}u^{\varepsilon}f(u^{\varepsilon}), \varphi\right\rangle+\left\langle\kappa(t)\left(u_{a}(t)-u^{\varepsilon}\right)v^{\varepsilon}\right\rangle=\mathcal {O}(\varepsilon)\quad \text{in}\quad [0,T_{01}].
\end{align}
Now we consider the interval $[T_{0i}, T_{1i}].$ Note that at $T_{0i}$-th level for each $i\in \{1,2, \cdot\cdot\cdots\},$ we have initial data that contains delta function as follows
\begin{align*}
    (v, u)(x, T_{0i})=\begin{cases}
        (v_0, u_0),\,\,\,\,\, &\text{if}\,\,\,\, x<X_{i,i},\\
        (\xi(T_{0i})\delta_{T_{0i}}, \chi(T_{0i})),\,\,\,\,\, &\text{if}\,\,\,\, x=X_{i,i},\\
        (0, u_j), \,\,\,\,\, &\text{if}\,\,\,\, X_{j,j}<x<X_{j,j+1},\\
        (v_{j+1}, u_{j+1}),\,\,\,\,\, &\text{if}\,\,\,\, X_{j,j+1}<x<X_{j+1,j+1},
    \end{cases}
\end{align*}
where 

\vspace{.2cm}
\begin{align*}
    &X_{j, j}=Y_j+\int_0^{T_{0i}} f(U_j(\xi))d\xi,~~
    X_{j, j+1}=Y_j+\int_0^{T_{0i}} f(U_{j+1}(\xi))\D{\xi},\\
    &X_{j+1, j+1}=Y_{j+1}+\int_0^{T_{0i}} f(U_{j+1}(\xi))\D{\xi} \,\,\,\, \text{for}\,\,\,\, j=i, i+1, i+2\cdot\cdot\cdots.
\end{align*}
 Therefore we use the shadow wave \eqref{sdw}, i.e.,

 \vspace{.3cm}
\begin{align}\label{sdw4.2}
U^{\varepsilon}=(v^{\varepsilon}, u^{\varepsilon})(x,t)=
\begin{cases}
\left(v_0, U_0(t)\right),\,\,\,\,\,\,\,\, &x< X_{i,i}+c(t-T_{0i})-\frac{\varepsilon}{2}(t-T_{0i})-x^{T_{0i}}_{\varepsilon},\\
\left(v_{T_{0i},\varepsilon}(t), u_{T_{0i},\varepsilon}(t)\right),\,\,\,\,\,\, &X_{i,i}+c(t-T_{0i})-\frac{\varepsilon}{2}(t-T_{0i})-x^{T_{0i}}_{\varepsilon}\\
&<x<X_{i,i}+c(t-T_{0i})+\frac{\varepsilon}{2}(t-T_{0i})+x^{T_{0i}}_{\varepsilon},\\
\left(0, U_i(t)\right),\,\,\,\,\,\,\,\,\,\,\, &X_{ii}+c(t-T_{0i})+\frac{\varepsilon}{2}(t-T_{0i})+x^{T_{0i}}_{\varepsilon}\\
&<x< Y_i+\int_0^t f(U_{i+1}(\xi))\D{\xi},\\
(v_{j+1}, U_{j+1}(t)),\,\,\,\, &Y_j+\int_0^t f(U_{j+1}(\xi))\D{\xi}<x<Y_{j+1}+\int_0^t f(U_{j+1}(\xi))\D{\xi}, 
\end{cases}
\end{align}
for $j=i, i+1, i+2, \cdots \cdot \cdots.$
In the intervals $[T_{0i}, T_{1i}]$, we have
\begin{align*}
         -\left\langle\frac{\partial}{\partial t} v^{\varepsilon}, \varphi\right\rangle&=\int_{T_{0i}}^{T_{1i}} \int_{-\infty}^{\infty}v^{\varepsilon} \frac{\partial}{\partial t}\varphi (x,t)\D{x}\D{t}-\int_{-\infty}^{\infty}\left[(v^{\varepsilon}\varphi)(x, T_{1i}-)-(v^{\varepsilon}\varphi)(x,T_{0i}+)\right]\D{x}\\
         &=I_1-\int_{-\infty}^{\infty}\left[(v^{\varepsilon}\varphi)(x, T_{1i}-)-(v^{\varepsilon}\varphi)(x,T_{0i}+)\right]\D{x}.
     \end{align*}
Substituting the shadow wave \eqref{sdw4.2} in the above equation, we get
\begin{align}
    I_1&=\int_{T_{0i}}^{T_{1i}}\int_{-\infty}^{X_{i,i}+c(t-T_{0i})-\frac{\varepsilon}{2}(t-T_{0i})-x^{T_{0i}}_{\varepsilon}}v_0\frac{\partial}{\partial t}\varphi(x, t)\D{x}\D{t}\nonumber\\
&+\int_{T_{0i}}^{T_{1i}}\int_{{X_{i,i}+c(t-T_{0i})-\frac{\varepsilon}{2}(t-T_{0i})-x^{T_{0i}}_{\varepsilon}}}^{{X_{i,i}+c(t-T_{0i})+\frac{\varepsilon}{2}(t-T_{0i})+x^{T_{0i}}_{\varepsilon}}}v_{T_{0i}, \varepsilon}(t)\frac{\partial}{\partial t}\varphi(x, t)\D{x}\D{t}\nonumber\\
&+\sum_{j=i}^{\infty}\int_{T_{0i}}^{T_{1i}}\int_{Y_j+\int_{0}^{t}f(U_{j+1}(\xi))d\xi}^{Y_{j+1}+\int_{0}^{t}f(U_{j+1}(\xi))d\xi}v_{j+1}\frac{\partial}{\partial t}\varphi(x, t)\D{x}\D{t} =I_{11}+I_{12}+I_{13}\nonumber.
\end{align}
Simplifying the above, we have
\begin{align*}
    I_{11}&=-\int_{T_{0i}}^{T_{1i}}v_0 \varphi\left(X_{i,i}+c(t-T_{0i})-\frac{\varepsilon}{2}(t-T_{0i})-x^{T_{0i}}_{\varepsilon},t\right) \left(\dot{c}(t-T_{0i})-\frac{\varepsilon}{2}\right)\D{t}\\
    &+\int_{-\infty}^{X_{i,i}+c(T_{1i}-T_{0i})-\frac{\varepsilon}{2}(T_{1i}-T_{0i})-x^{T_{0i}}_{\varepsilon}}v_0 \varphi(x, T_{1i}) \D{x}-\int_{-\infty}^{X_{i,i}-x^{T_{0i}}_{\varepsilon}}v_0\varphi(x, T_{0i})\D{x},
\end{align*}
\begin{align*}
    I_{12}&=\int_{X_{i,i}+c(T_{1i}-T_{0i})-\frac{\varepsilon}{2}(T_{1i}-T_{0i})-x^{T_{0i}}_{\varepsilon}}^{X_{i,i}+c(T_{1i}-T_{0i})+\frac{\varepsilon}{2}(T_{1i}-T_{0i})+x^{T_{0i}}_{\varepsilon}}v_{T_{0i}, \varepsilon}(T_{1i})\varphi(x,T_{1i})\D{x}-\int_{X_{i,i}-x^{T_{0i}}_{\varepsilon}}^{X_{i,i}+x^{T_{0i}}_{\varepsilon}}v_{T_{0i}, \varepsilon}(T_{0i})\varphi(x,T_{0i})\D{x}\\
    &+\int_{T_{0i}}^{T_{1i}} v_{T_{0i}, \varepsilon}(t)\varphi \left(X_{i,i}+c(t-T_{0i})-\frac{\varepsilon}{2}(t-T_{0i})-x^{T_{0i}}_{\varepsilon},t\right)\left(\dot{c}(t-T_{0i})-\frac{\varepsilon}{2}\right)\D{t}\\
    &-\int_{T_{0i}}^{T_{1i}} v_{T_{0i}, \varepsilon}(t)\varphi \left(X_{i,i}+c(t-T_{0i})+\frac{\varepsilon}{2}(t-T_{0i})+x^{T_{0i}}_{\varepsilon},t\right)\left(\dot{c}(t-T_{0i})+\frac{\varepsilon}{2}\right)\D{t}\\
    &-\int_{T_{0i}}^{T_{1i}}\int_{X_{i,i}+c(t-T_{0i})-\frac{\varepsilon}{2}(t-T_{0i})-x^{T_{0i}}_{\varepsilon}}^{X_{i,i}+c(t-T_{0i})+\frac{\varepsilon}{2}(t-T_{0i})+x^{T_{0i}}_{\varepsilon}}\frac{\partial}{\partial t} v_{T_{0i}, \varepsilon}(t)\varphi(x,t)\D{x}\D{t},
    \end{align*}
    and 
    \begin{align*}
    I_{13}&=\sum_{j=i}^{\infty}\int_{Y_j+\int_0^{T_{1i}}f(U_{j+1}(\xi))\D{\xi}}^{Y_{j+1}+\int_0^{T_{1i}}f(U_{j+1}(\xi))\D{\xi}}v_{j+1}\varphi(x,T_{1i})\D{x}\D{t}-\sum_{j=i}^{\infty}\int_{Y_j+\int_0^{T_{0i}}f(U_{j+1}(\xi))\D{\xi}}^{Y_{j+1}+\int_0^{T_{0i}}f(U_{j+1}(\xi))\D{\xi}}v_{j+1}\varphi(x,T_{0i})\D{x}\D{t}\\
    &+\sum_{j=i}^{\infty}\int_{T_{0i}}^{T_{1i}}v_{j+1}f(U_{j+1}(t))\left[\varphi\left(Y_j+\int_0^t f(U_{j+1}(\xi))\D{\xi},t\right)-\varphi\left(Y_{j+1}+\int_0^t f(U_{j+1}(\xi))\D{\xi},t\right)\right]\D{t}.
    \end{align*}

    \vspace{.2cm}
Summing up $I_{11}, I_{12}, I_{13}$ and inserting the expression of $I_1$ in above, we obtain
\begin{align*}
   &-\left\langle\frac{\partial}{\partial t} v^{\varepsilon}, \varphi\right\rangle\\
    &=\int_{T_{0i}}^{T_{1i}} \left(v_{T_{0i}, \varepsilon}(t)-v_0\right)\varphi \left(X_{i,i}+c(t-T_{0i})-\frac{\varepsilon}{2}(t-T_{0i})-x^{T_{0i}}_{\varepsilon},t\right)\left(\dot{c}(t-T_{0i})-\frac{\varepsilon}{2}\right)\D{t}\\
    &-\int_{T_{0i}}^{T_{1i}} v_{T_{0i}, \varepsilon}(t)\varphi \left(X_{i,i}+c(t-T_{0i})+\frac{\varepsilon}{2}(t-T_{0i})+x^{T_{0i}}_{\varepsilon},t\right)\left(\dot{c}(t-T_{0i})+\frac{\varepsilon}{2}\right)\D{t}\\
    &-\int_{T_{0i}}^{T_{1i}}\int_{X_{i,i}+c(t-T_{0i})-\frac{\varepsilon}{2}(t-T_{0i})-x^{T_{0i}}_{\varepsilon}}^{X_{i,i}+c(t-T_{0i})+\frac{\varepsilon}{2}(t-T_{0i})+x^{T_{0i}}_{\varepsilon}}\frac{\partial}{\partial t} v_{T_{0i}, \varepsilon}(t)\varphi(x,t)\D{x}\D{t}\\
    &+\sum_{j=i}^{\infty}\int_{T_{0i}}^{T_{1i}}v_{j+1}f(U_{j+1}(t))\left[\varphi\left(Y_j+\int_0^t f(U_{j+1}(\xi))\D{\xi},t\right)-\varphi\left(Y_{j+1}+\int_0^t f(U_{j+1}(\xi))\D{\xi},t\right)\right]\D{t}.
\end{align*}

\vspace{.3cm}
Again,
\begin{align*}
    &-\left\langle \frac{\partial}{\partial x} v^{\varepsilon}f(u^{\varepsilon}), \varphi\right\rangle\\
    &=\int_{T_{0i}}^{T_{1i}}\left(v_0 f(U_0(t))-v_{T_{0i}, \varepsilon}(t)f(u_{T_{0i}, \varepsilon}(t))\right)\varphi\left(X_{i,i}+c(t-T_{0i})-\frac{\varepsilon}{2}(t-T_{0i})-x^{T_{0i}}_{\varepsilon},t\right)\D{t}\\
    &+\int_{T_{0i}}^{T_{1i}}v_{T_{0i}, \varepsilon}(t)f(u_{T_{0i}, \varepsilon}(t))\varphi\left(X_{i,i}+c(t-T_{0i})+\frac{\varepsilon}{2}(t-T_{0i})+x^{T_{0i}}_{\varepsilon},t\right)\D{t}\\
    &+\sum_{j=i}^{\infty}\int_{T_{0i}}^{T_{1i}}v_{j+1}f(U_{j+1}(t))\left[\varphi\left(Y_{j+1}+\int_0^t f(U_{j+1}(\xi))\D{\xi},t\right)-\varphi\left(Y_{j}+\int_0^t f(U_{j+1}(\xi))\D{\xi},t\right)\right]\D{t}.
\end{align*}
Hence, we have 
\begin{align*}\label{eqq4.5}
    -\left\langle\frac{\partial}{\partial t} v^{\varepsilon}, \varphi\right\rangle -\left\langle \frac{\partial}{\partial x} v^{\varepsilon}f(u^{\varepsilon}), \varphi\right\rangle=\mathcal {O}(\varepsilon)\quad \text{in}\quad [T_{0i},T_{1i}]
\end{align*}
follows from \cref{IL},
and replacing $v^{\varepsilon}$ by $v^{\varepsilon}u^{\varepsilon}$ we conclude \eqref{eqq4.6} in $[T_{0i}, T_{1i}].$ Finally $[T_{1i}, T_{0 (i+1)}]$ can be treated in the same way by using an appropriate shadow wave.  The interaction procedure terminates at a finite stage due to the compact support of $\varphi$ and the fact $Y_{i+1}-Y_{i} \geq C_1 \varepsilon^{\alpha}$ for $0<\alpha<1.$ It can be seen that we have at most $\frac{C_1 C_{\varphi}}{\varepsilon^{\alpha}}$ number of interactions where $C_{\varphi}$ denotes the constant depending on $\varphi.$ Thus we conclude \eqref{E1} as both the equations in \eqref{eqq4.5} and \eqref{eqq4.6} are of order $\mathcal{O}(\varepsilon^{1-\alpha}).$

\noindent \textbf{Subcase 2.} Let $\tilde{u}>u_0.$ Since $\tilde{u}$ is the limit of $u_i$ as $i \to \infty,$ there exists a large $n_0\in \NN$ such that $u_{{n_0}+1}\geq u_0$ and $u_{n_0}<u_0.$ In this case the shock curve $x=c(t)$ will stay between $CD^0_1$ and $CD^{{n_0}+1}_2$ emanating from $X_{n_0, n_0}$ and the interaction stops after the time level $T_{0 n_0}.$ This completes the proof.

\end{proof}

\begin{lemma}\label{LM4.3}
    Let $u(x) \in L^{\infty}([R, \infty))\cap C([R, \infty))$ be a decreasing function and grant all other assumptions of \cref{LM3.2}. Then \eqref{E1} holds.
\end{lemma}
\begin{proof}
The proof is similar to the arguments of \cref{LM3.2} and hence we give a sketch of the proof by omitting the detailed analysis. We consider two cases.

\noindent \textbf{Case 1.} When $u_0 \leq u(Y_1).$ In this case we have $u_0 \leq u_1$ and $u_i>u_{i+1}$ for $i=1,2,\cdot\cdot\cdots.$ By \cref{RPL}, the solution consists of contact discontinuities and vacuum of the following form
\begin{align*}
(v, u)(x,t)=
\begin{cases}
\left(v_0, U_0(t)\right),\,\,\,\,\,\,\,\,\,\,\,\, &x< R+\int_0^t f(U_0(\xi))\D{\xi},\\
\left(0, w(x,t)\right),\,\,\,\,\,\, &R+\int_0^t f(U_0(\xi))\D{\xi}<x<R+\int_0^t f(U_1(\xi))\D{\xi},\\
\left(v_1, U_1(t)\right),\,\,\,\,\,\,\,\,\,\,\, &x>R+\int_0^t f(U_1(\xi))\D{\xi},
\end{cases}
  \end{align*}
and for each $Y_i,$ $i=1,2,3, \cdot\cdot\cdots$ a unique shadow wave $SDW_{i,i+1}$ emanates. Due to the overcompressibility $CD^1_2:=R+\int_0^1 f(U_1(\xi))\D{\xi}$ meets $SDW_{1,k}$ at a time level $T_{k(k+1)}$ for $k=2,3, \cdot\cdot\cdots.$ At this time level, we have a set of new initial data with delta-function and \eqref{E1} can be concluded by combining the arguments of \cref{IL} and \cref{LM3.2}.

\noindent \textbf{Case 2.} When $u_0 > u(Y_1).$ Since $u(x)$ is decreasing, in this case we have $u_i>u_{i+1}$ for $i=0,1,2, \cdot\cdot\cdots.$ By \cref{RPL} a unique shadow wave solution $SDW_{i, i+1}$ emerges from each $Y_i,$ $i=0,1,2, \cdot\cdot\cdots.$ Consider the initial data
\[(v,u)(x,0)=\begin{cases}(v_i, u_i),\,\,\,\, x< Y_i,\\
(v_{i+1}, u_{i+1}),\,\,\,\, x> Y_i,
\end{cases}\]
and 
\[(v,u)(x,0)=\begin{cases}(v_{i+1}, u_{i+1}),\,\,\,\, x< Y_{i+1},\\
(v_{i+2}, u_{i+2}),\,\,\,\, x> Y_{i+1}.
\end{cases}.\]
Then $SDW_{i,i+1}$ and $SDW_{i+1, i+2}$ emerges from $Y_i$ and $Y_{i+1}$,  respectively with the speed $f(\chi_i(t))$ and $f(\chi_{i+1}(t)).$ Since $SDW_{i, i+1}$ and $SDW_{i+1, i+2}$ are overcompressive, it follows that $f(\chi_{i+1}(t))\leq f(\chi_i(t))$ and therefore these two shadow waves meet at some time say $T_{i( i+1)}.$ It is difficult to determine for which $i$ first interaction happens, since the comparison of speed between any consecutive shadow waves $SDW_{i, i+1}$ and $SDW_{i+1, i+2}$ is not precisely known. Then at this level, we have a new set of initial data that contains a $\delta$-function. Again, since $u(x)$ is decreasing, applying \cref{IL} we conclude \eqref{E1}. The interaction procedure terminates after a finite number of steps because of the same reason given in \cref{LM3.2}.
\end{proof}
Now we are ready to prove the \cref{TH1.1}.
\begin{proof}[Proof of \cref{TH1.1}]
    Let $u(x)$ be a function having finitely many local extremes. For the sake of concreteness, we only consider that $u(x)$ has a local maximum and local minimum at $Y_{M}, Y_{m}\in \{Y_i\},$ $M<m,$ respectively for $i=0,1,2,\cdot\cdot\cdots.$ So, $u(x)$ is increasing on $[R, Y_M],$ decreasing on $(Y_M, Y_m]$ and again increasing on $(Y_m, \infty).$ 
    
    Now if $u_0 \leq u(Y_1)$ then no interactions occur up to the state $Y_{M-1}$ and the solution is given by $CD^{i-1}_1+Vac_{i-1,i}+CD^i_2,$ $i=1,2,3\cdot\cdot\cdots M.$ By \cref{RPL}, there exists a unique shadow wave $SDW_{M, M+1}$ from $Y_M.$ The waves continued to propagate until the first interaction occurs and this case is dealt with in \cref{LM3.2}. Further, a sequence of shadow waves $SDW_{i, i+1},$ $i=M,\cdot\cdot\cdots m-1$ emanates from  each $\{Y_i\}^{m-1}_{i=M}$ and they interact due to the overcompressibility and this case is studied in \cref{LM4.3}. Finally, since $Y_m$ is the local minimum, there exists a solution consisting $CD^m_1+Vac_{m, m+1}+CD^{m+1}_1$ starting from $Y_m.$ Again the solution propagates until the time when $SDW_{m-1,m}$ meets $CD^m_1$ and this interaction occur as $CD^m_1=f(U_m(t))<f(\chi_{m-1}(t))<f(U_{m-1}(t))$ where $f(\chi_{m-1}(t))$ denotes the speed of the shadow wave $SDW_{m-1,m}$ emerging from $Y_{m-1}.$ This case falls under \cref{LM3.2}.

    If $u_0>u(Y_1),$ then $SDW_{0,1}$ emanates from $Y_0=R$ and interacts with $CD^1_1$ which is considered in \cref{LM3.2}. The other cases are also similar and an application of \cref{LM3.2} and \cref{LM4.3} depending on the situation concludes the proof.
\end{proof}
An analogue of \cref{TH1.1} can be proved for the $3\times 3$ system \eqref{s7eq1} following the above steps and using the Riemann problem of \cref{sec2}.
\subsection {Proof of \cref{TH1.2}} The aim of this section is to prove the \cref{TH1.2}. We start with the well-known Riesz' representation theorem.
\begin{theorem}[Riesz' representation theorem]
    Let $I_\mu: C_c(\RR^d)\to \RR$ be a nonnegative linear functional. Then there exists a unique signed Radon measure $\mu$ such that
    \[I_\mu(\varphi):=\int_K \varphi \,\, \D{\mu}\]
    for any compact set $K\subset \RR^d$ with $\supp(\varphi)\subset K.$
\end{theorem}
We denote $\cM(\RR^d)$ as the space of all signed Radon measures endowed with the topology of the dual space of $C_c(\RR^d).$

\begin{definition}[Weak* convergence of measures]
    Let $\{\mu_{\nu}\}_{\nu \in \NN\cup \{0\}}$ be a sequence of nonnegative locally finite measures on $\RR^d.$ The sequence $\{\mu_{\nu}\}_{\nu \in \NN\cup \{0\}}$ is said to be converges weakly* to $\mu$ (denoted as $\mu_{\nu}\stackrel{\ast}{\rightharpoonup} \mu$) if 
    \[\lim_{\nu \to \infty} \int \varphi\,\, \D{\mu_{\nu}}=\int \varphi \,\, \D{\mu}\]
    for every $\varphi \in C_c(\RR^d).$
\end{definition}
\begin{lemma}[\cite{CDL08}, Proposition 2.5]\label{LM4.4}
    Let $\{\mu_{\nu}\}_{\nu \in \NN\cup \{0\}}$ be a sequence of uniformly locally bounded measures. Then there exists a subsequence still denoted as $\{\mu_{\nu}\}$ and a Radon measure $\mu$ such that $\mu_{\nu}\stackrel{\ast}{\rightharpoonup} \mu.$
\end{lemma}
To prove \cref{TH1.2}, we need to show that $|U^\nu|=(v^\nu, |u^\nu|)$ is locally uniformly bounded for each $\nu \in \NN\cup\{0\}.$ This is the content of the next lemma.
 \begin{lemma}\label{LM4.5}
     Suppose $v(x)>0$ is bounded  and $u(x)\in  L^{\infty}([R, \infty)) \cap C([R, \infty))$ and $u(x)$ has finitely many local extremes. Also assume $U^\nu=(v^\nu, u^\nu)$ be the admissible solution to \eqref{system dependent on t} and \eqref{initial data approx}. Then $\{U^\nu\}_{\nu \in \NN \cup\{0\}}$ is locally uniformly bounded for each $\nu.$
 \end{lemma}
 \begin{proof}
     The proof follows from overcompressibility and conservation of mass principle. 
     
\noindent{\bf Step 1.} For $t\in (t_0, T),$ using
     \[\left|\int_0^t F(\xi)\D{\xi}\right| \leq ||u_{a}||_{L^{\infty}([t_0, T])}\left(e^{||\kappa||_{L^{\infty}([t_0, T])}T}-1\right),\] we get
     \[\chi(t)\leq \left|U_{l,r}(t)\right| \leq ||u_{a}||_{L^{\infty}([t_0,T])}\left(e^{||\kappa||_{L^{\infty}([t_0, T])}T}-1\right)+|u_{l,r}|,\] and hence $\dot{c}(t)$ is also locally uniformly bounded. This, combined with the fact that $u(x)$ in \eqref{initial data approx} is bounded gives $|u^\nu|$ is locally uniformly bounded for each $\nu\in \NN \cup \{0\}.$
     
\noindent{\bf Step 2.} To prove $v^\nu \geq 0$ is locally uniformly bounded, observe that for any compact set $K\subset \RR\times [0, \infty)$ we have
\begin{align*}
    0\leq \int_{K} v^{\nu}(x,t)\D{x}\D{t} \leq C_K \max\left\{v_0, \sup_{x>R}v(x)\right\}<\infty.
\end{align*}
Therefore, $|U^\nu|=(v^\nu, |u^\nu|)$ is bounded in $L^1(K)$ for each compact set $K \subset \RR\times [0, \infty).$ This completes the proof.
 \end{proof}
 Now we complete the proof of \cref{TH1.2}. We will use the \textit{Jordan Decomposition} of signed measures. If $\mu$ is a signed measure, then there exist unique positive mutually singular measures $\mu^+$ and $\mu^-$ such that $\mu=\mu^+-\mu^-.$ The measures $\mu^+$ and $\mu^-$ are called positive and negative variations of $\mu$ and $|\mu|=\mu^++\mu^-$ is defined as the total variation of the measure $\mu.$
 \begin{proof}[Proof of \cref{TH1.2}]
     The proof can be completed following the same lines given in \cite[Theorem 7.1]{RN21}. We repeat it here for the sake of completeness. Using \cref{LM4.5} and decomposition of measures, $\{U^\nu_+\}_{\nu \in \NN\cup\{0\}}$ and $\{U^\nu_-\}_{\nu \in \NN\cup\{0\}}$ are locally uniformly bounded where $U^\nu=U^\nu_+-U^\nu_-.$ Thus by \cref{LM4.4}, $U^\nu_+ \stackrel{\ast}{\rightharpoonup} U^*_+$ and $U^\nu_- \stackrel{\ast}{\rightharpoonup} U^*_-$ where $U^*_{\pm}$ are Radon measures. Therefore $U^\nu$ converges weakly* to a Radon measure $U^*=U^*_+-U^*_-.$ Also a direct use of \cref{LM4.4} gives $|U^\nu|$ converges weakly* to $|U^*|.$
 \end{proof}
 \section{Some examples} \label{sec5}
 In this section, we present some examples of physically relevant models which can be derived from \eqref{system dependent on t}. The first obvious example is the pressureless gas dynamics model with flocking dissipation that can be derived by setting $f(u)=u,$ $\kappa(t)=1,$ and $u_a(t)=0.$ We provide some more examples below.
\subsection{System of nonlinear geometric optics}
 The following examples can be constructed by setting $f: U\mapsto \frac{U}{\sqrt{1+U^2}}$ in \eqref{system dependent on t}, i.e., consider the system of the following form
\begin{align*}
\begin{cases}
v_{t}+\left(\frac{vU}{\sqrt{1+U^2}}\right)_{x}=0,\\
(vU)_{t}+\left(\frac{vU^2}{\sqrt{1+U^2}}\right)_{x}=\kappa(t)(u_{a}(t)-U)v,\nonumber
\end{cases}
\end{align*}
with the initial data $(v, U)(x,0)=(v_0(x), U_0(x)).$ One can easily check that $f^{\prime}(U)>0.$

The motivation for considering the above system comes from the system of nonlinear geometric optics with a damping source. We use the transformation $U=\frac{u}{v}$, i.e., $u=vU$ to obtain the following:
\begin{align*}
\begin{cases}
v_{t}+\left(\frac{u^2}{\sqrt{u^2+v^2}}\right)_{x}=0,\\
u_{t}+\left(\frac{uv}{\sqrt{u^2+v^2}}\right)_{x}=\kappa(t)(u_{a}(t)v-u).\nonumber
\end{cases}
\end{align*}
For the homogeneous version of the above example, see \cite{zhangcpaa, yangjde}.
\subsection{A model with $f(u)=u^k,$ $k$ odd.} In their work \cite{richardacta}, Cruz and Juajibioy investigated the Riemann problem associated with the pressureless model featuring linear damping, described by the following equations:
\begin{align*}
\begin{cases}
v_{t}+\left(vu^k\right)_{x}=0,\\
(vu)_{t}+\left(vu^{k+1}\right)_{x}=-\alpha u v,\nonumber
\end{cases}
\end{align*}
where $k$ is an odd natural number, $\alpha>0$ is a constant. This system can immediately be obtained by plugging $f(u)=u^k,$ $\kappa(t)=\alpha,$ and $u_a(t)=0.$ Since $k$ is odd, clearly $f^{\prime}(u)=ku^{k-1}>0.$\\
\subsection{Pressureless hydrodynamic model} Recently, Piccoli \textit{et al.} \cite{piccolizamp} derived a pressureless hydrodynamic model from the second-order macroscopic traffic flow. The model reads
\begin{align*}
\begin{cases}
    v_t+\left(\frac{vu}{a+u}\right)_x=0,\\
    (vu)_t+\left(\frac{vu^2}{a+u}\right)_x=\frac{\mu p v^2(s_d(v)-u)}{2(\nu+1)},
\end{cases}
\end{align*}
where $v$ and $u$ denote the traffic density and  the local mean headway (i.e., the distance from their leading vehicles), respectively. Furthermore, $a>0$ denotes the magnitude of the driver sensitivity, $0<\mu<1$ corresponds to the equilibrium coefficient in order to control the deviation, whereas $\nu>0$ is related to the cost of control. The function $s_d(v)$ is used to express the safe distance from the preceding vehicle and $0\leq p\leq 1$ denotes the percentage of vehicles in a traffic flow. More recently, the Riemann problem for the non-homogeneous version of the above system is studied by Wang and Sun \cite{WSJMP23}. Setting $f(u)=\frac{u}{a+u}$ in \eqref{system dependent on t}, we propose a pressureless hydrodynamic model with a damping source that depends on time. For this $f,$ it is easy to check $f^{\prime}(u)=\frac{a}{(a+u)^2}>0.$

\vspace{.5cm}
\noindent\textbf{Part II: Explicit formula for \eqref{source term depending on x and t}.}
 \section{Eulerian droplet model involving space-time dependent source}\label{sec6}
 In this section, we deal with the case $f(u)=u$ and $u_a(x,t)$ depending on both space and time variables. We employ generalized variational techniques and obtain an explicit representation of the solution to the system \eqref{second system} with air velocity that has an algebraic decay in time. We start by introducing the generalized potential for the initial value problem as follows
 \begin{align}\label{second system} 
v_{t}+(vu)_{x}&=0,\\
(vu)_{t}+(vu^2)_{x}&=\kappa(t)(u_{a}(x,t)-u)v,\nonumber
\end{align}
adjoined with the initial data $(v(x,0), u(x,0))=(v_0(x), u_0(x)).$ We recall $\kappa(t)=\frac{1}{t+\upkappa}$ and $u_a(x,t)=\frac{x}{t+\upkappa}$ for any $\upkappa \in \RR^{+}.$
The characteristics equation for the above system is reduced to
 \begin{align}\label{second characteristic}
\begin{cases}
&\frac{\D{x(t)}}{\D{t}}=u(x(t),t),\\
&\frac{\D{u(x(t),t)}}{\D{t}}=\frac{1}{t+\upkappa}(\frac{x}{t+\upkappa}-u(x(t),t)),\\
&x(0)=x_0.
\end{cases}
\end{align}
From \eqref{second characteristic}, we obtain second-order ODE as follows:
\begin{align}\label{second ODE}
    &\frac{\D^2{x}}{\D{t^2}}+\frac{1}{t+\upkappa}\frac{\D{x}}{\D{t}}-\frac{1}{(t+\upkappa)^2}x=0,\\
    &x(0)=x_0.\nonumber
\end{align}
Solving the above ODE \eqref{second ODE}, we obtain
\begin{align*}
    &x(t)=\left(\frac{x_0}{2\upkappa}+\frac{u_0(x_0)}{2}\right)(t+\upkappa)+\left(\frac{x_0}{2}-\frac{u_0(x_0)\upkappa}{2}\right)\frac{\upkappa}{(t+\upkappa)},\\
    &u(x(t),t)=\left(\frac{x_0}{2\upkappa}+\frac{u_0(x_0)}{2}\right)-\left(\frac{x_0}{2}-\frac{u_0(x_0)\upkappa}{2}\right)\frac{\upkappa}{(t+\upkappa)^2}.
\end{align*}
Now we introduce our generalized potential as the following.
\begin{align}\label{potential}
    F(y,x,t)=\int_0^y\left(\left(\frac{\eta}{2\upkappa}+\frac{u_0(\eta)}{2}\right)(t+\upkappa)+\left(\frac{\eta}{2}-\frac{u_0(\eta)\upkappa}{2}\right)\frac{\upkappa}{(t+\upkappa)}-x\right) v_0(\eta)d\eta.
\end{align}
Given a point $(x,t),$ let $y_*(x,t)$ and $y^*(x,t)$ be the leftmost and the rightmost points on the $x$- axis such that
\begin{align*}
    \min_{y\in \RR} F(y,x,t)= F(y_*(x,t), x,t)=F(y^*(x,t), x,t).
\end{align*}
Below, we give the explicit representation for $m(x,t)$ and $u(x,t).$ The formula for the pair $(m,u)$ is given by
\begin{align}
&m(x,t)=\int_0^{y_*(x,t)} v_0(\eta)\D{\eta},\label{exf1}\\
&u(x,t)=\begin{cases}
&\frac{1+\frac{\upkappa^2}{(t+\upkappa)^2}}{(t+\upkappa)-\frac{\upkappa^2}{(t+\upkappa)}}\left[x-\frac{1}{2}\left(\frac{t+\upkappa}{\upkappa}+\frac{\upkappa}{t+\upkappa}\right)y_*(x,t)\right]+\frac{y_*(x,t)}{2}\left(\frac{1}{\upkappa}-\frac{\upkappa}{(t+\upkappa)^2}\right), \,\,\,\, \text{if}\,\,\,\, y_*(x,t)=y^*(x,t),\\
&\frac{1}{\int_{y_*}^{y^*}v_0(\eta) \D{\eta}}\int_{y_*}^{y^*}\left[\left(\frac{\eta}{2\upkappa}+\frac{u_0(\eta)}{2}\right)-\left(\frac{\eta}{2}-\frac{u_0(\eta)\upkappa}{2}\right)\frac{\upkappa}{(t+\upkappa)^2}\right]v_0(\eta)\D{\eta}, \,\,\,\, \text{if}\,\,\,\, y_*(x,t)\neq y^*(x,t)\label{exf}.
\end{cases}
\end{align}
\subsection{Derivation of the formula \cref{exf}} In this section we derive the above formula. Derivation of the formula \eqref{exf} consists of several steps that are formulated as lemmas.  We Define the left backward characteristics and right backward characteristics for $0\leq  t<t_0$ as
\begin{align}
    X_l(x_0, t_0, t)=\frac{y_*(x_0, t_0)}{2}\left(\frac{t+\upkappa}{\upkappa}+\frac{\upkappa}{t+\upkappa}\right)+\left((t+\upkappa)-\frac{\upkappa^2}{t+\upkappa}\right)\frac{x_0-\frac{y_*(x_0, t_0)}{2}\left(\frac{t_0+\upkappa}{\upkappa}+\frac{\upkappa}{t_0+\upkappa}\right)}{\left((t_0+\upkappa)-\frac{\upkappa^2}{t_0+\upkappa}\right)},\label{LBC}\\
    X_r(x_0, t_0, t)=\frac{y^*(x_0, t_0)}{2}\left(\frac{t+\upkappa}{\upkappa}+\frac{\upkappa}{t+\upkappa}\right)+\left((t+\upkappa)-\frac{\upkappa^2}{t+\upkappa}\right)\frac{x_0-\frac{y^*(x_0, t_0)}{2}\left(\frac{t_0+\upkappa}{\upkappa}+\frac{\upkappa}{t_0+\upkappa}\right)}{\left((t_0+\upkappa)-\frac{\upkappa^2}{t_0+\upkappa}\right)}.\nonumber
\end{align}
Also, let us denote
\begin{align*}
    X^+_{l,r}(x_0, t_0, t):=\left\{(x,t) \big| x \geq X_{l, r}(x_0, t_0, t)\right\},~~
    X^-_{l,r}(x_0, t_0, t):=\left\{(x,t) \big| x \leq X_{l, r}(x_0, t_0, t)\right\}.
\end{align*}
Next, we define the characteristic triangle.
\begin{definition}
The area $\{(x, t)\,\,| \,\, x\in \RR, 0\leq t \leq t_0\} \cap X^+_l(x_0, t_0, t) \cap X^-_r(x_0, t_0, t)$ is said to be a characteristic triangle associated to $(x_0, t_0)$ and it is denoted as $\Delta(x_0, t_0).$
\end{definition}
Next, we collect some basic properties of the functional $F(y,x,t),$ minimizers $y^*(x,t), y_*(x,t)$ and the characteristic triangles which can be proved easily following \cite{haung-wang, Wang-haung, neumann2022}.
\begin{lemma}\label{LM5.1}
Grant the assumptions on the initial data. Then, we have
\begin{enumerate}
   \item For a fixed $t,$ $y_*(x,t)$ and $y^*(x,t)$ are monotonically increasing in $x$ and for $x_1< x_2,$ $y^*(x_1, t) \leq y_*(x_2, t).$ Furthermore, $y_*(x,t)$ is lower semicontinuous and $y^*(x,t)$ is upper semicontinuous.
   \item For fixed $(x,t)$, let the $\min\limits_{y\in [0,\infty)}F(y,x,t)$ be 
attained at $y_1$. Then for any given point $(x^{\prime},t^{\prime})$ on the curve joining $(y_1,0)$ and $(x,t)$, we have $F(y,x^{\prime},t^{\prime}) > F(y_1,x^{\prime},t^{\prime})$ for $y \neq y_1.$
\item The function $F: \RR \times [0, \infty)\to \RR,$ i.e.,$ (x,t)\mapsto F(x,t)$ is locally Lipschitz continuous.
\item Let $t > 0$ be fixed, and  $x_1 \neq  x_2$ be arbitrary. Then the characteristic triangles
associated with $(x_1,t)$ and $(x_2,t)$ do not intersect in $\RR \times [0, \infty).$ Consequently, if two characteristic triangles intersect in $\RR \times [0, \infty)$, then one is contained in the other.
\item For any time $t_0>0$, we have
\begin{align*}
    \bigcup_{x \in \RR}\Delta(x,t_0)=\left\{(x,t)\big|x \in \RR,0 \leq t \leq t_0\right\}.
\end{align*}
\end{enumerate}
\end{lemma}
Now we are ready to derive the formula \eqref{exf}.
\begin{lemma}\label{lem6.2}
Let $t_1>0$. Each point $(x_1,t_1)$ uniquely determines a Lipschitz continuous curve $x = X(t)$, for $t\geq t_1$ with $x_1 = X(t_1)$ such that the characteristic triangles associated with points on the curve form an increasing family of sets. At every $t \geq t_1$, $u(x,t)$ is defined as the slope of the curve $X^{\prime}(t).$
\end{lemma}
\begin{proof}
Let $t<t^{\prime} <t^{\prime \prime}$ and $x^{\prime}=X(t^{\prime}), x^{\prime \prime}=X(t^{\prime \prime}).$ Applying $(1)$ of \cref{LM5.1}, we have
\begin{align*}
y_*(x^{\prime \prime}, t^{\prime \prime}) \leq y_*(x^{\prime}, t^{\prime})\leq y_*(x,t) \leq y^*(x,t) \leq y^*(x^{\prime}, t^{\prime}) \leq y^*(x^{\prime \prime}, t^{\prime \prime})
\end{align*}
and $\{y_*(x^{\prime \prime}, t^{\prime \prime}), y_*(x^{\prime}, t^{\prime})\}$ and $\{y^*(x^{\prime \prime}, t^{\prime \prime}), y^*(x^{\prime}, t^{\prime})\}$ tend to $y_*(x,t)$ and $y^*(x,t)$, respectively as $t^{\prime\prime}, t^{\prime} \to t.$ Now, we consider the following two cases.\\
\noindent \textbf{Case I.} Let $y_*(x,t)=y^*(x,t).$ Take any two points $(x_1, t^{\prime})$ and $(x_2, t^{\prime})$ on the backward characteristics $X_l(x^{\prime \prime}, t^{\prime\prime}, t)$ and $X_r(x^{\prime\prime}, t^{\prime\prime}, t)$, respectively. Then we have
\begin{align*}
    \frac {x^{\prime\prime}-x_2(t^{\prime})}{t^{\prime \prime}-t^{\prime}} \leq \frac{x^{\prime\prime}-x^{\prime}}{t^{\prime\prime}-t^{\prime}}\leq \frac {x^{\prime\prime}-x_1(t^{\prime})}{t^{\prime \prime}-t^{\prime}}.
\end{align*}
Using \eqref{LBC} and simplifying, we obtain
\begin{align}\label{eqq5.8}
\frac {x^{\prime\prime}-x_1(t^{\prime})}{t^{\prime \prime}-t^{\prime}}=&\frac{x^{\prime\prime}}{t^{\prime\prime}-t^{\prime}}\left[1-\frac{(t^{\prime}+\upkappa)-\frac{\upkappa^2}{t^{\prime}+\upkappa}}{(t^{\prime\prime}+\upkappa)-\frac{\upkappa^2}{t^{\prime\prime}+\upkappa}}\right]\nonumber\\
&-\frac{y_*(x^{\prime\prime}, t^{\prime\prime})}{2(t^{\prime\prime}-t^{\prime})}\left[\left(\frac{t^{\prime}+\upkappa}{\upkappa}+\frac{\upkappa}{t^{\prime}+\upkappa}\right)-\frac{\left(\frac{t^{\prime\prime}+\upkappa}{\upkappa}+\frac{\upkappa}{t^{\prime\prime}+\upkappa}\right)\left((t^{\prime}+\upkappa)-\frac{\upkappa^2}{t^{\prime}+\upkappa}\right)}{(t^{\prime\prime}+\upkappa)-\frac{\upkappa^2}{t^{\prime\prime}+\upkappa}}\right].
\end{align}
Similarly, $\frac {x^{\prime\prime}-x_2(t^{\prime})}{t^{\prime \prime}-t^{\prime}}$ can be written in the above form by replacing $y_*(x^{\prime\prime}, t^{\prime\prime})$ by $y^*(x^{\prime\prime}, t^{\prime\prime})$ in \eqref{eqq5.8}. Now passing to the limit as $t^{\prime\prime}, ~t^{\prime} \to t$ in \eqref{eqq5.8}, we get
\begin{align*}
   X^{\prime}(t) \leq x\left[\frac{\left(1+\frac{\upkappa^2}{(t+\upkappa)^2}\right)}{(t+\upkappa)-\frac{\upkappa^2}{t+\upkappa}}\right]-\frac{y_*(x,t)}{2}\left[\frac{\left(\frac{t+\upkappa}{\upkappa}+\frac{\upkappa}{t+\upkappa}\right)}{\left((t+\upkappa)-\frac{\upkappa^2}{t+\upkappa}\right)} \left(1+\frac{\upkappa^2}{(t+\upkappa)^2}\right)-\left(\frac{1}{\upkappa}-\frac{\upkappa}{(t+\upkappa)^2}\right)\right].
\end{align*}
Rearranging the above expression we get \eqref{exf} for the case $y_*(x,t)=y^*(x,t).$
\\\noindent \textbf{Case II.} Let $y_*(x,t)\neq y^*(x,t).$ First we note that
\begin{align}\label{eqq5.9}
    F(y^*(x^{\prime\prime}, t^{\prime\prime}), x^{\prime\prime}, t^{\prime\prime})-F(y_*(x^{\prime},t^{\prime}), x^{\prime\prime}, t^{\prime\prime})\leq F(y^*(x^{\prime\prime}, t^{\prime\prime}), x^{\prime}, t^{\prime})-F(y_*(x^{\prime}, t^{\prime}), x^{\prime}, t^{\prime}).
\end{align}
Inserting the potential \eqref{potential} in \eqref{eqq5.9} and simplifying, we get
\begin{align}\label{eqq5.11}
   &\frac{1}{t^{\prime\prime}-t^{\prime}} \int_{y_*(x^{\prime}, t^{\prime})}^{y^*(x^{\prime\prime}, t^{\prime\prime})}\left[\left(\frac{\eta}{2\upkappa}+\frac{u_0(\eta)}{2}\right)(t^{\prime\prime}-t^{\prime})+\left(\frac{\eta}{2}-\frac{u_0(\eta)\upkappa}{2}\right)\left(\frac{\upkappa}{(t^{\prime\prime}+\upkappa)}-\frac{\upkappa}{t^{\prime}+\upkappa}\right)\right]v_0(\eta)\D{\eta}\nonumber\\
   &\leq \frac{x^{\prime\prime}-x^{\prime}}{t^{\prime\prime}-t^{\prime}}\int_{y_*(x^{\prime}, t^{\prime})}^{y^*(x^{\prime\prime}, t^{\prime\prime})}v_0(\eta)\D{\eta}.
\end{align}
Passing to the limit as $t^{\prime\prime}, t^{\prime} \to t$ in \eqref{eqq5.11} we obtain 
\begin{align}\label{eqq5.12}
    X^{\prime}(t)\geq \frac{1}{\int_{y_*}^{y^*}v_0(\eta) \D{\eta}}\int_{y_*}^{y^*}\left[\left(\frac{\eta}{2\upkappa}+\frac{u_0(\eta)}{2}\right)-\left(\frac{\eta}{2}-\frac{u_0(\eta)\upkappa}{2}\right)\frac{\upkappa}{(t+\upkappa)^2}\right]v_0(\eta)\D{\eta}.
\end{align}
Again, considering the inequality 
\begin{align*}
    F(y^*(x^{\prime}, t^{\prime}), x^{\prime}, t^{\prime})-F(y_*(x^{\prime\prime},t^{\prime\prime}), x^{\prime}, t^{\prime})\leq F(y^*(x^{\prime}, t^{\prime}), x^{\prime\prime}, t^{\prime\prime})-F(y_*(x^{\prime\prime}, t^{\prime\prime}), x^{\prime\prime}, t^{\prime\prime}),
\end{align*}
we get
\begin{align}\label{eqq5.13}
   X^{\prime}(t)\leq \frac{1}{\int_{y_*}^{y^*}v_0(\eta) \D{\eta}}\int_{y_*}^{y^*}\left[\left(\frac{\eta}{2\upkappa}+\frac{u_0(\eta)}{2}\right)-\left(\frac{\eta}{2}-\frac{u_0(\eta)\upkappa}{2}\right)\frac{\upkappa}{(t+\upkappa)^2}\right]v_0(\eta)\D{\eta}. 
\end{align}
Combining \eqref{eqq5.12}-\eqref{eqq5.13}, we obtain \eqref{exf}.
\end{proof}
The next result shows that the curve $X(t)$ can actually be started from ${t=0}.$ The proof can be completed using the arguments  of \cite{haung-wang, Wang-haung} and also see \cite{neumann2022} for a more general case extended to the initial-boundary value problem.
\begin{theorem}
    Let $X(\eta, t)$ be a curve defined in \cref{lem6.2} and $S$ be a countable set of points on the $x$-axis. Then for all $(\eta, 0)\notin S$ there exists a unique Lipschitz continuous curve $x=X(\eta,t), t\geq 0$ such that $X(\eta, 0)=\eta$ and the characteristics triangles associated to the points form an increasing family of sets. Also, for all $(\eta, 0)\notin S$, we have
    \begin{align*}
        \frac{\partial}{\partial t} X(\eta, t)=u(X(\eta, t), t) \,\,\,\, \text{for a.e}\,\,\, t>0.
    \end{align*}
\end{theorem}

\subsection{Verification of weak formulation, entropy criterion, and initial condition} The goal of this section is to prove \cref{TH1.3}. We verify that the pair $(m,u)$ satisfies the weak formulation, Lax entropy criterion, and the initial condition.
\subsubsection{Verification of weak formulation \cref{weak formulation}-\eqref{wf2}} Verification of weak formulation \eqref{weak formulation}-\eqref{wf2} consists of several steps. First, let us define the momentum and energy potentials as follows:
\begin{align}
    &q(x,t)=\int_0^{y_*(x,t)}\left[\left(\frac{\eta}{2\upkappa}+\frac{u_0(\eta)}{2}\right)-\left(\frac{\eta}{2}-\frac{u_0(\eta)\upkappa}{2}\right)\frac{\upkappa}{(t+\upkappa)^2}\right]v_0(\eta)\D{\eta},\nonumber\\
    &E(x,t)=\frac{1}{2}\int_0^{y_*(x,t)}\left[\left(\frac{\eta}{2\upkappa}+\frac{u_0(\eta)}{2}\right)-\left(\frac{\eta}{2}-\frac{u_0(\eta)\upkappa}{2}\right)\frac{\upkappa}{(t+\upkappa)^2}\right]v_0(\eta)u(X(\eta,t),t)\D{\eta},\nonumber
    \end{align}
    and the functionals
    \begin{align}
    &H(y, x,t)=\int_0^{y}\left[\left(\frac{\eta}{2\upkappa}+\frac{u_0(\eta)}{2}\right)-\left(\frac{\eta}{2}-\frac{u_0(\eta)\upkappa}{2}\right)\frac{\upkappa}{(t+\upkappa)^2}\right]v_0(\eta)\left(X(\eta,t)-x\right)\D{\eta},\nonumber\\
    &I(y,x,t)=\int_0^y \left(\frac{\eta}{2}-\frac{u_0(\eta)\upkappa}{2}\right)\frac{2\upkappa v_0(\eta)}{(t+\upkappa)^3}\left(X(\eta,t)-x\right)\D{\eta},\nonumber\\
    &J(y,x,t)=\int_0^y\left(\frac{\eta}{2}-\frac{u_0(\eta)\upkappa}{2}\right)\frac{2\upkappa v_0(\eta)}{(t+\upkappa)^3}\D{\eta}.\nonumber
\end{align}

\vspace{.5cm}
\noindent{\bf Step 1.} The following relations hold:
\begin{align}\label{mass relation}
    (i)\,\,\,\D{q}=u\D{m}, \,\,\,\, (ii)\,\,\,\D{E}=\frac{1}{2}u^2 \D{m}.
\end{align}
in the sense of Radon-Nikodym derivatives in $x.$\\
\noindent{\it Proof of Step 1.} If $y_*(x,t)$ is a constant in a neighbourhood of $(x,t),$ then \eqref{mass relation} holds trivially. Assume that $y_*(x,t)$ is not constant in a neighbourhood of $(x,t)$ and $y_*(x,t)=y^*(x,t).$ For a fixed $t>0,$ let $x_1<x<x_2.$ By using the definition, we have
\begin{align*}
    F(y_*(x_1,t), x_1, t) \leq F(y_*(x_2,t), x_1, t).
\end{align*}
Then, we have
\begin{align*}
    \int_{y_*(x_1,t)}^{y_*(x_2,t)} \frac{u_0(\eta)}{2} \left((t+\upkappa)-\frac{\upkappa^2}{(t+\upkappa)}\right)v_0(\eta)\D{\eta}
    &\leq \int_{y_*(x_1,t)}^{y_*(x_2,t)} \left(x_1-\frac{\eta}{2}\left(\frac{t+\upkappa}{\upkappa}+\frac{\upkappa}{t+\upkappa}\right)\right)v_0(\eta)\D{\eta}\\
    &\leq\int_{y_*(x_1,t)}^{y_*(x_2,t)} \left(x_1-\frac{y_*(x_1,t)}{2}\left(\frac{t+\upkappa}{\upkappa}+\frac{\upkappa}{t+\upkappa}\right)\right)v_0(\eta)\D{\eta}.
\end{align*}
This implies
\begin{align*}
    &\int_{y_*(x_1,t)}^{y^*(x_2,t)}\left[\left(\frac{\eta}{2\upkappa}+\frac{u_0(\eta)}{2}\right)-\left(\frac{\eta}{2}-\frac{u_0(\eta)\upkappa}{2}\right)\frac{\upkappa}{(t+\upkappa)^2}\right]v_0(\eta)\D{\eta}\\
    &\leq \frac{1+\frac{\upkappa^2}{(t+\upkappa)^2}}{(t+\upkappa)-\frac{\upkappa^2}{(t+\upkappa)}}\left[x-\frac{1}{2}\left(\frac{t+\upkappa}{\upkappa}+\frac{\upkappa}{t+\upkappa}\right)y_*(x_1,t)\right]+\frac{y_*(x_2,t)}{2}\left(\frac{1}{\upkappa}-\frac{\upkappa}{(t+\upkappa)^2}\right)\int_{y_*(x_1,t)}^{y_*(x_2,t)}v_0(\eta)\D{\eta}.
\end{align*}
Hence, we obtain
\begin{align}\label{eqq5.21}
    \frac{q(x_1,t)-q(x_2,t)}{m(x_1,t)-m(x_2,t)} \leq \frac{1+\frac{\upkappa^2}{(t+\upkappa)^2}}{(t+\upkappa)-\frac{\upkappa^2}{(t+\upkappa)}}\left[x-\frac{1}{2}\left(\frac{t+\upkappa}{\upkappa}+\frac{\upkappa}{t+\upkappa}\right)y_*(x_1,t)\right]+\frac{y_*(x_2,t)}{2}\left(\frac{1}{\upkappa}-\frac{\upkappa}{(t+\upkappa)^2}\right).
\end{align}
Now passing to the limit as $x_1\nearrow x$ and $x_2\searrow x$ in \eqref{eqq5.21} we get
\begin{align*}
    \lim_{x_1, x_2 \to x}\frac{q(x_1,t)-q(x_2,t)}{m(x_1,t)-m(x_2,t)} \leq u(x,t).  
\end{align*}
Similarly, considering the inequality
\begin{align*}
    F(y_*(x_2,t), x_2, t) \leq F(y_*(x_1,t), x_2, t),
\end{align*}
and following the same argument as above one can easily obtain the other way inequality
\begin{align*}
    \lim_{x_1, x_2 \to x}\frac{q(x_1,t)-q(x_2,t)}{m(x_1,t)-m(x_2,t)} \geq u(x,t),
\end{align*}
and this proves $(i).$\\
If $y_*(x,t)<y^*(x,t),$ then
\begin{align*}
    \lim_{x_1, x_2 \to x}\frac{q(x_1,t)-q(x_2,t)}{m(x_1,t)-m(x_2,t)}&=\lim_{x_1, x_2 \to x}\frac{\int_{y_*(x_1,t)}^{y_*(x_2,t)}\left[\left(\frac{\eta}{2\upkappa}+\frac{u_0(\eta)}{2}\right)-\left(\frac{\eta}{2}-\frac{u_0(\eta)\upkappa}{2}\right)\frac{\upkappa}{(t+\upkappa)^2}\right]v_0(\eta)\D{\eta}}{\int_{y_*(x_1,t)}^{y^*(x_2,t)}v_0(\eta)\D{\eta}}=u(x,t).
\end{align*}
This completes the proof.

\vspace{1cm}
Proof of $(ii)$ follows a similar argument. First consider the case $y_*(x,t)=y^*(x,t).$ For $t>0$ and $x_1<x<x_2$, we have
\begin{align}\label{eqq5.22}
    E(x_2,t)-E(x_1,t)=\frac{1}{2}\int_{y_*(x_1,t)}^{y_*(x_2,t)}\left[\left(\frac{\eta}{2\upkappa}+\frac{u_0(\eta)}{2}\right)-\left(\frac{\eta}{2}-\frac{u_0(\eta)\upkappa}{2}\right)\frac{\upkappa}{(t+\upkappa)^2}\right]v_0(\eta)u(X(\eta,t),t)\D{\eta}.
\end{align}

\vspace{.5cm}
On the other hand, for $y_*(x_1,t)\leq \eta \leq y_*(x_2,t)$, we have
\begin{align}\label{eqq5.23}
    &\frac{1+\frac{\upkappa^2}{(t+\upkappa)^2}}{(t+\upkappa)-\frac{\upkappa^2}{(t+\upkappa)}}\left[x-\frac{1}{2}\left(\frac{t+\upkappa}{\upkappa}+\frac{\upkappa}{t+\upkappa}\right)y_*(x_1,t)\right]+\frac{y_*(x_1,t)}{2}\left(\frac{1}{\upkappa}-\frac{\upkappa}{(t+\upkappa)^2}\right)\nonumber\\
    &\leq u(X(\eta,t),t)\leq \frac{1+\frac{\upkappa^2}{(t+\upkappa)^2}}{(t+\upkappa)-\frac{\upkappa^2}{(t+\upkappa)}}\left[x-\frac{1}{2}\left(\frac{t+\upkappa}{\upkappa}+\frac{\upkappa}{t+\upkappa}\right)y_*(x_2,t)\right]+\frac{y_*(x_2,t)}{2}\left(\frac{1}{\upkappa}-\frac{\upkappa}{(t+\upkappa)^2}\right).
\end{align}
From \eqref{eqq5.22}-\eqref{eqq5.23}, we get
\begin{align}\label{eqq5.24}
    &\frac{1}{2}\frac{1+\frac{\upkappa^2}{(t+\upkappa)^2}}{(t+\upkappa)-\frac{\upkappa^2}{(t+\upkappa)}}\left[x-\frac{1}{2}\left(\frac{t+\upkappa}{\upkappa}+\frac{\upkappa}{t+\upkappa}\right)y_*(x_1,t)\right]+\frac{y_*(x_1,t)}{2}\left(\frac{1}{\upkappa}-\frac{\upkappa}{(t+\upkappa)^2}\right)\nonumber\\
    &\leq \frac{E(x_2,t)-E(x_1,t)}{q(x_2,t)-q(x_1,t)}\leq \frac{1}{2}\frac{1+\frac{\upkappa^2}{(t+\upkappa)^2}}{(t+\upkappa)-\frac{\upkappa^2}{(t+\upkappa)}}\left[x-\frac{1}{2}\left(\frac{t+\upkappa}{\upkappa}+\frac{\upkappa}{t+\upkappa}\right)y_*(x_2,t)\right]+\frac{y_*(x_2,t)}{2}\left(\frac{1}{\upkappa}-\frac{\upkappa}{(t+\upkappa)^2}\right).
\end{align}
Passing to the limit $x_1\nearrow x$ and $x_2\searrow x$ in \eqref{eqq5.24}, we obtain
\begin{align*}
    \lim_{x_1,x_2 \to x}\frac{E(x_2,t)-E(x_1,t)}{q(x_2,t)-q(x_1,t)}=\frac{1}{2}u(x,t)
\end{align*}
and consequently, by $(i),$ we find $\D{E}=\frac{1}{2}u^2\D{m}.$\\
Now if $y_*(x,t) < y^*(x,t),$ it is straightforward to see
\begin{align*}
    \lim_{x_1,x_2 \to x} \frac{E(x_2,t)-E(x_1,t)}{q(x_2,t)-q(x_1,t)}=\frac{1}{2}u(x,t)
\end{align*}
where we used the fact that $y_*(x_1,t) \to y_*(x,t)$ and $y_*(x_2,t)\to y^*(x,t)$ as $x_1, x_2 \to x$ and for $\eta \in [y_*(x,t), y^*(x,t)],$ we have $u(X(\eta,t),t)=u(x,t).$ This completes the proof of $(ii).$

\vspace{.5cm}
\noindent{\bf Step 2.} Define $F(x,t)=\min\limits_{y \in \RR}F(y,x,t).$ Then the following hold 
\begin{align}
    &(i)\,\,\,\frac{\partial}{\partial x}F(x,t)=-m(x,t),\label{eqq5.25}\\
    &(ii)\,\,\, \frac{\partial}{\partial t}F(x,t)=q(x,t)\label{eqq5.26}.
\end{align}
\noindent{\it Proof of Step 2.} To prove \eqref{eqq5.25}, first we fix $t>0$ and choose any two points $x_1,x_2 \in \RR.$ We claim that
\begin{align}\label{eqq6.27}
    \int_{x_1}^{x_2}m(x,t)\D{x}=F(x_2,t)-F(x_1,t).
\end{align}
Take any $x,x^{\prime}\in [x_1,x_2]$ with $x<x^{\prime}.$ It is enough to prove
\begin{align}\label{eqq5.27}
    (x^{\prime}-x)m(x,t) \leq F(x,t)-F(x^{\prime},t)\leq (x^{\prime}-x)m(x^{\prime},t).
\end{align}
The inequality \eqref{eqq5.27} can be proved by considering
\begin{align}
    F(x,t)-F(x^{\prime},t)&=F(y_*(x,t),x,t)-F(y_*(x^{\prime},t),x^{\prime},t)\nonumber\\
    &=[F(y_*(x,t),x,t)-F(y_*(x,t), x^{\prime},t)]+[F(y_*(x,t),x^{\prime},t)-F(y_*(x^{\prime},t),x^{\prime},t)]\label{eqq5.28}.
\end{align}
It is worth noting that the second term in \eqref{eqq5.28} is positive, leading us to the following inequality,
\begin{align*}
  F(x,t)-F(x^{\prime},t)\geq F(y_*(x,t),x,t)-F(y_*(x,t), x^{\prime},t).
\end{align*}
Similarly, we have
\begin{align*}
    F(x,t)-F(x^{\prime},t)\leq F(y_*(x^{\prime},t),x,t)-F(y_*(x,t^{\prime}),x,t^{\prime}).
\end{align*}
Combining the above two inequalities, we conclude \eqref{eqq5.27}. Note that, since $y_*(x,t)$ is increasing in $x$ and $v_0>0,$ we have $m(x,t)$ is also increasing in $x.$ Therefore, $m(x,t)$ is Riemann integrable. Now taking Riemann sum and using the inequality \eqref{eqq5.27}, we conclude \eqref{eqq6.27} and hence \eqref{eqq5.25}. 

To prove \eqref{eqq5.26}, we first fix $x\in \RR$ and choose any $t,t+h\in (0, \infty).$ Then we have
\begin{align*}
  F(x,t+h)-F(x,t)&=F(y_*(x,t+h),x,t+h)-F(y_*(x,t),x,t)\\
  &=[F(y_*(x,t+h),x,t+h)-F(y_*(x,t+h),x,t)]\\
  &+[F(y_*(x,t+h),x,t)-F(y_*(x,t),x,t)].
\end{align*}
Again, using minimization of $F,$ we obtain
\begin{align*}
    F(x,t+h)-F(x,t)\geq F(y_*(x,t+h),x,t+h)-F(y_*(x,t+h),x,t).
\end{align*}
Similarly, we have
\begin{align*}
    F(x,t+h)-F(x,t) \leq F(y_*(x,t), x,t+h)-F(y_*(x,t),x,t).
\end{align*}
Combining the above two inequalities, we get
\begin{align}\label{eqq6.30}
    &\int_{0}^{y_*(x,t+h)} \left[\left(\frac{\eta}{2\upkappa}+\frac{u_0(\eta)}{2}\right)+\left(\frac{\eta}{2}-\frac{u_0(\eta)\upkappa}{2}\right)\frac{1}{h}\left(\frac{\upkappa}{t+h+\upkappa}-\frac{\upkappa}{t+\upkappa}\right)\right]v_0(\eta)\D{\eta}\nonumber\\
    &\leq \frac{F(x,t+h)-F(x,t)}{h}\leq \int_{0}^{y_*(x,t)} \left[\left(\frac{\eta}{2\upkappa}+\frac{u_0(\eta)}{2}\right)+\left(\frac{\eta}{2}-\frac{u_0(\eta)\upkappa}{2}\right)\frac{1}{h}\left(\frac{\upkappa}{t+h+\upkappa}-\frac{\upkappa}{t+\upkappa}\right)\right]v_0(\eta)\D{\eta}.
\end{align}
As we know $y_*(x,t+h) \to y_*(x,t)$ as $h\to 0$. Then passing to the limit as $h\to 0$ in the above inequality \eqref{eqq6.30}, we obtain \eqref{eqq5.26}. This completes the proof.

\vspace{.5cm}
\noindent{\bf Step 3.} Now we show that $(m,u)$ satisfies the first equation of the weak formulation \eqref{weak formulation}. For a test function $\varphi$ with compact support in $\RR\times ]0,\infty[$ we infer using the step 1 and step 2:
\begin{align}\label{eqq6.33}
0&=\iint [F_x \varphi_t (x, t)- F_t \varphi_x (x, t)] \D{x} \D{t} 
=\iint [m(x,t) \varphi_t (x, t)- q(x,t)\varphi_x (x, t)] \D{x} \D{t}\nonumber\\
&=\iint  [m(x,t) \varphi_t (x, t)\D{x} \D{t}- \iint u(x,t)\varphi(x, t)\D{m} \D{t}\,.
\end{align}
This identity proves that $(m, u)$ satisfies the first equation of the system \eqref{second system}.

\vspace{.5cm}
\noindent{\bf Step 4.} Define $H(x,t)=\min\limits_{y\in \RR}H(y,x,t)$ and $I(x,t)=\min\limits_{y\in \RR}I(y,x,t).$ Then the following relations hold:
\begin{align}
    &(i)\,\,\,\frac{\partial}{\partial x}H(x,t)=-q(x,t),\label{eqq6.34}\\
    &(ii)\,\,\, \frac{\partial}{\partial t}H(x,t)=2E(x,t)+I(x,t)\label{eqq6.35}.
\end{align}
\noindent{\it Proof of Step 4.} To prove \eqref{eqq6.34}, following the previous argument, we fix $t>0$ and choose any two points $x_1,x_2 \in \RR$  and claim that
\begin{align*}
    \int_{x_1}^{x_2}q(x,t)\D{x}=H(x_2,t)-H(x_1,t).
\end{align*}
Take any $x,x^{\prime}\in [x_1,x_2]$ with $x<x^{\prime}.$ Exactly as in the proof of \eqref{eqq5.25}, we show
\begin{align*}
    (x^{\prime}-x)q(x,t) \leq H(x,t)-H(x^{\prime},t)\leq (x^{\prime}-x)q(x^{\prime},t),
\end{align*}
and this completes the proof of \eqref{eqq6.34}.\\
To prove \eqref{eqq6.35}, first we fix $x$ and for $t, t+h \in (0, \infty),$ and following the proof of \eqref{eqq5.26}, we have
\begin{align*}
     H(x,t+h)-H(x,t)\geq H(y_*(x,t+h),x,t+h)-H(y_*(x,t+h),x,t).
\end{align*}
Therefore, we have
\begin{align}
    \frac{H(x,t+h)-H(x,t)}{h}&\geq \int_0^{y_*(x,t+h)}v_0(\eta)\left(\frac{\eta}{2\upkappa}+\frac{u_0(\eta)}{2}\right)\frac{\left(X(\eta, t+h)-X(\eta,t)\right)}{h}\D{\eta}\nonumber\\
    &-\int_0^{y_*(x,t+h)}v_0(\eta)\left(\frac{\eta}{2}-\frac{u_0(\eta)\upkappa}{2}\right)\frac{1}{h}\left[\frac{\upkappa X(\eta, t+h)}{(t+h+\upkappa)^2}-\frac{\upkappa X(\eta, t)}{(t+\upkappa)^2}\right]\D{\eta}\nonumber\\
    &+x \int_0^{y_*(x,t+h)}\left(\frac{\eta}{2}-\frac{u_0(\eta)\upkappa}{2}\right)\frac{1}{h}\left[\frac{\upkappa}{(t+h+\upkappa)^2}-\frac{\upkappa}{(t+\upkappa)^2}\right]v_0(\eta)\D{\eta}\label{eqq6.38}.
\end{align}
Now passing to the limit as $h\to 0$ in \eqref{eqq6.38} and using $\frac{\D}{\D{t}}X(\eta,t)=u(X(\eta,t),t),$ we get
\begin{align*}
    \frac{\partial}{\partial t}H(x,t)&\geq \int_0^{y_*(x,t)}v_0(\eta)\left(\frac{\eta}{2\upkappa}+\frac{u_0(\eta)}{2}\right) u(X(\eta,t),t)\D{\eta}\\
    &-\int_0^{y_*(x,t)}v_0(\eta)\left(\frac{\eta}{2}-\frac{u_0(\eta)\upkappa}{2}\right)\left[\frac{\upkappa u(X(\eta,t),t)}{(t+\upkappa)^2}-\frac{2\upkappa X(\eta,t)}{(t+\upkappa)^3}\right]\D{\eta}\\
    &-x\int_0^{y_*(x,t)}v_0(\eta)\left(\frac{\eta}{2}-\frac{u_0(\eta)\upkappa}{2}\right)\frac{2\upkappa}{(t+\upkappa)^3}\D{\eta}\\
    &=2E(x,t)+I(x,t).
\end{align*}
Similarly, considering the inequality
\begin{align*}
    H(x,t+h)-H(x,t) \leq H(y_*(x,t), x,t+h)-H(y_*(x,t),x,t),
\end{align*}
we obtain
\begin{align*}
    \frac{\partial}{\partial t}H(x,t)\leq 2E(x,t)+I(x,t),
\end{align*}
and this completes the proof.

\vspace{.7cm}
\noindent{\bf Step 5.} Define $J(x,t)=\min\limits_{y\in \RR}J(y,x,t).$ Then we have the following relations:
\begin{align}
    &(i)\,\,\,\frac{\partial}{\partial x}I(x,t)=-J(x,t),\label{eqq6.39}\\
    &(ii)\,\,\, \D{J}=\frac{1}{t+\upkappa}\left(\frac{x}{t+\upkappa}-u\right)\D{m},\label{eqq6.40}
\end{align}
in the sense of Radon-Nikodym derivatives in $x.$\\
\noindent{\it Proof of Step 5.} The proof of \eqref{eqq6.39} is exactly same as the proof of \eqref{eqq5.25} and thus we omit it.

\noindent To prove \eqref{eqq6.40}, we consider the following two cases.\\
\noindent{\bf Case I.} Let $y_*(x,t)=y^*(x,t).$ First of all, after a simplification, we obtain
\begin{align*}
    \frac{1}{t+\upkappa}\left(\frac{x}{t+\upkappa}-u(x,t)\right)=\frac{2 \upkappa}{(t+\upkappa)\left((t+\upkappa)^2-k^2\right)}\left[y_*(x,t)-\frac{\upkappa}{(t+\upkappa)}x\right].
\end{align*}
On the other hand, fixing $t>0$ and taking $x_1<x<x_2,$ by the definition of potential we have
\begin{align*}
    F(y_*(x_2,t), x_2,t)\leq F(y_*(x_1,t),x_2,t).
\end{align*}
This implies
\begin{align*}
\int_{y_*(x_1,t)}^{y_*(x_2,t)}\frac{u_0(\eta)}{2}\frac{(t+\upkappa)^2-\upkappa^2}{(t+\upkappa)}v_0(\eta)\D{\eta}\leq \int_{y_*(x_1,t)}^{y_*(x_2,t)}\left[x_2-\frac{\eta}{2}\frac{(t+\upkappa)^2+\upkappa^2}{\upkappa (t+\upkappa)}\right]v_0(\eta)\D{\eta}.
\end{align*}
Multiplying $\frac{2\upkappa^2}{(t+\upkappa)^2((t+\upkappa)^2-\upkappa^2)}\geq 0$ in both sides of the above inequality and simplifying, we get
\begin{align*}
   & \int_{y_*(x_1,t)}^{y_*(x_2,t)}\frac{u_0(\eta)}{2}\frac{2\upkappa^2}{(t+\upkappa)^3}v_0(\eta)\D{\eta}\\
    &\leq \frac{2\upkappa}{(t+\upkappa)}\int_{y_*(x_1,t)}^{y_*(x_2,t)} \left[\frac{\upkappa x_2}{(t+\upkappa)((t+\upkappa)^2-\upkappa^2)}-\frac{\eta}{2} \frac{(t+\upkappa)^2+\upkappa^2)}{(t+\upkappa)^2((t+\upkappa)^2-\upkappa^2)}\right]v_0(\eta)\D{\eta}\\
    &=\frac{2\upkappa}{(t+\upkappa)}\int_{y_*(x_1,t)}^{y_*(x_2,t)} \left[\frac{\upkappa x_2}{(t+\upkappa)((t+\upkappa)^2-\upkappa^2)}+\frac{\eta}{2}\left(\frac{1}{(t+\upkappa)^2}-\frac{2}{(t+\upkappa)^2-\upkappa^2}\right)\right]v_0(\eta)\D{\eta}.
\end{align*}
This implies
\begin{align*}
    &\int_{y_*(x_1,t)}^{y_*(x_2,t)}\left[\frac{2\upkappa}{(t+\upkappa)((t+\upkappa)^2-\upkappa^2)} \cdot \eta-\frac{2\upkappa^2 x_2}{(t+\upkappa)((t+\upkappa)^2-\upkappa^2)}\right]v_0(\eta)\D{\eta}\\
    &\leq\int_{y_*(x_1,t)}^{y_*(x_2,t)} \left(\frac{\eta}{2}-\frac{u_0(\eta) \upkappa}{2}\right) \frac{2\upkappa v_0(\eta)}{(t+\upkappa)^3}\D{\eta}=J(x_2,t)-J(x_1,t),
\end{align*}
and finally, we obtain
\begin{align*}
    \frac{2\upkappa}{(t+\upkappa)((t+\upkappa)^2-\upkappa^2)}\left[y_*(x_2,t)-\frac{\upkappa}{(t+\upkappa)}x_2\right] \leq \frac{J(x_2,t)-J(x_1,t)}{m(x_2,t)-m(x_1,t)}.
\end{align*}
Hence
\begin{align*}
    \lim_{x_1,x_2 \to x}\frac{J(x_2,t)-J(x_1,t)}{m(x_2,t)-m(x_1,t)}&\geq \frac{2\upkappa}{(t+\upkappa)((t+\upkappa)^2-\upkappa^2)}\left[y_*(x,t)-\frac{\upkappa}{(t+\upkappa)}x\right]\\
    &=\frac{1}{t+\upkappa}\left(\frac{x}{t+\upkappa}-u(x,t)\right).
\end{align*}
Similarly, considering the inequality
\begin{align*}
    F(y_*(x_1,t), x_1,t) \leq F(y_*(x_2, t), x_1, t),
\end{align*}
we get 
\begin{align*}
    \lim_{x_1,x_2 \to x}\frac{J(x_2,t)-J(x_1,t)}{m(x_2,t)-m(x_1,t)} \leq \frac{1}{t+\upkappa}\left(\frac{x}{t+\upkappa}-u(x,t)\right),
\end{align*}
and this completes the proof.

\vspace{.5cm}
\noindent{\bf Case II.} Let $y_*(x,t)<y^*(x,t).$ In this case we have
\begin{align}
    \frac{J(x_1,t)-J(x_2,t)}{m(x_1,t)-m(x_2,t)}=&\frac{2}{t+\upkappa}\cdot\frac{\int_{y_*(x_1,t)}^{y_*(x_2,t)}\left[\left(\frac{\eta}{2}-\frac{u_0(\eta)\upkappa}{2}\right)\frac{\upkappa}{(t+\upkappa)^2}-\left(\frac{\eta}{2\upkappa}+\frac{u_0(\eta)}{2}\right)\right]v_0(\eta)\D{\eta}}{\int_{y_*(x_1,t)}^{y_*(x_2,t)}v_0(\eta)\D{\eta}}\nonumber\\
    &+\frac{2}{t+\upkappa}\cdot \frac{\int_{y_*(x_1,t)}^{y_*(x_2,t)}\left(\frac{\eta}{2\upkappa}+\frac{u_0(\eta)}{2}\right)v_0(\eta)\D{\eta}}{\int_{y_*(x_1,t)}^{y_*(x_2,t)}v_0(\eta)\D{\eta}}\nonumber\\
    =&\frac{1}{t+\upkappa}\cdot \frac{\int_{y_*(x_1,t)}^{y_*(x_2,t)}\left[\left(\frac{\eta}{2}-\frac{u_0(\eta)\upkappa}{2}\right)\frac{\upkappa}{(t+\upkappa)^2}-\left(\frac{\eta}{2\upkappa}+\frac{u_0(\eta)}{2}\right)\right]v_0(\eta)\D{\eta}}{\int_{y_*(x_1,t)}^{y_*(x_2,t)}v_0(\eta)\D{\eta}}\nonumber\\
    &+\frac{1}{t+\upkappa}\cdot \frac{\int_{y_*(x_1,t)}^{y_*(x_2,t)}\left[\left(\frac{\eta}{2}-\frac{u_0(\eta)\upkappa}{2}\right)\frac{\upkappa}{(t+\upkappa)^2}+\left(\frac{\eta}{2\upkappa}+\frac{u_0(\eta)}{2}\right)\right]v_0(\eta)\D{\eta}}{\int_{y_*(x_1,t)}^{y_*(x_2,t)}v_0(\eta)\D{\eta}}\label{eqq6.41}.
\end{align}
Using the inequality $F(y_*(x_1,t),x_1,t) \leq F(y_*(x_2,t),x_1,t),$ we get
\begin{align}\label{eqq6.42}
    \frac{x_1}{t+\upkappa} \leq \frac{\int_{y_*(x_1,t)}^{y_*(x_2,t)}\left[\left(\frac{\eta}{2}-\frac{u_0(\eta)\upkappa}{2}\right)\frac{\upkappa}{(t+\upkappa)^2}+\left(\frac{\eta}{2\upkappa}+\frac{u_0(\eta)}{2}\right)\right]v_0(\eta)\D{\eta}}{\int_{y_*(x_1,t)}^{y_*(x_2,t)}v_0(\eta)\D{\eta}}.
\end{align}
Therefore plugging \eqref{eqq6.42} into \eqref{eqq6.41}, we have
\begin{align*}
   \lim_{x_1,x_2 \to x}\frac{J(x_1,t)-J(x_2,t)}{m(x_1,t)-m(x_2,t)}&\geq \frac{1}{t+\upkappa}\cdot \frac{\int_{y_*(x,t)}^{y^*(x,t)}\left[\left(\frac{\eta}{2}-\frac{u_0(\eta)\upkappa}{2}\right)\frac{\upkappa}{(t+\upkappa)^2}-\left(\frac{\eta}{2\upkappa}+\frac{u_0(\eta)}{2}\right)\right]v_0(\eta)\D{\eta}}{\int_{y_*(x,t)}^{y^*(x,t)}v_0(\eta)\D{\eta}} 
   +\frac{x}{(t+\upkappa)^2}\\
   &=\frac{1}{t+\upkappa}\left(\frac{x}{t+\upkappa}-u(x,t)\right).
\end{align*}
Similarly considering the inequality 
\begin{align*}
    F(y_*(x_2,t),x_2,t)\leq F(y_*(x_1,t),x_2,t),
\end{align*}
we obtain
\begin{align*}
    \lim_{x_1,x_2 \to x}\frac{J(x_1,t)-J(x_2,t)}{m(x_1,t)-m(x_2,t)}\leq \frac{1}{t+\upkappa}\left(\frac{x}{t+\upkappa}-u(x,t)\right). 
\end{align*}
This proves $(ii).$

\vspace{.5cm}
\noindent{\bf Step 6.} We show that the pair $(m,u)$ satisfies the second equation of \eqref{wf2}. For that, we have
\begin{align}
    0&=\iint [H_x \varphi_{t x} (x, t)- H_t \varphi_{xx} (x, t)] \D{x} \D{t}
    =-\iint [q \varphi_{tx}+(2E+I)\varphi_{xx}]\D{x}\D{t}\nonumber\\
    &=\iint [q_x \varphi_t+(2E+I)_x \varphi_x]\D{x}\D{t}=\iint [u \varphi_t +u^2 \varphi_x]\D{m}\D{t}-\iint J \varphi_x \D{x}\D{t}\nonumber \\
    &=\iint [u \varphi_t +u^2 \varphi_x]\D{m}\D{t}+\iint J_x \varphi \D{x}\D{t}=\iint u \varphi_t +u^2 \varphi_x +\frac{1}{t+\upkappa} \left(\frac{x}{t+\upkappa}-u\right) \varphi \D{m}\D{t}\label{eqq6.43}.
\end{align}
The identity \eqref{eqq6.43} combined with \eqref{eqq6.33} completes the proof of the weak formulation.
\subsubsection{Entropy criterion} Now we show that $(m,u)$ satisfies the Oleinik type entropy condition. For any discontinuity point $(x,t),$ after simplifying and considering the construction of solution $u(x,t),$ we obtain the following expressions.
\begin{align*}
    u(x-0, t)=\frac{(t+\upkappa)^2+\upkappa^2}{(t+\upkappa)((t+\upkappa)^2-\upkappa^2)}x-\frac{2\upkappa}{(t+\upkappa)^2-\upkappa^2}y_*(x,t),\\\\
    u(x+0, t)=\frac{(t+\upkappa)^2+\upkappa^2}{(t+\upkappa)((t+\upkappa)^2-\upkappa^2)}x-\frac{2\upkappa}{(t+\upkappa)^2-\upkappa^2}y^*(x,t).
\end{align*}
Since $(x,t)$ is a point of discontinuity, $y_*(x,t)<y^*(x,t)$ and using $F(y_*(x,t), x,t)=F(y^*(x,t), x,t),$ we have
\begin{align*}
    \int_{y_*(x,t)}^{y^*(x,t)}\left[\frac{\eta}{2} \cdot\frac{(t+\upkappa)^2+\upkappa^2}{\upkappa(t+\upkappa)}+\frac{u_0(\eta)}{2}\cdot \frac{(t+\upkappa)^2-\upkappa^2}{(t+\upkappa)}-x\right] v_0(\eta)\D{\eta}=0.
\end{align*}
Multiplying $\frac{(t+\upkappa)^2+\upkappa^2}{(t+\upkappa)((t+\upkappa)^2-\upkappa^2)}$ in the above equation and rearranging the terms, we obtain
\begin{align}
    &\int_{y_*(x,t)}^{y^*(x,t)} \left[\frac{\eta}{2} \cdot \frac{(t+\upkappa)^2-\upkappa^2}{\upkappa (t+\upkappa)^2}+\frac{u_0(\eta)}{2}\cdot \frac{(t+\upkappa)^2+\upkappa^2}{(t+\upkappa)^2}\right]v_0(\eta)\D{\eta}\nonumber\\
    &=\int_{y_*(x,t)}^{y^*(x,t)} \left[\frac{(t+\upkappa)^2+\upkappa^2}{(t+\upkappa)((t+\upkappa)^2-\upkappa^2)}x-\eta \cdot \frac{2\upkappa}{(t+\upkappa)^2-\upkappa^2}\right]v_0(\eta)\D{\eta}\label{eqq6.44}\\
    &\leq \left[\frac{(t+\upkappa)^2+\upkappa^2}{(t+\upkappa)((t+\upkappa)^2-\upkappa^2)}x-y_*(x,t) \cdot \frac{2\upkappa}{(t+\upkappa)^2-\upkappa^2}\right] \int_{y_*(x,t)}^{y^*(x,t)}v_0(\eta)\D{\eta}\nonumber\\
    &=u(x-0,t)\int_{y_*(x,t)}^{y^*(x,t)}v_0(\eta)\D{\eta}\label{eqq6.45}.
\end{align}
Again, by $y^*(x,t)$ in \eqref{eqq6.44}, we get the other way inequality
\begin{align}\label{eqq6.46}
   \int_{y_*(x,t)}^{y^*(x,t)} \left[\frac{\eta}{2} \cdot \frac{(t+\upkappa)^2-\upkappa^2}{\upkappa (t+\upkappa)^2}+\frac{u_0(\eta)}{2}\cdot \frac{(t+\upkappa)^2+\upkappa^2}{(t+\upkappa)^2}\right]v_0(\eta)\D{\eta} \geq u(x+0, t)\int_{y_*(x,t)}^{y^*(x,t)}v_0(\eta)\D{\eta}.
\end{align}
Therefore, combining \eqref{eqq6.45} and \eqref{eqq6.46}, we have $u(x+0,t)\leq u(x,t) \leq u(x-0,t).$ 

Furthermore, for any $x_1\neq x_2,$ we have
\begin{align*}
    \frac{u(x_2,t)-u(x_1,t)}{x_2-x_1}\leq \frac{u(x_2-0,t)-u(x_1+0,t)}{x_2-x_1}\leq \frac{(t+\upkappa)^2+\upkappa^2}{(t+\upkappa)((t+\upkappa)^2-\upkappa^2)}.
\end{align*}
\subsubsection{Verification of initial condition} In this section, we show that the pair $(m, u)$ satisfies the initial condition in the sense that for almost every $x,$ we have $\lim\limits_{t\to 0}u(x,t)=u_0(x)$ and $\lim\limits_{t\to 0}m(x,t)=\int_0^x \rho_0(\eta)\D{\eta}.$ Since $y_*(x,t)$ and $y^*(x,t)$ converges to $x$ as $t\to 0+,$ from the definition of $m(x,t)$ given by \eqref{exf1}, we get $\lim\limits_{t\to 0+}m(x,t)=\int_0^x v_0(\eta)\D{\eta}.$

We show the first assertion for any Lebesgue point $x_0$ of $u_0(x)$ and $v_0(x).$ To be more precise, we show that 
\begin{align}\label{initial condition}
    \lim_{t\to 0+}u(x_0,t)=u_0(x_0).
\end{align} When $y_*(x,t)<y^*(x,t),$ \eqref{initial condition} follows directly from the construction of $u(x,t)$. We only consider the case $y_*(x_0,t)=y^*(x_0,t).$ First we set
$T=\frac{(t+\upkappa)^2-\upkappa^2}{2\upkappa}.$ For any $\upepsilon >0,$ considering the inequality
\begin{align*}
    F(y_*(x_0,t), x_0, t)\leq F(y^*(x_0,t)+\upepsilon T, x_0, t),
\end{align*}
and calculating as above, we get 
\begin{align*}
    &\int_{y_*(x_0,t)}^{y^*(x_0,t)+\upepsilon T}\left[\frac{\eta}{2} \cdot \frac{(t+\upkappa)^2-\upkappa^2}{\upkappa (t+\upkappa)^2}+\frac{u_0(\eta)}{2}\cdot \frac{(t+\upkappa)^2+\upkappa^2}{(t+\upkappa)^2}\right]v_0(\eta)\D{\eta}\nonumber\\
    &\geq \int_{y_*(x_0,t)}^{y^*(x_0,t)+\upepsilon T}\left[\frac{(t+\upkappa)^2+\upkappa^2}{(t+\upkappa)((t+\upkappa)^2-\upkappa^2)}x_0-\eta \cdot \frac{2\upkappa}{(t+\upkappa)^2-\upkappa^2}\right]v_0(\eta)\D{\eta},
\end{align*}
similar to the inequality \eqref{eqq6.44}.
This implies
\begin{align}
    &\frac{\int_{y_*(x_0,t)}^{y^*(x_0,t)+\upepsilon T}\left[\left(\frac{\eta}{2\upkappa}+\frac{u_0(\eta)}{2}\right)-\left(\frac{\eta}{2}-\frac{u_0(\eta)\upkappa}{2}\right)\frac{\upkappa}{(t+\upkappa)^2}\right]v_0(\eta)\D{\eta}}{\int_{y_*(x_0,t)}^{y^*(x_0,t)+\upepsilon T} v_0(\eta)\D{\eta}}\nonumber\\
    &\geq \left[\frac{(t+\upkappa)^2+\upkappa^2}{(t+\upkappa)((t+\upkappa)^2-\upkappa^2)}x_0-y_*(x_0,t) \cdot \frac{2\upkappa}{(t+\upkappa)^2-\upkappa^2}-\upepsilon\right]=u(x_0,t)-\upepsilon.\label{eq6.42}
\end{align}
Since $x_0$ is a Lebesgue point of $u_0(x)$ and $v_0(x),$ passing to the limit as $t\to 0+$ in \eqref{eq6.42}, we conclude
\begin{align}\label{eq6.43}
    u_0(x_0) \geq \limsup_{t\to 0+} u(x_0,t)-\upepsilon.
\end{align}
Similarly, considering the inequality
   $F(y_*(x_0,t), x_0, t)\leq F(y_*(x_0,t)-\upepsilon T, x_0, t),$
we obtain 
\begin{align}\label{eq6.44}
    u_0(x_0) \leq \liminf_{t\to 0+}u(x_0,t)+\upepsilon.
\end{align}
Since $\upepsilon$ is arbitrary, combining \eqref{eq6.43}-\eqref{eq6.44}, we conclude \eqref{initial condition}.
\begin{remark}
Note that one could insert the following shadow wave solution to the system \eqref{second system}
\begin{align*}
U^{\varepsilon}=(v^{\varepsilon}, u^{\varepsilon})(x,t)=
\begin{cases}
\left(V_l (t), U_l(x,t)\right),\,\,\,\,\,\,\,\,\,\,\,\, &x< c(t)-\frac{\varepsilon}{2}t-x_{\varepsilon},\\
\left(v_{\varepsilon}(t), u_{\varepsilon}(t)\right),\,\,\,\,\,\, &c(t)-\frac{\varepsilon}{2}t-x_{\varepsilon}<x<c(t)+\frac{\varepsilon}{2}t+x_{\varepsilon},\\
\left(V_r(t), U_r(x,t)\right),\,\,\,\,\,\,\,\,\,\,\, &x>c(t)+\frac{\varepsilon}{2}t+x_{\varepsilon},
\end{cases}
\end{align*}
where $V_{l,r}(t)$ and $U_{l,r}(x,t)$ are given by
\begin{align*}
    U_{l,r}(x,t):= \frac{2u_{l,r}\left(\frac{\upkappa}{t+\upkappa}\right)}{\left(\frac{t+\upkappa}{\upkappa}+\frac{\upkappa}{t+\upkappa}\right)}+\frac{\left(\frac{1}{\upkappa}-\frac{\upkappa}{(t+\upkappa)^2}\right)}{\left(\frac{t+\upkappa}{\upkappa}+\frac{\upkappa}{t+\upkappa}\right)}x,~~
    V_{l,r}(t):=\frac{2 v_{l,r}}{\left(\frac{t+\upkappa}{\upkappa}+\frac{\upkappa}{t+\upkappa}\right)},
\end{align*}
and  $x_{\varepsilon}, v_{\varepsilon}(t)$ are $\mathcal{O}(\varepsilon)$ and $\mathcal{O}(1/\varepsilon)$, respectively. Following the similar calculations as in \cref{sec2}, we obtain the system of ODEs given as follows:
\begin{align}\label{second system of ODE}
&\frac{\D{\xi(t)}}{\D{t}}=\dot{c}(t)[V(t)]-[V(t)U(x,t)], \,\,\, \xi(0)=\bar{m},\nonumber \\
&\frac{\D{\left(\xi(t)\chi(t)\right)}}{\D{t}}+\frac{1}{t+\upkappa}\left(\chi(t)-\frac{x}{t+\upkappa}\right)\xi(t)=\dot{c}(t)[V(t)U(x,t)]-[V(t)U^2(x,t)],\,\,\, \xi(0)\chi(0)=\bar{m}\bar{u},\nonumber\\
&\dot{c}(t)=\chi(t).
\end{align}
where $\lim\limits_{\varepsilon \to 0}2\left(\frac{\varepsilon}{2}t+x_{\varepsilon}\right)v_{\varepsilon}(t)=\xi(t)$,  $\lim\limits_{\varepsilon \to 0}u_{\varepsilon}(t)=\chi(t)$ and $[\cdot]:=\cdot_r-\cdot_l$ denotes the jump across the discontinuity curve $x=c(t).$
Since the system of ODEs \eqref{second system of ODE} involves $x$ and $t$ both, it is not always straightforward to find an explicit expression for $\xi(t)$ and hence for $\chi(t)$ as we found in \cref{sec2}. Also, one of the advantages to the variational approach is that we can allow a larger set of initial datum, whereas to proceed with shadow wave tracking we need some regularity on the initial data. However, the variational approach requires the assumption $v_0(x)>0$ while  $v_0(x)=0$ case is allowed in shadow wave tracking.
\end{remark}
\subsection*{Acknowledgement} The authors gratefully acknowledge the comments and suggestions made by the anonymous referee to improve the manuscript.

\vspace{.5cm}
\noindent\textbf{Conflict of interest.} On behalf of all authors, the corresponding author states that there is no conflict of interest.

\vspace{.5cm}
\noindent\textbf{Data availability.} This manuscript has no associated data.


\begin{thebibliography}{10}

\bibitem{AW2000}A. Aw and M. Rascle: Resurrection of ``second order'' models of traffic flow. SIAM J. Appl. Math. 60 (2000), no.3, 916–938.

\bibitem{bressan} A. Bressan: Hyperbolic systems of conservation laws. The one-dimensional Cauchy problem. Vol. 20. Oxford University Press, Oxford, 2000.

\bibitem{bressan1} A. Bressan: Global solutions of systems of conservation laws by wave-front tracking. { J. Math. Anal. Appl.} 170 (1992), no. 2, 414–432.

\bibitem{B1995} A. Bressan: The unique limit of the Glimm scheme. Arch. Rational Mech. Anal. 130 (1995), no. 3, 205–230. 

\bibitem{BN2014} A. Bressan and T. Nguyen: Non-existence and non-uniqueness for multidimensional sticky particle systems. Kinetic and Related Models, 7 (2014), no. 2, 205-218.

\bibitem{kiselev2018} A. Kiselev and C. Tan: Global regularity for 1D Eulerian dynamics with singular interaction forces. SIAM J. Math. Anal. 50 (2018), no. 6, 6208–6229.

\bibitem{KK1980} B. L. Keyfitz and H.C. Kranzer: A system of nonstrictly hyperbolic conservation laws arising in elasticity theory. Arch. Rational Mech. Anal. 72 (1980), no.3, 219–241.




\bibitem{piccolizamp}  B. Piccoli, T. Andrea and, M. Zanella: Model-based assessment of the impact of driver-assist vehicles using kinetic theory. {Z. Angew. Math. Phys.} 71 (2020), no. 5, Paper No. 152, 25 pp.

\bibitem{shenjde} C. Shen and M. Sun: Exact Riemann solutions for the drift-flux equations of two-phase flow under gravity. {J. Differential Equations} 314 (2022), 1–55.


\bibitem{shen} C. Shen: The Riemann problem for the pressureless Euler system with the Coulomb-like friction term. { IMA J. Appl. Math.} 81 (2016), no. 1, 76–99.
	

\bibitem{dafermos} C. M. Dafermos: Polygonal approximations of solutions of the initial value problem for a conservation law. { J. Math. Anal. Appl.} 38 (1972), 33–41.
	
\bibitem{dafermos1} C. M. Dafermos: Hyperbolic conservation laws in continuum physics. Fourth edition. Grundlehren der mathematischen Wissenschaften, 325. Springer-Verlag, Berlin, 2016.

\bibitem{CDL08} C. De Lellis: Rectifiable sets, densities, and tangent measures. European Mathematical Society (EMS), Zürich, 2008.

\bibitem{darko}  D. Mitrović and M. Nedeljkov: Delta shock waves as a limit of shock waves. { J. Hyperbolic Differ. Equ. }  4 (2007), no. 4, 629–653.



 \bibitem{Coddington}  E. A. Coddington and N. Levinson: Theory of ordinary differential equations. McGraw-Hill Book Co., Inc., New York-Toronto-London, 1955


\bibitem{rykov} E. Weinan, Y. G. Rykov, and Y. G. Sinai:  Generalized variational principles, global weak solutions and behavior with random initial data for systems of conservation laws arising in adhesion particle dynamics.  { Comm. Math. Phys.} 177 (1996), no. 2, 349–380. 


\bibitem{bouchut1} F. Bouchut: On zero pressure gas dynamics.  { Advances in kinetic theory and computing}, 171–190, Ser. Adv. Math. Appl. Sci., 22, World Sci. Publ., River Edge, NJ, 1994. 
	
\bibitem{bouchut-james} F. Bouchut and F. James: Duality solutions for pressureless gases, monotone scalar conservation laws, and uniqueness. { Comm. Partial Differential Equations} 24 (1999), no. 11-12, 2173–2189. 

\bibitem{BJ98} F. Bouchut and F. James: One-dimensional transport equations with discontinuous coefficients. {Nonlinear Anal.} 32 (1998), no. 7, 891–933.
	
\bibitem{bouchut-jin} F. Bouchut, S. Jin, and X. Li: Numerical approximations of pressureless and isothermal gas dynamics.  { SIAM J. Numer. Anal.} 41 (2003), no. 1, 135–158.
	


\bibitem{haung-wang} F. Huang and Z. Wang: Well posedness for pressureless flow. { Comm. Math. Phys.} 222 (2001), no. 1, 117–146. 
	
	
\bibitem{huang} F. Huang: Weak solution to pressureless type system. { Comm. Partial Differential Equations} 30 (2005), no. 1-3, 283–304.

\bibitem{huangtwo} F. Huang, D. Wang, and D. Yuan: Nonlinear stability and existence of vortex sheets for inviscid liquid-gas two-phase flow. { Discrete Contin. Dyn. Syst.} 39 (2019), 3535–3575.

\bibitem{cavalletti} F. Cavalletti, M. Sedjro, and M. Westdickenberg: A simple proof of global existence for the 1D pressureless gas dynamics equations. {SIAM J. Math. Anal.} 47 (2015), no. 1, 66–79. 


\bibitem{holden} H. Holden and N. H. Risebro: Front tracking for hyperbolic conservation laws. Second edition. Applied Mathematical Sciences, 152. Springer, Heidelberg, 2015



	



\bibitem{yangjde}  H. Yang: Riemann problems for a class of coupled hyperbolic systems of conservation laws. { J. Differential Equations} 159 (1999), no. 2, 447–484.


\bibitem{chengaml}  H. Cheng and H. Yang: The Riemann problem for the inhomogeneous pressureless Euler equations. { Appl. Math. Lett.} 135 (2023), no. 6, Paper No. 108442, 8 pp.

\bibitem{hanchun2007} H. Yang and W. Sun: The Riemann problem with delta initial data for a class of coupled hyperbolic systems of conservation laws. {Nonlinear Anal.} 67 (2007), no. 11, 3041–3049.

\bibitem{boudin} L. Boudin: A solution with bounded expansion rate to the model of viscous pressureless gases. { SIAM J. Math. Anal.} 32 (2000), no. 1, 172–193.

	
\bibitem{mathiaud} L. Boudin and J. Mathiaud: A numerical scheme for the one-dimensional pressureless gases system.  { Numer. Methods Partial Differential Equations } 28 (2012), no. 6, 1729–1746.

\bibitem{neumann2022}L. Neumann, M. Oberguggenberger, M. R. Sahoo, and A. Sen: Initial-boundary value problem for 1D pressureless gas dynamics, {J. Differential Equations} 316 (2022), 687–725.



\bibitem{NS2009}L. Natile and G. Sava\'{r}e: A Wasserstein approach to the one-dimensional sticky particle
system. {SIAM J. Math. Anal.} 41 (2009), no. 4, 1340–1365.



\bibitem{marko-manas} M. Nedeljkov, L. Neumann, M. Oberguggenberger, Michael, and M. R. Sahoo: Radially symmetric shadow wave solutions to the system of pressureless gas dynamics in arbitrary dimensions. { Nonlinear Anal.} 163 (2017), 104–126.


\bibitem{marko1} M. Nedeljkov: Shadow waves: entropies and interactions for delta and singular shocks.  { Arch. Ration. Mech. Anal.} 197 (2010), no. 2, 489–537.
	
\bibitem{marko2} M. Nedeljkov and S. Ružičić: On the uniqueness of solution to generalized Chaplygin gas. { Discrete Contin. Dyn. Syst.} 37 (2017), no. 8, 4439–4460.
	
\bibitem{marko3} M. Nedeljkov: Higher order shadow waves and delta shock blow up in the Chaplygin gas. { J. Differential Equations} 256 (2014), no. 11, 3859–3887.
		

\bibitem{risebro} N. H. Risebro: A front-tracking alternative to the random choice method. { Proc. Amer. Math. Soc.} 117 (1993), no. 4, 1125–1139.


\bibitem{zhangq} Q. Zhang, F. He, and Y. Ba: Delta-shock waves and Riemann solutions to the generalized pressureless Euler equations with a composite source term. { Appl. Anal. } 102 (2023), no. 2, 576–589.


\bibitem{richardacta} R. De la cruz and J. Juajibioy: Delta shock solution for a generalized zero-pressure gas
dynamics system with linear damping. {Acta Appl. Math. } 177 (2022), Paper No. 1, 1-25.


\bibitem{hajde} S. Y. Ha, F. Huang, and Y. Wang: A global unique solvability of entropic weak solution to the one-dimensional pressureless Euler system with a flocking dissipation. {J. Differential Equations} 257 (2014), no. 5, 1333-1371.
	

\bibitem{RN21} S. Ružičić and M. Nedeljkov: Shadow wave tracking procedure and initial data problem for pressureless gas model. Acta Appl. Math. 171 (2021), Paper No. 10, pp. 36.

\bibitem{sanajmaa} S. Keita and Y. Bourgault: Eulerian droplet model: delta-shock waves and solution of the Riemann problem. {  J. Math. Anal. Appl.} 472 (2019), no. 1, 1001-1027.

\bibitem{Keitathesis} S. Keita: Eulerian Droplet Models: Mathematical Analysis, Improvement and Applications, Ph.D thesis, http://dx.doi.org/10.20381/ruor-22165.

\bibitem{evje1} S. Evje and T. Flatten: On the wave structure of two-phase flow models. { SIAM J. Appl. Math.} 67 (2007), 487–511.

\bibitem{evje2} S. Evje and K. H. Karlsen: Global existence of weak solutions for a viscous two-phase model. { J. Differ. Equations} 245 (2008), no. 9, 2660–2703.


\bibitem{leslie2023} T. M. Leslie and C. Tan: Sticky particle Cucker–Smale dynamics and the entropic selection principle for the 1D Euler-alignment system. {Comm. Partial Differential Equations} 48 (2023), no. 5, 753–791.



 \bibitem{nguyen} T. Nguyen and A. Tudorascu: One-dimensional pressureless gas systems with/without viscosity. { Comm. Partial Differential Equations} 40 (2015), no. 9, 1619–1665.

 \bibitem{NT08}T. Nguyen and A. Tudorascu: Pressureless Euler/Euler-Poisson systems via adhesion dynamics and scalar conservation laws.  SIAM J. Math. Anal. 40 (2008), no. 2, 754–775.

\bibitem{DZ96} X. Ding and Z. Wang: Existence and uniqueness of discontinuous solutions defined by Lebesgue-Stieltjes integral. Sci. China Ser. A 39 (1996), no. 8, 807–819.

\bibitem{ding-huang} Y. Ding and F. Huang: On a nonhomogeneous system of pressureless flow.  { Quart. Appl. Math.} 62 (2004), no. 3, 509–528.

\bibitem{Lu2013} Y. G. Lu: Existence of global entropy solutions to general system of Keyfitz-Kranzer type. J. Funct. Anal. 264 (2013), no. 10, 2457–2468.

\bibitem{LuSIMA} Y. G. Lu: Existence of global weak entropy solutions to some nonstrictly hyperbolic systems. SIAM J. Math. Anal. 45 (2013), no. 6, 3592–3610.

\bibitem{BG} Y. Brenier and E. Grenier: Sticky particles and scalar conservation laws.  { SIAM J. Numer. Anal.} 35 (1998), no. 6, 2317–2328.

\bibitem{zhangcpaa}  Y. Zhang, and Y. Zhang: Riemann problems for a class of coupled hyperbolic systems of conservation laws with a source term. { Commun. Pure Appl. Anal} 18 (2019), no. 3, 1523–1545.

\bibitem{YB99}Y. Bourgault, W. G. Habashi, J. Dompierre, and G.S. Baruzzi: A finite element method study of Eulerian droplets impingement models, Internat. J. Numer. Methods Fluids, 29 (1999), no. 4, 429–449.


\bibitem{WSJMP23} Y. Wang and M. Sun: Formation of delta shock and vacuum state for the pressureless hydrodynamic model under the small disturbance of traffic pressure.  { J. Math. Phys.}, 64 (2023), no. 1, Paper No. 011508, 23 pp.
	
\bibitem{Wang-haung} Z. Wang, F. Huang, and X. Ding: On the Cauchy problem of transportation equations. { Acta Math. Appl. Sinica (English Ser.)} 13 (1997), no. 2, 113–122.
	
\bibitem{wang-ding} Z. Wang and X. Ding: Uniqueness of generalized solution for the Cauchy problem of transportation equations. { Acta Math. Sci. (English Ed.)} 17 (1997), no. 3, 341–352.

	
\end{thebibliography}
\end{document}